\documentclass[11pt]{amsart}
\usepackage{amsmath,amssymb,amsthm}
\IfFileExists{stmaryrd.sty}{
  \usepackage{stmaryrd}
  \makeatletter
  \InputIfFileExists{Ustmry.fd}{}{}
  \makeatother
  \DeclareFontShape{U}{stmry}{b}{n}{<->ssub*stmry/m/n}{}
  \DeclareFontShape{U}{stmry}{bx}{n}{<->ssub*stmry/m/n}{}
}{}

\usepackage{enumitem}
\usepackage{geometry}
\usepackage{comment}
\usepackage{multirow}
\usepackage{longtable}
\numberwithin{equation}{section}
\usepackage[backend=biber, giveninits=true]{biblatex}
\DeclareLanguageMapping{english}{english-apa}
\usepackage{aliascnt}
\usepackage{hyperref}
\usepackage[capitalise]{cleveref}
\newtheorem{theorem}{Theorem}[section]

\newtheorem{thmabc}{Theorem}

\newaliascnt{prop}{theorem}
\newtheorem{prop}[prop]{Proposition}
\aliascntresetthe{prop}
\crefname{prop}{Proposition}{Propositions}
\newaliascnt{cor}{theorem}
\newtheorem{cor}[cor]{Corollary}
\aliascntresetthe{cor}
\crefname{cor}{Corollary}{Corollaries}
\newaliascnt{lemma}{theorem}
\newtheorem{lemma}[lemma]{Lemma}
\aliascntresetthe{lemma}
\crefname{lemma}{Lemma}{Lemmas}
\theoremstyle{definition}
\newaliascnt{definition}{theorem}
\newtheorem{definition}[definition]{Definition}
\aliascntresetthe{definition}
\crefname{definition}{Definition}{Definitions}
\newaliascnt{ex}{theorem}
\newtheorem{ex}[ex]{Example}
\aliascntresetthe{ex}
\crefname{ex}{Example}{Examples}
\Crefname{ex}{Example}{Examples}
\theoremstyle{remark}
\newaliascnt{remark}{theorem}
\newtheorem{remark}[remark]{Remark}
\aliascntresetthe{remark}
\crefname{remark}{Remark}{Remarks}

\newtheorem*{acknowledgements}{Acknowledgements}

\DeclareMathOperator{\diag}{diag}

\newcommand{\igu}[1]{\textup{I}_{#1}^+}
\newcommand{\igm}[1]{\textup{I}_{#1}^-}
\newcommand{\ig}[1]{\textup{I}_{#1}}
\newcommand{\N}{\mathbb{N}}
\newcommand{\C}{\mathbb{C}}
\newcommand{\Z}{\mathbb{Z}}
\newcommand{\Q}{\mathbb{Q}}

\newcommand{\Mat}{\mathrm{Mat}}
\newcommand{\Sat}{\mathcal{S}}
\newcommand{\GL}{\mathrm{GL}}

\newcommand{\Sp}{\mathrm{Sp}}
\newcommand{\Des}{\mathrm{Des}}
\newcommand{\vol}{\mathrm{vol}}
\newcommand{\Supp}{\mathrm{Supp}}
\newcommand{\des}{\operatorname{des}}
\newcommand{\negg}{\operatorname{neg}}

\def \bfd {\mathbf{d}}
\def \bfw {\mathbf{w}}

\def \mfh {\mathfrak{h}}
\def \mfp {\mathfrak{p}}
\def \Zp {\mathbb{Z}_p}

\def \mcP {\mathcal{P}}
\def \mcW {\mathcal{W}}
\def \lri {\mathfrak{o}}
\def \mfh {\mathfrak{h}}
\def \rarr {\rightarrow}
\newcommand{\desB}{\operatorname{des}_B}
\newcommand{\DesB}{\operatorname{Des}_B}
\DeclareMathOperator{\Res}{Res}

\title{Symplectic lattice counting and zeta functions of higher Heisenberg groups}

\author{Jianhao Shen}
\author{Christopher Voll}

\email{jshen@math.uni-bielefeld.de}
\email{C.Voll.98@cantab.net}

\address{Fakult\"at f\"ur Mathematik, Universit\"at Bielefeld, D-33501
Bielefeld, Germany}

\subjclass[2020]{20C08, 11M41, 11S45, 20F18, 17B30, 05A15}

\begin{document}

\begin{abstract}
  We derive explicit formulae for the subalgebra zeta functions of all higher
  Heisenberg Lie algebras over an arbitrary compact discrete valuation ring
  (cDVR)~$\lri$. To this end, we develop Hecke-theoretic techniques for the
  enumeration, by two distinct invariants, of sublattices of an $\lri$-lattice
  of finite rank endowed with a non-degenerate symplectic form.
\end{abstract}

\date{\today} \maketitle

\tableofcontents

\thispagestyle{empty}

\section{Introduction}

In this paper we combine two themes in enumerative algebra. The first is the
enumeration of subalgebras of the higher Heisenberg algebras over compact
discrete valuation rings (cDVRs). The second is the enumeration, by two
natural invariants, of sublattices in symplectic lattices over cDVRs.

\subsection{Subalgebra zeta functions of higher Heisenberg groups and algebras}
\label{subsec:intro-zeta}
In \cite{GSS/88}, Grunewald, Segal, and Smith quantified the subgroup
growth of finitely generated nilpotent groups. The \emph{subgroup zeta
function} of such a group \(G\) is the Dirichlet-type generating
series
\[
  \zeta_G(s)=\sum_{m=1}^\infty a_m(G)m^{-s},
\]
where \(s\) is a complex variable and \(a_m(G)\) denotes the number of
subgroups of \(G\) of index~\(m\). The nilpotency of \(G\) is reflected in
the Euler product decomposition
\begin{equation}\label{equ:euler}
  \zeta_{G}(s) = \prod_{p \textup{ prime}}\zeta_{G,p}(s),
\end{equation}
where the local subgroup zeta functions are
\[
  \zeta_{G,p}(s)=\sum_{i=0}^{\infty}a_{p^i}(G)p^{-is}.
\]
A central result of \cite{GSS/88} establishes that each of these Euler
factors is a rational function in~\(p^{-s}\). Computing these
functions explicitly, though, is a challenging problem, wide open for
general nilpotent groups. It is therefore notable that the problem
remained open for almost all members of one of the simplest
non-abelian families, namely the higher Heisenberg groups \(H_n(\Z)\)
defined below.

The problem of computing (local) subgroup zeta functions of nilpotent groups
may be linearized in the following sense. By the Malcev correspondence, there
exists a Lie ring (or \(\Z\)-Lie lattice) \(L_G\) such that, for all but
finitely many primes \(p\), the local subgroup zeta function
\(\zeta_{G,p}(s)\) coincides with the local \emph{subalgebra zeta function}
\begin{equation}\label{def:loc.subalg}
  \zeta_{L_G(\Zp)}(s) = \sum_{i=0}^\infty a_{p^i}(L_G(\Zp))p^{-is},
\end{equation}
where \(a_{p^i}(L_G(\Zp))\) is the number of subalgebras of the \(\Zp\)-Lie
algebra \(L_G(\Zp):=L_G\otimes_\Z\Zp\) of finite additive index~\(p^i\). (In
nilpotency class~\(2\), the correspondence holds for \emph{all} primes~\(p\).)
Via base extension, we obtain \(\lri\)-Lie lattices \(L_G(\lri)\) for any
cDVR~\(\lri\). Their subalgebra zeta functions \(\zeta_{L_G(\lri)}(s)\) are
defined in close analogy with~\eqref{def:loc.subalg}, enumerating
\(\lri\)-subalgebras of \(L_G(\lri)\) of finite additive index. They are known
to be rational in~\(q^{-s}\), where \(q\) is the residue field cardinality of
\(\lri\).

In the present paper we give three explicit descriptions of the
subalgebra zeta functions associated with the \emph{Heisenberg algebra
\(\mfh_n(\lri)\) of degree \(n\in\N\)} over an arbitrary
cDVR~\(\lri\); cf.\ Theorems~\ref{thm:A}, \ref{thm:B},
\ref{thm:C}. The \(\Z\)-Lie lattice
\begin{equation}\label{eq:pres}
  \mfh_n(\Z) =\langle x_1,\dots,x_{2n},y \mid [x_{2i-1},x_{2i}]=y
  \textup{ for }i\in[n]\rangle
\end{equation}
is the Lie lattice associated, via the Malcev correspondence, to the
\emph{Heisenberg group \(H_n(\Z)\) of degree \(n\)}, i.e.\ the
\(n\)-fold central amalgamated power of the discrete Heisenberg group
\[
  H(\Z)=
  \begin{pmatrix}
    1 & \Z & \Z \\ 0 & 1 & \Z \\ 0 & 0 & 1
  \end{pmatrix}.
\]
In presentations such as \eqref{eq:pres}, we assume Lie brackets not implied
by the specified ones to vanish. We refer to Lie algebras of the form
\(\mfh_n(\lri)\) collectively as
\emph{higher Heisenberg algebras}.

As a corollary we obtain formulae for the local zeta functions of the
higher Heisenberg groups \(H_n(\Z)\). For \(n=1\), the formula for
\(\zeta_{H_1(\Z),p}(s)=\zeta_{\mfh_1(\Zp)}(s)\) is given
in~\cite[Prop.~8.1]{GSS/88}. The formula for \(n=2\) appears in
\cite{Woodward2005}; the formula for \(n=3\) is due to Klopsch and the
second author and is recorded without proof in~\cite[\S3.3.13.9]{Berman2005}.
Our formulae appear to be new for higher values of~\(n\).

Our approach combines lattice enumeration with ideas and methods from
the theory of Hecke algebras and their modules on the one hand and
combinatorially defined generating functions called \emph{Igusa
functions} on the other hand. On the Hecke side, it is informed by the
theory of alternating modules developed in \cite{SV24,SV25}; on the
Igusa side, it is related to the type-\(\mathsf{A}\) Igusa functions
employed in \cite{Voll2005} (and many other works) and their
type-\(\mathsf{B}\) analogues, e.g.\ in \cite{SV/14}.  We also develop
formulae for the \emph{reduced subalgebra zeta functions} of the
higher Heisenberg algebras, a \((q\to1)\)-degeneration of the local
zeta functions introduced by Evseev~\cite{Evseev2009Reduced}.

Our first main result expresses the subalgebra zeta function of
\(\mfh_n(\lri)\) as a sum of \(2^n\) instantiations of \emph{augmented
Igusa functions}~\(\igu{n}\); cf.~\Cref{def:igusa.A}.

\begin{thmabc}[=\Cref{thm:final-zeta}]\label{thm:A}
  We have
  \[
  \zeta_{\mfh_n(\lri)}(s)
  =
  \sum_{\mathbf{w}\in\mcW_n}
  C_n(\mathbf{w})\,
  \igu{n}\!\bigl(q^{-2};X_0(\mathbf w),X_1(\mathbf w),\dots,X_n(\mathbf w)\bigr),
\]
with \(\mcW_n\) and \(C_n(\mathbf w)\) as in \Cref{def:Wn} and \(X_k(\bfw)\)
as defined in~\eqref{def:Xk}.
\end{thmabc}

Our second main result expresses the zeta function as a much more compact
\((n+1)\)-term sum, revealing a surprisingly rigid pole structure. It is
obtained by a recursive regrouping and normalization of the \(2^n\)-term
expression in~\Cref{thm:A}. Let \((a;q)_m:=\prod_{i=0}^{m-1}(1-aq^i)\) denote the
\emph{\(q\)-Pochhammer symbol}.

\begin{thmabc}[=\Cref{thm:simplified-zeta}]\label{thm:B}
We have
\[
  \zeta_{\mfh_n(\lri)}(s)
  =
  \sum_{r=0}^{n}
  \frac{1}{\left(1-q^{\,2n+\frac{r(2n+1-r)}{2}}\,T^{n+1}\right)}
  \cdot
  \frac{
    (-q)^r(1-q^{2n-2r+1})(q^2;q^2)_n
  }{
    (q;q)_{2n-r+1}(q;q)_r
    (q^{r}T;q^{2})_{n-r}
    (q^{2n-r}T;q)_{r}
  }.
\]
\end{thmabc}

Our third main result expresses the zeta function in terms of the
\emph{type-\(\mathsf B\) Igusa function}~\(\ig{B_n}\); see
\Cref{def:igusa.B}. The numerator of this rational function is a
polynomial whose monomials are indexed by the hyperoctahedral group
\(B_n\) of order \(2^n n!\), involving the well-known statistics
\(\negg\) and \(\desB\) on \(B_n\) recalled in \Cref{sec:Bn}, as well
as a new statistic \(C\) defined in~\eqref{eq:Cg}. For \(i\in[n]_0\),
set \(c_i=\frac{n(n+5)-i(i+1)}{2} = 2n + \binom{n+1}{2}-\binom{i+1}{2}\).

\begin{thmabc}[=\Cref{thm:Bn-target}]\label{thm:C}
    We have
\begin{align*}
  \zeta_{\mfh_n(\lri)}(s)
  &= \frac{1}{(T;q)_{2n}}\,
     \ig{B_n}\!\left(
       q^{-1},-q^nT;\, q^{c_0}T^{n+1},\dots,q^{c_n}T^{n+1}
     \right)\\ &=
    \frac{
      \displaystyle
      \sum_{g\in B_n}
      (-1)^{\negg(g)}\,
      q^{C(g)}\,
      T^{(n+1)\desB(g)+\negg(g)}
    }{
      \displaystyle
      (T;q)_{2n}
      \prod_{m=0}^{n}
      \bigl(1-q^{c_m}T^{n+1}\bigr)
    }.
  \end{align*}
\end{thmabc}

\begin{remark}
Modifying the data $c_i$ to $c_i' = \binom{n+1}{2}-\binom{i+1}{2}$
(also in the definition of $C$) yields the \emph{graded subalgebra
zeta functions} $\zeta^{\textup{gr}}_{\mfh_n(\lri)}(s)$ enumerating
\emph{homogeneous} subalgebras; cf.\ \cite[\S3]{R/18}. We record the
curious fact that $c_i'$ is the dimension of the algebraic variety of
symmetric $n\times n$-matrices of rank~$n-i$.
\end{remark}

We describe the poles of the \(\mfp\)-adic zeta
functions~\(\zeta_{\mfh_n(\lri)}(s)\) in~\Cref{subsec:poles}. Our work
also yields an elementary proof of the following local functional
equations.

\begin{cor}[\cite{Voll/10}]\label{cor:funeq}
  We have
  \[
    \left.\zeta_{\mfh_n(\lri)}(s)\right|_{q \rarr q^{-1}}
    =
    -q^{\binom{2n+1}{2}-(2n+1)s}\zeta_{\mfh_n(\lri)}(s).
  \]
\end{cor}

Owing to the Euler product~\eqref{equ:euler}, our local results have global
consequences for the analytic and asymptotic properties of the subgroup growth
of the discrete Heisenberg groups \(H_n(\Z)\). Recall that the abscissa of
convergence \(\alpha_n:=\alpha(H_n(\Z))\) of \(\zeta_{H_n(\Z)}(s)\) is the
infimum of the real numbers \(a\) such that \(\zeta_{H_n(\Z)}(s)\) converges
for \(\operatorname{Re}(s)>a\). It determines the degree of polynomial growth of
\[
  s_N(H_n(\Z)):=\sum_{m\le N}a_m(H_n(\Z))
\]
as \(N\rarr \infty\). By \cite[Thm.~1.1]{duSG00}, the zeta function
\(\zeta_{H_n(\Z)}(s)\) admits some meromorphic continuation beyond its
abscissa of convergence. A Tauberian theorem then relates the
asymptotics of \(s_N(H_n(\Z))\) to the position and order of the pole
at \(s=\alpha_n\), and to the behaviour of the continued function near
this pole.  The authors of \cite{duSG00} illustrate their theorem with
the integral Heisenberg group \(H_1(\Z)\): the subgroup zeta function
  \[
    \zeta_{H_1(\Z)}(s)=
    \frac{\zeta(s)\zeta(s-1)\zeta(2s-2)\zeta(2s-3)}{\zeta(3s-3)}
  \]
  has a double pole at \(\alpha_1=2\), resulting in the term
  \(\log N=(\log N)^{2-1}\) in
  \begin{equation}\label{equ:n=2}
    \sum_{m\le N}a_m(H_1(\Z))
    \sim
    \frac{\zeta(2)^2}{2\zeta(3)}N^2\log N .
  \end{equation}

  Our work shows that the generic members of the family
  \(\left(\zeta_{H_n(\Z)}(s)\right)_{n\in\N}\) follow a different
  pattern: the global subgroup zeta functions are no longer products
  of translates of Riemann zeta functions and their inverses; the pole
  at \(\alpha_n=2n\) is simple for~\(n\ge2\). To make this precise, we
  set
  \begin{equation}\label{def:N}
    N_n(X,Y)  = \sum_{g\in B_n}(-1)^{\negg(g)}X^{C(g)}Y^{D(g)}\in\Z[X,Y],
  \end{equation}
  with \(D(g):=(n+1)\desB(g)+\negg(g)\) and \(C(g)\) as in~\eqref{eq:Cg}, and
  \begin{equation}\label{def:R}
    R_n=
    \left(\prod_pN_n(p,p^{-2n})\right)
    \prod_{i=0}^{2n-2}\zeta(2n-i)
    \prod_{i=0}^{n}
    \zeta\left(2n^2-\binom{n+1}{2}+\binom{i+1}{2}\right).
  \end{equation}

  \begin{cor}\label{cor:global}
    For all \(n\geq1\), the abscissa of convergence of \(\zeta_{H_n(\Z)}(s)\) is
    \(\alpha_n=2n\). Moreover, for \(n\geq2\), the pole at \(s=2n\) is simple and
  \[
    s_N(H_n(\Z))\sim \frac{R_n}{2n}N^{2n}\qquad(N\to\infty).
  \]
\end{cor}

\begin{remark}
  It appears that only for \(n=1\) is the number \(R_n\) a product of
  special values of Riemann's zeta function and its inverse. In
  general, \(R_n\) comprises an Euler product over evaluations at
  primes of a generating polynomial for statistics on the Weyl group
  \(B_n\). This polynomial appears to have large irreducible factors;
  see \Cref{tab:N} for \(n\in\{1,2,3\}\). This seems to be a recurrent
  theme in zeta functions in enumerative algebra; see
  \cite[Sec.~4.1]{CSV/18} or \cite[Sec.~7.3]{AMV/25} for other recent
  instances of this phenomenon.

  It seems of great interest to understand the analytic behaviour of
  \(\zeta_{H_n(\Z)}(s)\) beyond its leading pole, for instance to detect
  further poles or natural boundaries for meromorphic
  continuation. For a template of such an analysis, see~\cite{BV/26}.
\end{remark}

In \Cref{subsec:red} we derive explicit formulae for the reduced
subalgebra zeta functions associated with the higher Heisenberg
algebras.

\subsection{Enumerating $p$-adic lattices by (symplectic) type}
\label{subsec:latt}
Our results on the subalgebra zeta functions of the higher Heisenberg
algebras rest on lattice enumeration in symplectic spaces. Let
\(L=\lri^{2n}\) be the free \(\lri\)-module of rank~\(2n\), equipped
with a non-degenerate symplectic form \(\langle\,,\rangle\). We count
finite-index sublattices \(\Lambda\le L\) simultaneously by two
invariants: their \emph{quotient type}, namely the isomorphism type of
\(L/\Lambda\), encoded by a partition~\(\lambda\in\mcP_{2n}\) with at
most $2n$ nonzero parts, and their \emph{symplectic type}, namely the
elementary divisor type of the Gram matrix of
\(\langle\,,\rangle|_\Lambda\), encoded by a partition
\(\mu\in\mcP_n\) with at most $n$ nonzero parts.

Given \(\lambda\in\mcP_{2n}\) and \(\mu\in\mcP_n\), denote by
\(N_\lri(\lambda,\mu)\) the number of such sublattices;
cf.\ \Cref{def:lat}. Motivated by \cite[Thm.~3.3]{SV25}, we let
\(N'_\lri(\lambda,\mu)\) be the number of Lagrangian submodules of
type \(\lambda\) in an alternating module of type~\(\mu\). Using
Hecke-theoretic methods, we show that these invariants are related to
each other via the \emph{Birkhoff
number}~\(\alpha_n(\mu;q^2)\)~\eqref{eq:Birk-n-closed}.

\begin{prop}\label{prop:fact}
  For all \(\lambda\in\mcP_{2n}\) and \(\mu\in\mcP_n\), we have
  \[
    N_\lri(\lambda,\mu)
    =
    N'_\lri(\lambda,\mu)\,\alpha_n(\mu;q^2).
  \]
\end{prop}

It follows that \(N_\lri(\lambda,\mu)\neq0\) iff
\(N'_\lri(\lambda,\mu)\neq0\), which occurs only if~\(|\lambda|=|\mu|\).

\subsection{Related work: zeta functions associated with discrete Heisenberg groups}
It has become somewhat of a cottage industry in recent years to
compute local and global subobject zeta functions of groups, rings,
and modules.  Many generalize the prototypical subgroup zeta function
of the discrete Heisenberg group $H_1(\Z)$, first computed in
\cite{GSS/88}.

Zeta functions of semidirect products generalizing the discrete
Heisenberg group are treated in~\cite{voll2006}. Zeta functions of
free class-\(2\) nilpotent groups \(F_{2,d}\)---including the smallest one
\(H_1(\Z)=F_{2,2}\)---are computed in~\cite{SVV/24}. Numerous further
examples are amenable to the computer algebra package
\textsf{\textup{Zeta}} (\cite{RossmannZeta, R/18}), written and
maintained by Rossmann.

While the \emph{normal subgroup zeta functions} of the higher Heisenberg
groups are rather trivial (see~\eqref{equ:Heisenberg.normal}), enumerating
normal subgroups of direct products of Heisenberg groups gives rise to
interesting combinatorial generating functions \cite{SV1/15}, leading to
generalizations of the type-\(\mathsf{A}\) Igusa functions featuring in the
present paper in \cite{CSV/24}.

The higher Heisenberg groups considered in the present paper are some of the
relatively few finitely generated nilpotent groups for which explicit formulae
for the (local) \emph{pro-isomorphic zeta functions} are available. These zeta
functions enumerate subgroups whose profinite completion is isomorphic to the
ambient group's completion. Theorem~5.10 in \cite{BGS/22} yields an expression for
the local pro-isomorphic zeta functions of groups of the form \(H_n(\Z)\) in
terms of what we call augmented Igusa functions of degree~\(n\);
cf.~\Cref{def:igusa.A}.

Relatives of the discrete Heisenberg group \(H_1(\Z)\) were among the
first finitely generated nilpotent groups for which
\emph{representation zeta functions}---enumerating twist-isoclasses of
finite-dimensional irreducible complex representations---were
computed. Generalizing \cite{E/14}, Stasinski and the second author
developed a framework for computing such zeta functions
in~\cite{SV/14}. Their general results use methods from \(p\)-adic
integration; their more specific ones introduce combinatorial
generating functions akin to the ones featuring in the present paper;
cf.\ \Cref{rem:SV14.1,rem:SV14.2}. In particular,
\cite[Thm.~B]{SV/14} gives formulae for the representation zeta
functions of three infinite families generalizing the integral
Heisenberg group \(H_1(\Z)\). (As in the case of normal subgroup
growth, the representation growth of the higher Heisenberg groups
\(H_n(\Z)\) is near-trivial.)

Recent instances of applications of Hecke-theoretic methods to enumeration
problems associated with \(p\)-adic lattices are \cite{AMV/25, MV2}.

\subsection{Organization}

In \Cref{sec:counting-lam-mu} we introduce the counting functions
\(N_\lri(\lambda,\mu)\) and~\(N'_\lri(\lambda,\mu)\). In
\Cref{sec:hecke} we develop the Hecke-theoretic machinery necessary to
prove~\Cref{prop:fact}. In \Cref{sec:counting-symp} we compute the
aggregate count \(N'_\lri(\mu):=
\sum_{\lambda\in\mcP_{2n}}N'_\lri(\lambda,\mu)\);
cf.\ \Cref{prop:closed.N'}. In \Cref{sec:heis} we relate these results
to the subalgebra zeta functions of the higher Heisenberg algebras. In
\Cref{sec:Wn-parametrized-formula},
\Cref{sec:simplifying-zeta-formula}, and \Cref{sec:Bn}, we prove
\Cref{thm:A}, \Cref{thm:B}, and \Cref{thm:C}, respectively. Finally,
in \Cref{sec:properties} we study reduced zeta functions and further
local and global analytic properties.

\subsection{Further notation}\label{subsec:setup}

We write \(\N=\{1,2,\dots\}\) and \(\N_0=\N\cup\{0\}\). Throughout,
\(n\in\N\), and we set \([n]:=\{1,\dots,n\}\) and
\([n]_0:=\{0,1,\dots,n\}\).

We let \(k_{\mathfrak p}\) be a non-archimedean local field with ring of
integers \(\lri\), maximal ideal \(\mathfrak p=(\pi)\), and residue field of
cardinality \(q\), a prime power. Throughout we write \(T=q^{-s}\).

We write \(\mcP_r\) for the set of partitions with at most \(r\) parts. For
\(\lambda=(\lambda_1,\dots,\lambda_r)\in\mcP_r\), padded with trailing zeros if
necessary, set \(|\lambda|=\sum_{i=1}^r\lambda_i\).

\bigskip
We record some further notation used throughout the paper in the following
table.

\begingroup
\renewcommand\arraystretch{1.14}
\begin{longtable}{l|l|l}
  Symbol & Description & Reference \\ \hline \hline
  $\lri_\lambda$ & finite $\lri$-module of type $\lambda$ & \Cref{sec:counting-lam-mu} \\
  $M_\mu$ & alternating module of type $\mu$ & \Cref{subsec:alt-module} \\
  $L$ & ambient symplectic lattice $\lri^{2n}$ & \Cref{def:lat} \\
  $J$, $J_\mu$ & standard and type-$\mu$  Gram matrices & \Cref{def:gram} \\
  $N'_\lri(\lambda,\mu)$, $N_\lri(\lambda,\mu)$ & intrinsic and lattice counts & \Cref{def:int,def:lat} \\
  $N'_\lri(\mu)$, $N_\lri(\mu)$ & counts by alternating type only & \Cref{def:sym-comp,def:sym} \\
  $[r]_t$, $\binom{n}{r}_t$, $\binom{n}{I}_t$ & $t$-integer and Gaussian coefficients & \Cref{subsec:relate-counts} \\
  $\alpha_n(\mu;q)$ & Birkhoff number & \Cref{eq:Birk-n-closed} \\
  $n(\lambda)$ & $\sum_{i\ge1}(i-1)\lambda_i$ & \Cref{eq:n-lambda} \\
  $\mathcal{S}, \mathcal{S}^{\textup{alt}}$ & (alternating) Satake transform & \Cref{subsec:satake-algebra}, \Cref{subsec:alt-satake}\\
  $\mcW_n$, $C_n(\mathbf w)$ & index set and weights for $N'_\lri(\mu)$ & \Cref{def:Wn} \\
  $\igm{n},\ig{n},\igu{n}$ & truncated, usual, augmented Igusa functions & \Cref{def:igusa.A} \\
  $\mcW_{k,r}$, $\mathcal I_n^{k,r}$ & fibres and fibre sums in the simplification & \Cref{subsec:Wn-structure,def:Inkr} \\
  $\ig{B_n}$, $\igm{B_n}$ & type-$\mathsf{B}$ Igusa functions & \Cref{def:igusa.B} \\
  $Z^{\mathrm{red}}_{\mfh_n}(T)$ & reduced zeta function &
  \Cref{subsec:red}
\end{longtable}
\endgroup

\section{Counting lattices by quotient and symplectic types}
\label{sec:counting-lam-mu}
Let $\lri$ be a cDVR with maximal ideal $\mfp$, uniformizer $\pi$, field of
fractions $k_\mfp$, and residue field of cardinality $q$. Let $n\in\N$. In
this section we address the enumeration of finite-index sublattices of a
non-degenerate $\lri$-lattice $(L,\langle\,,\rangle)$, whose underlying
$\lri$-module is isomorphic to $\lri^{2n}$, by the two invariants sketched in
\Cref{subsec:latt}, namely their quotient type $\lambda\in\mcP_{2n}$ and
symplectic type $\mu\in\mcP_n$.

In \Cref{prop:fact}, the main result of this section, we express the
`lattice count' \(N_\lri(\lambda,\mu)\), i.e.\ the number of such sublattices
(cf.\ \Cref{def:lat}), in terms of an `intrinsic count'
\(N'_\lri(\lambda,\mu)\) of Lagrangian submodules of type \(\lambda\) inside a
finite alternating module of type~\(\mu\) (cf.\ \Cref{def:int}) and the
\emph{Birkhoff number} of lattices of given quotient type in a free
$\lri$-lattice of given rank. To prepare this result, we introduce some
notation. For $r\ge1$, set
\[
  \Z^r_{\mathrm{dom}}
  =
  \{\lambda=(\lambda_1,\dots,\lambda_r)\in\Z^r \mid
    \lambda_1\ge\cdots\ge\lambda_r\},
\]
so that $\mcP_r=\{\lambda\in\Z^r_{\mathrm{dom}} \mid \lambda_r\ge0\}$. We set
$|\lambda|=\sum_{i=1}^r\lambda_i$. For
$\lambda=(\lambda_1,\dots,\lambda_r)\in\mcP_r$, define the finite
$\lri$-module
\[
  \lri_\lambda
  :=
  (\lri/\pi^{\lambda_1})\oplus\cdots\oplus(\lri/\pi^{\lambda_r}).
\]
An $\lri$-module is said to be \emph{of type $\lambda$} if it is isomorphic to
$\lri_\lambda$. Clearly $|\lri_\lambda| = q^{|\lambda|}$.

\subsection{Alternating modules}
\label{subsec:alt-module}
Set $M_0:=0$. For $k\ge1$, consider
\[
  M_k
  =
  (\lri/\mfp^k)e
  \oplus
  (\lri/\mfp^k)f,
\]
a free $\lri/\mfp^k$-module of rank two with abstract generators $e$ and~$f$.
We equip $M_k$ with an $\lri$--bilinear \emph{perfect alternating pairing}
\[
  (\ ,\ ) \colon M_k \times M_k \longrightarrow k_{\mathfrak p}/\lri,
  \qquad
  (e,f)=\pi^{-k}\bmod\lri,
  \qquad
  (e,e)=(f,f)=0.
\]

For $n\geq 0$ and $\mu=(\mu_1,\dots,\mu_n)\in\N_0^n$, set
\[
  M_\mu := M_{\mu_1}\oplus\cdots\oplus M_{\mu_n}.
\]
This endows $M_\mu$ with the structure of an alternating $\lri$--module in the
sense of \cite[\S3.1]{SV25}. We refer to $\mu$ as the \emph{alternating type}
of $M_\mu$.

A submodule $N\subseteq M_\mu$ is called \emph{Lagrangian} if $N=N^\perp$ with
respect to~$(\,,)$.

\begin{definition}[Intrinsic count]\label{def:int}
  Given $\lambda\in\mcP_{2n}$ and $\mu\in\mcP_n$, we define
  \[
    N'_\lri(\lambda,\mu)
    =
    \#\bigl\{
      N\subseteq M_\mu \mid N=N^\perp,\ N\simeq\lri_\lambda
    \bigr\},
  \]
  the number of Lagrangian submodules of \(M_\mu\) of type \(\lambda\).
\end{definition}
Note that the type $\lambda\in\mcP_{2n}$ of a Lagrangian submodule satisfies
$|\lambda|=|\mu|$, reflecting the fact that~$|N| = \sqrt{|M_\mu|}$. Hence
$N'_\lri(\lambda,\mu)=0$ unless $|\lambda|=|\mu|$.

\subsection{Invariants of symplectic lattices}
\label{subsec:inv.symp.lat}
Let $\beta=(e_1,\dots,e_n,f_1,\dots,f_n)$ be a symplectic basis for the
non-degenerate symplectic $\lri$-lattice $(L,\langle\,,\rangle)$. That is, for
$i,j\in[n]$,
\[
  \langle e_i,f_j\rangle=\delta_{ij},
  \qquad
  \langle e_i,e_j\rangle=\langle f_i,f_j\rangle=0.
\]
The Gram matrix of $\langle\,,\rangle$ with respect to $\beta$ is
\[
  J=
  \begin{pmatrix}0_n&I_n\\-I_n&0_n\end{pmatrix}
  \in\Mat_{2n}(\lri).
\]
Let $\Lambda\subseteq L$ be a finite-index sublattice. Then $L/\Lambda$ is a
finite $\lri$-module, hence $L/\Lambda\cong\lri_\lambda$ for a unique
$\lambda\in\mcP_{2n}$. We call $\lambda$ the \emph{quotient type} of
$\Lambda$.

The restriction $\langle\,,\rangle|_\Lambda$ is an alternating
$\lri$-bilinear form which is non-degenerate over $k_{\mathfrak p}$. By
elementary divisor theory, there exists an $\lri$-basis
$\beta_\Lambda=(e_1',\dots,e_n',f_1',\dots,f_n')$ of $\Lambda$ such that the
Gram matrix of $\langle\,,\rangle|_\Lambda$ with respect to $\beta_\Lambda$ is
\begin{equation}\label{def:gram}
  J_\mu
  =
  \begin{pmatrix}
    0_n & \pi^\mu\\
    -\pi^\mu & 0_n
  \end{pmatrix},
  \qquad
  \pi^\mu=\diag(\pi^{\mu_1},\dots,\pi^{\mu_n}),
\end{equation}
for a unique $\mu\in\mcP_n$. We call $\mu$ the \emph{alternating type} of
$\Lambda$.

\begin{definition}[Lattice count]\label{def:lat}
  Given $\lambda\in\mcP_{2n}$ and $\mu\in\mcP_n$, we define
  \[
    N_\lri(\lambda,\mu)
    =
    \#\bigl\{
      \Lambda\subseteq L \mid
      L/\Lambda\simeq\lri_\lambda,\
      \langle\,,\rangle|_\Lambda \sim J_\mu
    \bigr\},
  \]
  the number of finite-index sublattices in $L$ with quotient type $\lambda$
  and alternating type $\mu$.
\end{definition}
\begin{remark}\label{rem:restore-parameters}
  When necessary, we write \(N_{n,\lri}(\lambda,\mu)\) and
  \(N'_{n,\lri}(\lambda,\mu)\) to emphasize the dependence on the rank
  parameter \(n\) and the discrete valuation ring \(\lri\). Throughout most of
  the paper, \(\lri\) is fixed and \(n\) is implicitly determined by the
  conditions \(\lambda\in\mcP_{2n}\) and \(\mu\in\mcP_n\). Accordingly, we
  usually suppress these parameters from the notation.
\end{remark}

\subsection{Relating lattice and intrinsic counts}
\label{subsec:relate-counts}
In \Cref{prop:fact} we relate the numbers \(N_\lri(\lambda,\mu)\) and
\(N'_\lri(\lambda,\mu)\). The proof uses Hecke-theoretic methods and is carried
out in the subsequent section, concluding in \Cref{subsec:factorization}.

We state this result in terms of the \emph{Birkhoff number}
\(\alpha_n(\mu;q)\), the number of finite-index sublattices of $\lri^n$ of
quotient type $\mu$. For a partition \(\mu=(\mu_1,\dots,\mu_n)\in\mcP_n\), let
\(m_j(\mu)\) denote the multiplicity of the part \(j\). For an indeterminate
\(Y\), define integral polynomials
\[
  [r]_Y := \frac{1-Y^{\,r}}{1-Y},
  \qquad
  [r]_Y! := \prod_{k=1}^r [k]_Y,
  \qquad
  \binom{n}{r}_Y := \frac{[n]_Y!}{[r]_Y!\,[n-r]_Y!}.
\]
More generally, for a subset \(I=\{i_1<\cdots<i_\ell\}\subseteq[n]_0\), define
the \emph{Gaussian multinomial coefficient}
\[
  \binom{n}{I}_Y
  :=
  \binom{n}{i_\ell}_Y
  \binom{i_\ell}{i_{\ell-1}}_Y
  \cdots
  \binom{i_2}{i_1}_Y \in \Z[Y].
\]
Set
\[
  \rho := \tfrac12\,(n-1,n-3,\dots,1-n)\in\Q^n.
\]
By \cite[Eq.~V.2.9]{Macdonald}, the Birkhoff number may be written in the form
\begin{equation}\label{eq:Birk-n-closed}
  \alpha_n(\mu;q)
  =
  q^{\langle \mu,2\rho\rangle}
  \,
  \frac{[n]_{q^{-1}}!}{
    \displaystyle\prod_{j\ge 0}[\,m_j(\mu)\,]_{q^{-1}}!
  }.
\end{equation}

We define the difference vector \(\mathbf d=d(\mu)\in\N_0^n\) by
\(d_i=\mu_i-\mu_{i+1}\) for \(i<n\) and \(d_n=\mu_n\). Set
\(\boldsymbol{\rho}'=(\rho'_1,\dots,\rho'_n)\), where \(\rho'_k=k(n-k)\), and
define the restricted positive support
\[
  \Supp^+_{n-1}(\mathbf d)
  :=
  \{\,i\in[n-1] \mid d_i>0\,\}.
\]
Then \(\langle \mu,2\rho\rangle=\mathbf d\cdot\boldsymbol{\rho}'\). With this
notation, the Birkhoff number admits the following reformulation, which will be
useful in \Cref{subsec:Z-as-igusa}:
\begin{equation}\label{eq:Birk-n-support}
  \alpha_n(\mu;q)
  =
  q^{\mathbf d\cdot\boldsymbol{\rho}'}
  \binom{n}{\Supp^+_{n-1}(\mathbf d)}_{q^{-1}}.
\end{equation}

\begin{remark}
  Appending zeros to \(\lambda\) and \(\mu\) changes neither the alternating
  module \(M_\mu\) nor the intrinsic count \(N'_\lri(\lambda,\mu)\). However,
  $\alpha_n(\mu;q^2)$ depends on the number of parts~$n$. Hence
  $N_{n,\lri}(\lambda,\mu)$ varies with $n$, even though the partitions
  $\lambda$ and $\mu$ are extended only by zeros.
\end{remark}
\begin{remark}\label{rem:sum-over-mu-birk}
  For every $\lambda\in\mcP_{2n}$, summing the lattice counts over all possible
  alternating types $\mu$ recovers the Birkhoff number of sublattices of
  quotient type $\lambda$:
  \[
    \sum_{\mu\in\mcP_n} N_\lri(\lambda,\mu)
    =
    \alpha_{2n}(\lambda;q).
  \]
  By contrast, summing over all quotient types yields the invariant
  \(N_\lri(\mu):=\sum_{\lambda\in\mcP_{2n}}N_\lri(\lambda,\mu)\),
  cf.\ \Cref{def:sym}, which will be investigated in \Cref{sec:counting-symp}.
\end{remark}

\section{Hecke algebras, alternating modules, and counting problems}\label{sec:hecke}
In this section we develop the Hecke-theoretic framework underlying
our counting problems. In \Cref{subsec:spherical-hecke} we introduce
the spherical Hecke algebra of $\GL_{2n}$, in
\Cref{subsec:alt-hecke-module} its alternating Hecke module. We
interpret their structure constants in terms of Hall numbers and
lattice counts in \Cref{subsec:hecke-hall} and
\Cref{subsec:hecke-to-counting}. The Satake and alternating Satake
transforms are recalled in \Cref{subsec:satake-algebra} and
\Cref{subsec:alt-satake}, where Hecke convolution is translated into
identities of Hall--Littlewood polynomials. In
\Cref{subsec:factorization} we bring these tools together to
prove~\Cref{prop:fact}.

\subsection{The spherical Hecke algebra}\label{subsec:spherical-hecke}
We begin by recalling the spherical Hecke algebra of
$G=\GL_{2n}(k_{\mathfrak p})$ with respect to $K=\GL_{2n}(\lri)$,
together with its standard basis and structure constants; see
\cite[Chap.~V, \S2]{Macdonald} for details.  Let
\[
  G:=\GL_{2n}(k_{\mathfrak p}),\qquad
  K:=\GL_{2n}(\lri),
\]
and let
\[
  G^{+}:=G\cap\Mat_{2n}(\lri)
\]
be the subsemigroup of integral matrices. The spherical Hecke algebra
$\mathcal H(G,K)$ is the $\C$-algebra of compactly supported,
$K$-bi-invariant functions on~$G$, equipped with convolution
\[
  (f_1*f_2)(g)
  =
  \int_G f_1(h)\,f_2(h^{-1}g)\,dh,
\]
where the Haar measure is normalized by $\vol(K)=1$. We denote by
$\mathcal H(G^{+},K)$ the subalgebra consisting of functions supported on
$G^{+}$. By the Cartan decomposition,
\[
  G=\bigsqcup_{\lambda\in\Z^{2n}_{\mathrm{dom}}}K\pi^\lambda K,
  \qquad
  \pi^\lambda:=\diag(\pi^{\lambda_1},\dots,\pi^{\lambda_{2n}}).
\]
Setting $c_\lambda:=\mathbf 1_{K\pi^\lambda K}$ yields a $\C$-basis
$\{c_\lambda\}_{\lambda\in\Z^{2n}_{\mathrm{dom}}}$ of $\mathcal
H(G,K)$. The convolution is given by
\[
  c_\lambda*c_\eta
  =
  \sum_{\theta\in\Z^{2n}_{\mathrm{dom}}}
  g_{\lambda,\eta}^{\ \theta}\,c_\theta,
\]
where only finitely many of the \emph{Hecke algebra structure
constants} $g_{\lambda,\eta}^{\ \theta}\in\C$ are nonzero.
\subsection{The alternating Hecke module}
\label{subsec:alt-hecke-module}
Following \cite{HironakaSato}, we introduce a module $\mathcal{M}$ for
the Hecke algebra $\mathcal H(G,K)$ and its realization via
alternating matrices. Let
\[
  H:=\Sp_{2n}(k_{\mathfrak p})=\{g\in G:gJg^t=J\},
\]
where
\(
  J=
  \begin{pmatrix}0&I_n\\-I_n&0
  \end{pmatrix}.
\) We consider the (\emph{alternating}) \emph{Hecke module}
\[
  \mathcal M:=\C_c(K\backslash G/H),
\]
consisting of left \(K\)- and right \(H\)-invariant functions
\(f:G\to\C\) with compact support modulo \(H\). The spherical Hecke algebra
$\mathcal H(G,K)$ acts on $\mathcal M$ by convolution:
\[
  (f\star\varphi)(g)
  =
  \int_G f(h)\,\varphi(h^{-1}g)\,dh,
  \qquad
  f\in\mathcal H(G,K),\ \varphi\in\mathcal M.
\]

Equivalently, let $X$ be the space of nondegenerate alternating
$2n\times2n$ matrices over~$k_{\mathfrak p}$. The group $G$ acts
transitively on $X$ by $g\cdot x:=g\,x\,g^t$, and the stabilizer of
$J$ is precisely~$H$.  Thus the map
\begin{equation}\label{eq:G-over-H-to-X}
  G/H\longrightarrow X,\qquad gH\longmapsto gJg^t.
\end{equation}
is a $G$-equivariant bijection. Under this identification, $\mathcal
M$ may be viewed as the space of compactly supported $K$-invariant
functions on~$X$.

By the elementary divisor theory for alternating forms, each $K$-orbit
in $X$ contains a unique representative of the form
\[
  J_\mu=
  \begin{pmatrix}
    0&\pi^\mu\\
    -\pi^\mu&0
  \end{pmatrix},
  \qquad
  \mu=(\mu_1\ge\cdots\ge\mu_n)\in\Z^n_{\mathrm{dom}},
\]
where $\pi^\mu=\diag(\pi^{\mu_1},\dots,\pi^{\mu_n})$.
Note that
\begin{equation}\label{def:a.mu}
  J_\mu = a_\mu\, J\, a_\mu^{\,t},
  \qquad
  a_\mu:=\diag(\pi^{\mu_1},\dots,\pi^{\mu_n},1,\dots,1).
\end{equation}
Consequently,
\[
  X=\bigsqcup_{\mu\in\Z^n_{\mathrm{dom}}} K\cdot J_\mu,
  \qquad
  G=\bigsqcup_{\mu\in\Z^n_{\mathrm{dom}}} K a_\mu H.
\]
Setting $\phi_\mu:=\mathbf 1_{K a_\mu H}$ yields a $\C$-basis
$\{\phi_\mu\}_{\mu\in\Z^n_{\mathrm{dom}}}$ of~$\mathcal M$. The
Hecke action is given by
\begin{equation}\label{def:hecke.act}
  c_\lambda\star\phi_\mu
  =
  \sum_{\nu\in\Z^n_{\mathrm{dom}}}
  b_{\lambda,\mu}^{\ \nu}\,\phi_\nu,
\end{equation}
where only finitely many of the \emph{Hecke module structure constants}
$b_{\lambda,\mu}^{\ \nu}\in\C$ are nonzero.

\begin{remark}
  We also use \(c_\lambda\) and \(\phi_\mu\) for arbitrary
  \(\lambda\in\Z^{2n}\) and \(\mu\in\Z^n\), by the same formulas as
  above.  Permuting the entries of \(\lambda\) or \(\mu\) does not
  change the corresponding function. Thus these symbols always denote
  the basis element indexed by the dominant rearrangement. In
  particular, \(\phi_{-\mu}\) is unambiguous.
\end{remark}

\begin{remark}\label{rem:dep-on-o}
  When it is helpful to emphasize the base ring, we write
  $g_{\lambda,\eta}^{\ \theta}(\mathfrak o)$ and
  $b_{\lambda,\mu}^{\ \nu}(\mathfrak o)$ for the structure constants
  defined above.  By \Cref{prop:hecke-poly} below, they depend on
  $\mathfrak o$ only through the residue field cardinality $q$, and we
  may therefore also write $g_{\lambda,\eta}^{\ \theta}(q)$ and
  $b_{\lambda,\mu}^{\ \nu}(q)$.
\end{remark}

\subsection{Hecke algebra coefficients and Hall numbers}
\label{subsec:hecke-hall}

We use Macdonald's identification of the structure constants of the
spherical Hecke algebra of \(\GL_{2n}\) with Hall numbers: for
\(\lambda,\eta\in\mcP_{2n}\), \cite[Ch.~V, (2.6)]{Macdonald} gives
\[
  c_\lambda*c_\eta
  =
  \sum_{\theta\in\mcP_{2n}}
  g_{\lambda,\eta}^{\ \theta}(q)\,c_\theta,
\]
where \(g_{\lambda,\eta}^{\ \theta}(q)\) is the Hall polynomial,
enumerating submodules \(N\subseteq \lri_\theta\) such that
\(N\simeq\lri_\eta\) and \(\lri_\theta/N\simeq\lri_\lambda\).  For
arbitrary dominant integral weights, one reduces to partitions by a
central shift, as in \cite[Ch.~V, (2.5)]{Macdonald}.  In particular,
\(g_{\lambda,\eta}^{\ \theta}=0\) unless
\(|\theta|=|\lambda|+|\eta|\), and $ g_{\lambda,\eta}^{\ \theta} =
g_{\eta,\lambda}^{\ \theta}$ by the symmetry of Hall numbers
\cite[Ch.~II, (1.5)]{Macdonald}.

For later use we record the lattice interpretation underlying
\cite[Ch.~V, (2.6)]{Macdonald}. Recall that \(G^+\) is the integral semigroup
defined above. Then
\begin{equation}\label{eq:GplusK-lattices}
  G^{+}/K
  \;\longleftrightarrow\;
  \{\text{finite-index sublattices of }\lri^{2n}\},
  \qquad
  gK\longmapsto L(g),
\end{equation}
where \(L(g)\) is the \(\lri\)-span of the columns of \(g\). Under this
correspondence,
\[
  g\in K\pi^\lambda K
  \qquad\Longleftrightarrow\qquad
  \lri^{2n}/L(g)\simeq \lri_\lambda .
\]

\subsection{Hecke module coefficients and lattice enumerations}
\label{subsec:hecke-to-counting}
We analyze the Hecke module coefficients and specialize them to obtain
$N_\lri(\lambda,\mu)$ and $N'_\lri(\lambda,\mu)$.

\subsubsection*{Hecke module coefficients}
We now turn to the Hecke module $\mathcal M$ with basis
$\{\phi_\mu\}_{\mu\in\Z^n_{\mathrm{dom}}}$.
Recall the Hecke action~\eqref{def:hecke.act}.
Evaluating the convolution at the coset $a_\nu H$, we obtain
\[
  b_{\lambda,\mu}^{\ \nu}
  =
  (c_\lambda\star\phi_\mu)(a_\nu)
  =
  \int_G
  \mathbf 1_{K\pi^\lambda K}(g)\,
  \mathbf 1_{K a_\mu H}(g^{-1}a_\nu)\,dg.
\]
Equivalently,
\begin{equation}\label{eq:b-lambda-mu-nu-cosets}
  b_{\lambda,\mu}^{\ \nu}
  =
  \bigl|\,(K\pi^\lambda K\cap a_\nu H a_\mu^{-1}K)/K\,\bigr|.
\end{equation}
\subsubsection*{Specialization to the lattice count $N_\lri(\lambda,\mu)$}
We now specialize \eqref{eq:b-lambda-mu-nu-cosets} to recover the lattice counting
function $N_\lri(\lambda,\mu)$. By
\eqref{eq:GplusK-lattices}, left \(K\)-cosets in \(G^+\) correspond to
finite-index sublattices of \(L=\lri^{2n}\), via \(gK\mapsto L(g)\), and
\[
  g\in K\pi^\lambda K
  \qquad\Longleftrightarrow\qquad
  L/L(g)\simeq\lri_\lambda .
\]

For $\mu\in\Z^n_{\mathrm{dom}}$ let $a_\mu$ be as in
\eqref{def:a.mu}. Note that the alternating form induced on $L(g)$ has
Gram matrix $g^t J g$. It has alternating type $\mu$ if and only if
\[
  k^tg^tJgk=J_\mu
  \qquad\text{for some }k\in K,
\]
equivalently, $g^t \in K a_\mu H$, or
$g \in H a_\mu K$.

Taking $\nu=\mathbf 0_n$ and replacing \(\mu\) by \(-\mu\) in
\eqref{eq:b-lambda-mu-nu-cosets}, we obtain
\[
  b_{\lambda,-\mu}^{\mathbf 0_n}
  =
  \bigl|\,(K\pi^\lambda K\cap H a_\mu K)/K\,\bigr|.
\]
The right-hand side counts precisely those lattices $\Lambda\subseteq
L$ with type $\lambda$ and alternating type~$\mu$.  By definition,
this is $N_\lri(\lambda,\mu)$, and hence
\[
  \boxed{N_\lri(\lambda,\mu)=b_{\lambda,-\mu}^{\mathbf 0_n}.}
\]

\subsubsection*{Specialization to the intrinsic count \(N'_\lri(\lambda,\mu)\)}

By \cite[Thm.~3.3]{SV25}, the Hecke module coefficients admit the intrinsic
description
\[
  b_{\lambda,\nu}^{\mu}
  =
  \#\bigl\{
    N\subseteq M_\mu \mid
    N\supseteq N^\perp,\
    N/N^\perp\simeq M_\nu,\
    M_\mu/N\simeq\lri_\lambda
  \bigr\}.
\]

In the special case \(\nu=\mathbf 0_n\), we have \(N=N^\perp\), so \(N\) is
Lagrangian. Moreover, the perfect alternating pairing identifies \(M_\mu/N\) with the
dual of \(N^\perp\), hence with a module isomorphic to \(N\). Thus the
condition \(M_\mu/N\simeq\lri_\lambda\) is equivalent to
\(N\simeq\lri_\lambda\), and hence
\[
  \boxed{N'_\lri(\lambda,\mu)=b_{\lambda,\mathbf 0_n}^{\mu}.}
\]

\subsubsection*{Polynomiality of structure constants}

We record the uniform polynomial dependence of all Hecke-theoretic structure
constants introduced above.

\begin{prop}\label{prop:hecke-poly}
  For all $\lambda,\eta,\theta\in\Z^{2n}$ and $\mu,\nu\in\Z^n$, the
  structure constants $g_{\lambda,\eta}^{\ \theta}(\mathfrak o)$ and
  $b_{\lambda,\mu}^{\ \nu}(\mathfrak o)$ depend on $\mathfrak o$ only
  through the residue field cardinality $q$. More precisely, there
  exist polynomials $g_{\lambda,\eta}^{\ \theta}(T),\,
  b_{\lambda,\mu}^{\ \nu}(T)\in\Z[T]$ such that
  \[
    g_{\lambda,\eta}^{\ \theta}(\mathfrak o)=g_{\lambda,\eta}^{\ \theta}(q),
    \qquad
    b_{\lambda,\mu}^{\ \nu}(\mathfrak o)=b_{\lambda,\mu}^{\ \nu}(q).
  \]
\end{prop}

\begin{proof}
  For \(\lambda,\eta,\theta\in\mcP_{2n}\), the description in
  \Cref{subsec:hecke-hall} identifies
  \(g_{\lambda,\eta}^{\ \theta}(\mathfrak o)\) with a Hall number. By
  \cite[\S II.(4.1)]{Macdonald}, these Hall numbers are given by
  integral polynomials in~\(q\). The general case
  \(\lambda,\eta,\theta\in\Z^{2n}\) follows by shifting indices by a
  sufficiently large multiple of \(\mathbf 1_{2n}\).  Polynomiality of
  \(b_{\lambda,\mu}^{\ \nu}(\mathfrak o)\) in $q$ was proved
  in~\cite[Cor.~3.4]{SV24}.
\end{proof}

\begin{cor}\label{cor:N-poly}
  For all $\lambda\in\mcP_{2n}$ and $\mu\in\mcP_n$, the lattice counts
  $N_\lri(\lambda,\mu)$ and $N'_\lri(\lambda,\mu)$ depend polynomially
  on the residue field cardinality $q$: there exist
  polynomials
  \[
    N(\lambda,\mu;T),\; N'(\lambda,\mu;T)\in\Z[T]
  \]
  such that for every compact discrete valuation ring $\mathfrak o$ with residue
  field of size $q$,
  \[
    N_{\lri}(\lambda,\mu)=N(\lambda,\mu;q),
    \qquad
    N'_{\lri}(\lambda,\mu)=N'(\lambda,\mu;q).
  \]
\end{cor}

\begin{proof}
  Immediate from \Cref{prop:hecke-poly} as, by the discussion above, we have
  \[
    N_{\lri}(\lambda,\mu)=b_{\lambda,-\mu}^{\mathbf 0_n}(\mathfrak o),
    \qquad
    N'_{\lri}(\lambda,\mu)=b_{\lambda,\mathbf 0_n}^{\mu}(\mathfrak o).\qedhere
  \]
\end{proof}

\subsection{Satake transform for the spherical Hecke algebra}
\label{subsec:satake-algebra}
We first recall the Satake transform for the spherical Hecke algebra
$\mathcal H(G,K)$, following~\cite[Chap.~V, \S3]{Macdonald}, together
with the necessary background on Hall--Littlewood polynomials from
\cite[Chap.~III, \S3]{Macdonald}.

\subsubsection*{Hall--Littlewood polynomials}
For $\lambda=(\lambda_1,\dots,\lambda_{2n})\in\Z^{2n}_{\mathrm{dom}}$, let
$P_\lambda(x;t)$ denote the Hall--Littlewood polynomial with parameter $t$,
defined by
\[
  P_\lambda(x;t)
  =
  \frac{1}{v_\lambda(t)}
  \sum_{w\in S_{2n}}
  w\!\left(
    x_1^{\lambda_1}\cdots x_{2n}^{\lambda_{2n}}
    \prod_{i<j}\frac{x_i-tx_j}{x_i-x_j}
  \right),
  \qquad
  v_\lambda(t)=
  \prod_{r\in \Z}\prod_{j=1}^{m_r(\lambda)}
  \frac{1-t^j}{1-t},
\]
where $m_r(\lambda)=\#\{i \mid \lambda_i=r\}$.
The polynomial $P_\lambda(x;t)$ is symmetric in $x_1,\dots,x_{2n}$.
In what follows we specialize to $t=q^{-1}$.

\subsubsection*{The Satake transform}
For $\lambda\in\Z^{2n}_{\mathrm{dom}}$, set
\begin{equation}\label{eq:n-lambda}
  n(\lambda):=\sum_{i\ge1}(i-1)\lambda_i .
\end{equation}

The \emph{Satake transform}---denoted $\theta$ in \cite[Chap.~V,
  \S2]{Macdonald}---identifies the spherical Hecke algebra $\mathcal
H(G,K)$ with the algebra of symmetric Laurent polynomials by
\begin{align}
  \Sat\colon \mathcal H(G,K)
  &\xrightarrow{\;\sim\;}  \C[x_1^{\pm1},\ldots,x_{2n}^{\pm1}]^{S_{2n}},\nonumber\\
  c_\lambda &\mapsto
  q^{-n(\lambda)}\,P_\lambda(x;\,q^{-1})\label{eq:satake-def}.
\end{align}
Under this identification, the subalgebra
$\mathcal H(G^{+},K)\subseteq\mathcal H(G,K)$ corresponds to the subalgebra
spanned by $\{P_\lambda(x;\,q^{-1}) \mid \lambda\in\mcP_{2n}\}$.

\noindent
\emph{Hecke multiplication and Hall numbers.} Under the Satake
transform, the convolution product in $\mathcal H(G,K)$ corresponds to
multiplication of Hall--Littlewood polynomials:
\begin{equation*}
  q^{-n(\lambda)}P_\lambda(x;\,q^{-1})\,
  q^{-n(\eta)}P_\eta(x;\,q^{-1})
  =
  \sum_{\theta}
  g_{\lambda,\eta}^{\ \theta}(q)\,
  q^{-n(\theta)}P_\theta(x;\,q^{-1}).
\end{equation*}
Consequently, the structure constants $g_{\lambda,\eta}^{\ \theta}(q)$ coincide
with the Hall numbers of finite $\lri$-modules in the sense of
\cite[Chap.~II]{Macdonald}.
In particular, they are polynomials in~$q$, vanish unless
$|\theta|=|\lambda|+|\eta|$ because $\deg(P_\lambda)=|\lambda|$, and satisfy
the symmetry
\(
  g_{\lambda,\eta}^{\ \theta}
  =
  g_{\eta,\lambda}^{\ \theta}.
\)

\begin{remark}\label{rem:satake-vs-fourier}
  In \cite[Chap.~V, \S3]{Macdonald}, the \emph{spherical Fourier transform}
  is defined via spherical functions $\omega_s$.
  More precisely, for $\lambda\in\Z^{2n}_{\mathrm{dom}}$ we have
  \[
    \widehat{c}_\lambda(\omega_s)
    =
    q^{\rho_{2n}\cdot\lambda}\,
    P_\lambda\!\bigl(q^{-s_1},\dots,q^{-s_{2n}};\,q^{-1}\bigr),
  \]
  where $\rho_{2n}=\tfrac12(2n-1,2n-3,\ldots,1-2n)$. Equivalently,
  $\widehat{c}_\lambda(\omega_s)$ is obtained from $\Sat(c_\lambda)$
  by the specialization $x_i \mapsto q^{\frac{2n-1}{2}-s_i}$ for
  $i\in[2n]$.

  While some authors use the term ``Satake transform'' for the map
  $c_\lambda\mapsto\widehat{c}_\lambda$, we reserve this terminology
  for the algebra isomorphism \eqref{eq:satake-def}, viewing the
  spherical Fourier transform as its specialization.
\end{remark}

\subsection{Alternating Satake transform and Hecke module coefficients}
\label{subsec:alt-satake}
We describe the action of the spherical Hecke algebra $\mathcal
H(G^{+},K)$ on the Hecke module $\mathcal M=\C_c(K\backslash G/H)$ in
terms of Hall--Littlewood polynomials, via a specialization of the
Satake transform recalled in \Cref{subsec:satake-algebra}, following
\cite{HironakaSato} and \cite{SV25}.

\subsubsection*{Specialized Satake transform}
While the Satake transform identifies the Hecke algebra $\mathcal
H(G,K)$ with $\C[x_1^{\pm1},\ldots,x_{2n}^{\pm1}]^{S_{2n}}$ via
$\Sat(c_\lambda)=q^{-n(\lambda)}P_\lambda(x;\,q^{-1})$, we now
consider the algebra homomorphism
\begin{align*}
  \Sat' \colon \mathcal H(G,K)
  &\rightarrow
  \C[x_1^{\pm1},\ldots,x_n^{\pm1}]^{S_n},\\
  c_\lambda &\mapsto
  q^{|\lambda|-n(\lambda)}\,P_\lambda(x,x/q;\,q^{-1}),
\end{align*}
where $x,x/q$ denotes the specialization
$(x_1,x_1/q,\;x_2,x_2/q,\;\ldots,\;x_n,x_n/q)$.

\subsubsection*{Alternating Satake transform}
The Hecke module $\mathcal M$ has $\C$-basis
$\{\phi_\mu\}_{\mu\in\Z^{n}_{\mathrm{dom}}}$ indexed by $K$-orbits in
$G/H$. Define a linear map
\begin{align}
  \Sat^{\mathrm{alt}}\colon \mathcal{M}
  &\rightarrow
  \C[x_1^{\pm1},\ldots,x_n^{\pm1}]^{S_n},\nonumber\\
  \phi_\mu &\mapsto
  q^{-2n(\mu)}\,P_\mu(x;\,q^{-2})\label{eq:alt-satake-def}.
\end{align}
This map is a linear isomorphism.

\subsubsection*{Compatibility with the Hecke action}
By \cite[Thm.~3.3]{SV24} and \cite[Thm.~3.5]{SV25}, the transforms
$\Sat'$ and $\Sat^{\mathrm{alt}}$ are compatible with the Hecke
action: for $f\in\mathcal H(G,K)$ and $\varphi\in\mathcal M$,
\begin{equation}\label{eq:intertwining}
  \Sat^{\mathrm{alt}}(f\star\varphi)
  =
  \Sat'(f)\cdot \Sat^{\mathrm{alt}}(\varphi).
\end{equation}

Applying \(\Sat^{\mathrm{alt}}\) to \eqref{def:hecke.act} and using
\eqref{eq:intertwining}, we obtain \begin{equation*}
  \bigl(q^{|\lambda|-n(\lambda)}P_\lambda(x,x/q;\,q^{-1})\bigr)
  \cdot
  \bigl(q^{-2n(\mu)}P_\mu(x;\,q^{-2})\bigr)
  =
  \sum_{\nu\in\Z^{n}_{\mathrm{dom}}}
  b_{\lambda,\mu}^{\ \nu}(q)\,
  q^{-2n(\nu)}P_\nu(x;\,q^{-2}),
\end{equation*}
which characterizes the Hecke module coefficients~$b_{\lambda,\mu}^{\ \nu}(q)$.

\subsection{Proof of \texorpdfstring{\Cref{prop:fact}}{Proposition~\ref{prop:fact}}}
\label{subsec:factorization}

Recall from \Cref{subsec:hecke-to-counting} that
\(N_\lri(\lambda,\mu)=b_{\lambda,-\mu}^{\mathbf 0_n}(q)\) and
\(N'_\lri(\lambda,\mu)=b_{\lambda,\mathbf 0_n}^{\mu}(q)\). Thus
\Cref{prop:fact} is equivalent to
\begin{equation}\label{eq:b-factorization-target}
  b_{\lambda,-\mu}^{\mathbf 0_n}(q)
  =
  b_{\lambda,\mathbf 0_n}^{\mu}(q)\,\alpha_n(\mu;q^2).
\end{equation}

By the definition \eqref{def:hecke.act} of the Hecke action,
\(c_\lambda\star\phi_{-\mu}
=\sum_{\rho}b_{\lambda,-\mu}^{\ \rho}(q)\phi_\rho\). Applying
\(\Sat^{\mathrm{alt}}\) and using \eqref{eq:intertwining} gives
\begin{equation}\label{eq:satake-step1}
  \Sat'(c_\lambda)\,\Sat^{\mathrm{alt}}(\phi_{-\mu})
  =\sum_{\rho}b_{\lambda,-\mu}^{\ \rho}(q)\,\Sat^{\mathrm{alt}}(\phi_\rho).
\end{equation}
Similarly, applying \(\Sat^{\mathrm{alt}}\) to
$c_\lambda\star\phi_{\mathbf 0_n}$ gives
\begin{equation}\label{eq:satake-step2}
  \Sat'(c_\lambda)=\sum_{\nu}b_{\lambda,\mathbf 0_n}^{\ \nu}(q)\,\Sat^{\mathrm{alt}}(\phi_\nu),
\end{equation}
as \(\Sat^{\mathrm{alt}}(\phi_{\mathbf 0_n})=1\). Substitution into
\eqref{eq:satake-step1} gives
\begin{equation}\label{eq:satake-step3}
  \sum_{\nu}b_{\lambda,\mathbf 0_n}^{\ \nu}(q)\,
  \Sat^{\mathrm{alt}}(\phi_\nu)\,\Sat^{\mathrm{alt}}(\phi_{-\mu})
  =\sum_{\rho}b_{\lambda,-\mu}^{\ \rho}(q)\,\Sat^{\mathrm{alt}}(\phi_\rho).
\end{equation}

For the spherical Hecke algebra of \(\GL_n\), the Satake isomorphism gives
\(\Sat_n(c_\nu)\Sat_n(c_{-\mu})=\sum_\rho g_{\nu,-\mu}^{\ \rho}(q)\Sat_n(c_\rho)\).
Since, by \eqref{eq:satake-def} and \eqref{eq:alt-satake-def},
\(\Sat^{\mathrm{alt}}(\phi_\nu)=\Sat_n(c_\nu)|_{q\mapsto q^2}\), it follows that
\[
  \Sat^{\mathrm{alt}}(\phi_\nu)\,\Sat^{\mathrm{alt}}(\phi_{-\mu})
  =\sum_{\rho}g_{\nu,-\mu}^{\ \rho}(q^2)\,\Sat^{\mathrm{alt}}(\phi_\rho).
\]
Substituting this into \eqref{eq:satake-step3} and comparing coefficients in
the basis $\{\Sat^{\mathrm{alt}}(\phi_\rho)\}$ yields
\begin{equation}\label{eq:key-identity}
  b_{\lambda,-\mu}^{\ \mathbf 0_n}(q)
  =
  \sum_{\nu}
  b_{\lambda,\mathbf 0_n}^{\ \nu}(q)\,
  g_{\nu,-\mu}^{\mathbf 0_n}(q^2).
\end{equation}

It remains to evaluate the Hall numbers
$g_{\nu,-\mu}^{\mathbf 0_n}(q^2)$.
Let $\lri'$ be a complete DVR with residue field of size $q^2$ and
uniformizer $\pi$.
For $k\gg0$ (e.g.\ $k=\mu_1$) and $\mathbf 1_n=(1,\dots,1)$, stability of Hall
numbers gives
\[
  g_{\nu,-\mu}^{\mathbf 0_n}(q^2)
  =
  g_{\nu,\;k\mathbf 1_n-\mu}^{\;k\mathbf 1_n}(q^2).
\]
By the geometric interpretation of Hall numbers (\Cref{subsec:hecke-hall}),
this counts sublattices $\pi^k(\lri')^n \subseteq L' \subseteq (\lri')^n$ with
\[
  L'/\pi^k(\lri')^n \simeq \lri'_{k\mathbf 1_n-\mu},
  \qquad
  (\lri')^n/L' \simeq \lri'_\nu .
\]
The type of a submodule of $(\lri'/\pi^k)^n$ uniquely determines the
type of its quotient; hence such $L'$ exist only if $\nu=\mu$.  In
this case, the number of choices equals the number of sublattices of
$(\lri')^n$ of quotient type $\mu$, namely the Birkhoff number
$\alpha_n(\mu;q^2)$.  Therefore
\[
  g_{\nu,-\mu}^{\mathbf 0_n}(q^2)
  =
  \begin{cases}
    \alpha_n(\mu;q^2), & \textup{ if }\nu=\mu,\\
    0, & \textup{ if }\nu\neq\mu.
  \end{cases}
\]

Substituting this into \eqref{eq:key-identity} yields
\eqref{eq:b-factorization-target}, completing the proof of \Cref{prop:fact}.

\section{Counting lattices by symplectic type}\label{sec:counting-symp}

We now aggregate the counting functions \(N'_\lri(\lambda,\mu)\) and
\(N_\lri(\lambda,\mu)\) over all \(\lambda\), obtaining counts
\(N'_\lri(\mu)\) and \(N_\lri(\mu)\), depending only on the symplectic
type~\(\mu\).  The count \(N_\lri(\mu)\) will be applied to the
subalgebra zeta functions of higher Heisenberg algebras in
\Cref{sec:heis}.

\subsection{Aggregated counts by symplectic type}
\begin{definition}\label{def:sym-comp}
  For a composition \(\mu=(\mu_1,\dots,\mu_n)\in\N_0^n\), define
  \[
    N'_\lri(\mu)
    :=
    \#\{\,N\subseteq M_\mu \mid N=N^\perp\,\}.
  \]
\end{definition}

Thus \(N'_\lri(\mu)\) counts all Lagrangian submodules of the alternating
\(\lri\)-module \(M_\mu\), independently of their \(\lri\)-module types.  For
\(\mu\in\mcP_n\), \Cref{def:int} gives
\[
  N'_\lri(\mu)
  =
  \sum_{\lambda\in\mcP_{2n}} N'_\lri(\lambda,\mu),
\]
since every Lagrangian submodule of \(M_\mu\) has a well-defined
\(\lri\)-module type.

\begin{definition}\label{def:sym}
  For a partition \(\mu\in\mcP_n\), define
  \[
    N_\lri(\mu)
    :=
    \sum_{\lambda\in\mcP_{2n}} N_\lri(\lambda,\mu).
  \]
\end{definition}

Thus \(N_\lri(\mu)\) counts all finite-index sublattices \(\Lambda\leq L\)
of alternating type~\(\mu\).

\begin{remark}\label{rem:sym-poly}
  By \Cref{cor:N-poly}, the quantities \(N'_\lri(\mu)\) and
  \(N_\lri(\mu)\) are polynomials in~\(q\).  We write the corresponding
  polynomials as \(N'_\lri(\mu;t),\; N_\lri(\mu;t)\in\Z[t]\).
\end{remark}

The following is an immediate consequence of~\Cref{prop:fact}.

\begin{cor}\label{cor:N.N'}
  \[
    N_\lri(\mu)
    =
    N'_\lri(\mu)\,\alpha_n(\mu;q^2).
  \]
\end{cor}
In the remainder of this section, we compute \(N'_\lri(\mu)\).

\subsection{Properties of \texorpdfstring{$N'_\lri(\mu)$}{N'_\lri(mu)}}\label{subsec:property-N'}

\begin{prop}\label{prop:N}
The function \(N'_\lri\) satisfies the following properties:
\begin{enumerate}[label=\textup{(N\arabic*)},leftmargin=*]
  \setcounter{enumi}{-1}
  \item \(N'_\lri(\varnothing)=1\).
  \item \(N'_\lri(\mu)=N'_\lri(\mu,0)\) for all \(\mu\in\N_0^n\).
  \item \(N'_\lri(\mu)=N'_\lri(\mu')\) for any permutation \(\mu'\) of \(\mu\).
  \item If \(\mu=(\mu_1,\dots,\mu_n)\in\mcP_n\) and \(\mu_1\ge1\), then
  \[
    N'_\lri(\mu)
    =
    N'_\lri(\mu_1-1,\mu_2,\dots,\mu_n)
    +
    q^{|\mu|}N'_\lri(\mu_2,\dots,\mu_n).
  \]
\end{enumerate}
\end{prop}

\begin{proof} \textup{(N0)} follows from \(M_\varnothing=0\), \textup{(N1)} is
immediate from \Cref{def:sym-comp}, and \textup{(N2)} follows from
\(M_\mu\simeq M_{\mu'}\) as alternating \(\lri\)-modules. It remains
to prove \textup{(N3)}.

Write \(M_\mu=\bigoplus_{i=1}^n (U_i\oplus V_i)\), where \(U_i\cong
V_i\cong\lri/\pi^{\mu_i}\lri\), the summands \(U_i\oplus V_i\) are
mutually orthogonal, and the alternating pairing induces a perfect
pairing
  \[
    (\ ,\ ) \colon U_i \times V_i \longrightarrow \pi^{-\mu_i}\lri/\lri
  \]
  for each \(i\), while \((U_i,U_j)=(V_i,V_j)=0\) for all \(i,j\).
  Define
  \[
    M'=(\pi U_1)\oplus V_1\oplus U_2\oplus V_2\oplus\cdots\oplus U_n\oplus V_n.
  \]
  Then \(M'/M'^\perp\simeq M_{(\mu_1-1,\mu_2,\dots,\mu_n)}\).
  Partition the set of Lagrangians \(N\subseteq M_\mu\) into
  \[
    \mathcal{A}=\{\,N \mid N\subseteq M'\,\},
    \qquad
    \mathcal{B}=\{\,N \mid N\not\subseteq M'\,\}.
  \]

  If \(N\subseteq M'\), then \(M'^\perp\subseteq N\), and \(N/M'^\perp\) is
  Lagrangian in \(M'/M'^\perp\simeq
  M_{(\mu_1-1,\mu_2,\dots,\mu_n)}\). Hence
  \(|\mathcal{A}|=N'_\lri(\mu_1-1,\mu_2,\dots,\mu_n)\).

  If \(N\not\subseteq M'\), then \(N\) contains a vector
  \(v\in M_\mu\setminus M'\).  By extending \(v\) to a symplectic basis and
  conjugating by an isometry, we may assume \(v=e_1\in U_1\).  Therefore
  \(\langle v\rangle^\perp/\langle v\rangle\cong M_{(\mu_2,\dots,\mu_n)}\).
  Thus the number of Lagrangians \(N\) containing any fixed such \(v\) is
  \(N'_\lri(\mu_2,\dots,\mu_n)\).

  The number of admissible vectors is \(|M_\mu|-|M'|=q^{2|\mu|}(1-q^{-1})\),
  and each \(N\in\mathcal{B}\) contains exactly
  \(|N|-|N\cap M'|=q^{|\mu|}(1-q^{-1})\) such vectors.  Double counting the
  pairs \((v,N)\) with \(v\in N\setminus M'\) gives
  \[
    |\mathcal{B}|\,q^{|\mu|}(1-q^{-1})
    =
    q^{2|\mu|}(1-q^{-1})\,N'_\lri(\mu_2,\dots,\mu_n),
  \]
  hence \(|\mathcal{B}|=q^{|\mu|}\,N'_\lri(\mu_2,\dots,\mu_n)\). Finally,
  \[
    N'_\lri(\mu)
    =
    |\mathcal{A}|+|\mathcal{B}|
    =
    N'_\lri(\mu_1-1,\mu_2,\dots,\mu_n)
    +
    q^{|\mu|}\,N'_\lri(\mu_2,\dots,\mu_n). \qedhere
  \]
  \end{proof}

\begin{remark}\label{rem:permanence}
  By \Cref{rem:sym-poly}, write \(N'_\lri(\mu)=N'(\mu;q)\) with
  \(N'(\mu;t)\in\Z[t]\). Since \textup{(N0)--(N3)} hold for all prime powers
  \(q\), they also hold as polynomial identities in \(t\).
\end{remark}

\subsection{A closed formula for \texorpdfstring{$N'_\lri(\mu)$}{N'_\lri(mu)}}
The properties \textup{(N0)--(N3)} in \Cref{prop:N} uniquely determine
\(N'_\lri(\mu)\).  We record a closed formula for this quantity.

\begin{definition}\label{def:Wn}
  For \(n\in\N\), set
  \begin{equation}\label{def:C}
    \mathcal{W}_n :=\bigl\{
    (w_1,\dots,w_n)\in\N_0^n \mid w_0=0,\
    w_i\in\{\,w_{i-1},\,2i-1-w_{i-1}\,\}
    \text{ for } i\in[n] \bigr\}.
  \end{equation}

  For \(\bfw \in \mcW_n\), set
  \[
    C_n(\bfw) =\prod_{i=1}^n \frac{1}{1-q^{2i-1-2w_i}}.
  \]
\end{definition}

Thus \(\mcW_n\) is obtained from \(\mcW_{n-1}\) by appending
\(w_n\in\{w_{n-1},2n-1-w_{n-1}\}\), and hence \(|\mcW_n|=2^n\). For example,
\(\mcW_1=\{(0),(1)\}\),
\(\mcW_2=\{(0,0),(0,3),(1,1),(1,2)\}\), and
\[
  \mathcal{W}_3=\{(0,0,0),(0,0,5),(0,3,2),(0,3,3),
  (1,1,1),(1,1,4),(1,2,2),(1,2,3)\}.
\]

\begin{prop}[Closed formula for \(N'_\lri(\mu)\)]\label{prop:closed.N'}
  Let \(\mu\in\mcP_n\) be a partition. Then
  \begin{equation}\label{eq:Nprime-closed}
    N'_\lri(\mu)
    =
    \sum_{\bfw\in\mathcal{W}_n}
    C_n(\bfw)\,q^{\bfw\cdot\mu}.
  \end{equation}
\end{prop}

\begin{remark}\label{rem:poly-despite-form}
  Its presentation as a sum of rational functions notwithstanding, the
  right-hand side of \eqref{eq:Nprime-closed} is a polynomial
  \(N'(\mu;q)\) in~\(q\), by \Cref{rem:sym-poly}. Note further that
  \[
    N'(\mu;1)=\prod_{i=1}^n(\mu_i+1).
  \]
  This follows from \textup{(N0)--(N3)}. Indeed, for \(t=1\) the recursion
  becomes
  \(N'(\mu;1)=N'(\mu_1-1,\mu_2,\dots,\mu_n;1)+N'(\mu_2,\dots,\mu_n;1)\).
\end{remark}

\begin{proof}
  Set \(\mcW_0=\{\varnothing\}\), \(C_0(\varnothing)=1\), and, for \(n\ge0\),
  \[
    F_n(\bfw;\mathbf{v}):=C_n(\bfw)q^{\bfw\cdot\mathbf{v}},
    \qquad
    P_n(\mathbf{v}):=\sum_{\bfw\in\mcW_n}F_n(\bfw;\mathbf{v}).
  \]
  Thus \(P_0(\varnothing)=1\).  We prove \(P_n(\mu)=N'_\lri(\mu)\) by showing
  that \(P_n\) satisfies the following analogues of \textup{(N0)--(N3)}.  Note
  that \textup{(P2)} is weaker than \textup{(N2)}.

  \begin{enumerate}[label=\textup{(P\arabic*)},leftmargin=*]
  \item For \(n\ge0\) and \(\mathbf{v}\in\N_0^n\),
  \(P_{n+1}(\mathbf{v},0)=P_n(\mathbf{v})\).
  \item If \(\mathbf{v}'\) is obtained from \(\mathbf{v}\) by swapping adjacent
  entries \(v_i,v_{i+1}\) with \(v_i=v_{i+1}\pm1\), then
  \(P_n(\mathbf{v})=P_n(\mathbf{v}')\).
  \item For \(n\ge1\) and \(v_1\ge1\),
  \[
    P_n(v_1,\dots,v_n)-P_n(v_1-1,v_2,\dots,v_n)
    =
    q^{v_1+\cdots+v_n}\,P_{n-1}(v_2,\dots,v_n).
  \]
  \end{enumerate}
  These properties suffice by strong induction on \(|\mu|\), uniformly in
  \(n\). If \(|\mu|=0\), then \(\mu=(0,\dots,0)\). By \textup{(N0), (N1)} and
  \textup{(P1)}, \(N'_\lri(\mu)=P_n(\mu)=1\).

  Now assume \(|\mu|>0\), and assume the claim known for all partitions of
  smaller size. Then \(\mu_1\ge1\). Let \(\mu^\downarrow\) be the partition
  obtained by sorting \((\mu_1-1,\mu_2,\dots,\mu_n)\) into weakly decreasing
  order. This sorting uses only adjacent swaps of entries differing by \(1\);
  hence \textup{(P2)} gives
  \(
    P_n(\mu_1-1,\mu_2,\dots,\mu_n)=P_n(\mu^\downarrow).
  \)
  Applying \textup{(P3)} and \textup{(N2), (N3)}, we obtain
  \[
    P_n(\mu)=P_n(\mu^\downarrow)+q^{|\mu|}P_{n-1}(\mu_2,\dots,\mu_n),
    \qquad
    N'_\lri(\mu)=N'_\lri(\mu^\downarrow)
    +q^{|\mu|}N'_\lri(\mu_2,\dots,\mu_n).
  \]
  The right-hand sides coincide by induction hypothesis. Therefore
  \(P_n(\mu)=N'_\lri(\mu)\). It remains to prove \textup{(P1)--(P3)}.

  \emph{Proof of \textup{(P1)}.}
  The case \(n=0\) is immediate, since \(P_1(0)=1=P_0(\varnothing)\).
  Assume \(n\ge1\) and fix \(\bfw\in\mathcal{W}_n\).
  There are exactly two extensions to \(\mathcal{W}_{n+1}\), namely
  \[
    \bfw^{(0)}=(\bfw,w_n),
    \qquad
    \bfw^{(1)}=(\bfw,2n+1-w_n).
  \]
  Writing \(x:=2n+1-2w_n\), we get
    \begin{equation*}
      F_{n+1}(\bfw^{(0)};(\mathbf{v},0))
      +F_{n+1}(\bfw^{(1)};(\mathbf{v},0)) = F_n(\bfw;\mathbf{v})
      \left(\frac{1}{1-q^{x}}+\frac{1}{1-q^{-x}}\right)
      =F_n(\bfw;\mathbf{v}).
    \end{equation*}  
      Summing over all \(\bfw\in\mathcal{W}_n\) yields
  \(P_{n+1}(\mathbf{v},0)=P_n(\mathbf{v})\).
  
  \emph{Proof of \textup{(P3)}.}
  Let \(n\ge1\) and \(\mathbf{v}=(v_1,\dots,v_n)\in\N_0^n\) with \(v_1\ge1\).
  By definition,
  \[
    P_n(\mathbf{v})-P_n(\mathbf{v}-\mathbf{e}_1)
    =
    \sum_{\bfw\in\mathcal{W}_n}F_n(\bfw;\mathbf{v})(1-q^{-w_1}),
  \]
  where \(\mathbf{e}_1=(1,0,\dots,0)\). Since \(w_1\in\{0,1\}\), only terms
  with \(w_1=1\) contribute. Let
  \(\mathcal{W}_n^{(1)}=\{\bfw\in\mathcal{W}_n:w_1=1\}\).
  For \(\bfw=(1,w_2,\dots,w_n)\in\mathcal{W}_n^{(1)}\), define
  \(\bfw'=(w_2-1,\dots,w_n-1)\); this gives a bijection
  \(\mathcal{W}_n^{(1)}\simeq\mathcal{W}_{n-1}\).

  Writing \(\mathbf{v}'=(v_2,\dots,v_n)\), we have
  \(\bfw\cdot\mathbf{v}=|\mathbf{v}|+\bfw'\cdot\mathbf{v}'\) and
  \[
    \begin{aligned}
      F_n(\bfw;\mathbf{v})
      &=C_n(\bfw)\,q^{\bfw\cdot\mathbf{v}} \\
      &=\frac{C_{n-1}(\bfw')}{1-q^{-1}}
        q^{|\mathbf{v}|+\bfw'\cdot\mathbf{v}'}
      =
      \frac{q^{|\mathbf{v}|}}{1-q^{-1}}\,
      F_{n-1}(\bfw';\mathbf{v}').
    \end{aligned}
  \]
  Therefore
  \(F_n(\bfw;\mathbf{v})(1-q^{-1})=q^{|\mathbf{v}|}F_{n-1}(\bfw';\mathbf{v}')\).
  Summing over \(\mathcal{W}_n^{(1)}\) gives
  \[
    P_n(\mathbf{v})-P_n(\mathbf{v}-\mathbf{e}_1)
    =q^{|\mathbf{v}|}P_{n-1}(\mathbf{v}').
  \]

  \emph{Proof of \textup{(P2)}.}
  It is enough to prove that, if \(v_i=v_{i+1}\), then
  \(P_n(\mathbf{v}+\mathbf e_i)=P_n(\mathbf{v}+\mathbf e_{i+1})\).

  Recall that \(F_n(\bfw;\mathbf{v})=C_n(\bfw)\,q^{\bfw\cdot\mathbf{v}}\).
  Therefore
  \[
    P_n(\mathbf{v}+\mathbf{e}_i)-P_n(\mathbf{v}+\mathbf{e}_{i+1})
    =
    \sum_{\bfw\in\mathcal{W}_n}
    C_n(\bfw)(q^{w_i}-q^{w_{i+1}})q^{\bfw\cdot\mathbf{v}}.
  \]
  If \(w_i=w_{i+1}\), the summand vanishes. Thus only
  \(\bfw\in\mathcal{W}_n^{\ne}:=\{\bfw:w_i\ne w_{i+1}\}\) contribute.
  For such \(\bfw\), write \((w_i,w_{i+1})=(y,2i+1-y)\), and define
  \[
    \tau_i(\dots,y,2i+1-y,\dots)
    =
    (\dots,2i-1-y,y+2,\dots).
  \]
  This is an involution on the contributing set; the possible previous and
  next entries are unchanged. It also preserves \(w_i+w_{i+1}\), so
  \(q^{\bfw\cdot\mathbf{v}}=q^{\tau_i(\bfw)\cdot\mathbf{v}}\).

  Since \(C_n\) is multiplicative in \(j\) and \(\tau_i\) only affects the
  \(i\)th and \((i+1)\)st factors, a direct computation with
  \(t:=2y-2i+1\) gives
  \[
    \begin{aligned}
      &q^{\bfw\cdot\mathbf{v}}
      \Bigl(
      C_n(\bfw)\,(q^{w_i}-q^{w_{i+1}})
      +
      C_n(\tau_i\bfw)\,(q^{(\tau_i\bfw)_i}-q^{(\tau_i\bfw)_{i+1}})
      \Bigr)\\
      &\qquad
      =C_{\mathrm{rest}}\,q^{\bfw\cdot\mathbf{v}}
      \Biggl(
        \frac{q^{y}(1-q^{2-t})}{(1-q^{-t})(1-q^{t-2})}
        +
        \frac{q^{y}(q^{-t}-q^{2})}{(1-q^{t})(1-q^{-t-2})}
      \Biggr)
      =0.
    \end{aligned}
  \]
  Thus the sum cancels in \(\tau_i\)-pairs and the difference vanishes.
  
  \medskip
  This completes the proof that \(P_n(\mu)=N'_\lri(\mu)\) for all partitions
  \(\mu\in\mcP_n\).
\end{proof}

\subsection{Examples}\label{subsec:Nprime-examples}
We conclude by recording several examples. For \(a\ge b\ge c\ge 0\), we have
\[
  N'_\lri(a)=1+q+q^2+\cdots+q^a
  =\frac{q^{a+1}-1}{q-1}
  =\frac{1}{1-q}+\frac{q^a}{1-q^{-1}}.
\]

\[
  \begin{aligned}
    N'_\lri(a,b)
    &=\frac{1}{(1-q)(1-q^{3})}
    +\frac{q^{3b}}{(1-q)(1-q^{-3})}
    +\frac{q^{a+b}}{(1-q^{-1})(1-q)}
    +\frac{q^{a+2b}}{(1-q^{-1})^2}\\
    &=\frac{1-q^{1+a+b}(1+q+q^2)+q^{2+a+2b}(1+q+q^2)-q^{3+3b}}
    {(1-q)(1-q^3)}.
  \end{aligned}
\]
\[
  \begin{gathered}
    N'_\lri(a,b,c)
    =
    \sum_{(w_1,w_2,w_3)\in\mathcal{W}_3}
    \frac{q^{w_1 a+w_2 b+w_3 c}}
    {(1-q^{1-2w_1})(1-q^{3-2w_2})(1-q^{5-2w_3})},\\
    \mathcal{W}_3=\{(0,0,0),(0,0,5),(0,3,2),(0,3,3),
    (1,1,1),(1,1,4),(1,2,2),(1,2,3)\}.
  \end{gathered}
\]

\section{Subalgebra zeta functions of higher Heisenberg algebras over cDVRs}
\label{sec:heis}

In this section we connect the Hecke-theoretic ideas developed so far with
subalgebra zeta functions of higher Heisenberg algebras over cDVRs.

\subsection{Higher Heisenberg algebras and symplectic spaces}
\label{subsec:heis-lie-zeta}

Recall that \(\lri\) is a cDVR with residue field cardinality
\(q\). Recall further that \(\mfh_n(\lri)=\mfh_n(\Z)\otimes_\Z\lri\)
is the Heisenberg algebra over \(\lri\) of degree~\(n\). As an
\(\lri\)-module, it is free of rank~\(2n+1\).  To bring the
Hecke-theoretic ideas from the previous sections to bear, set
\(Z_n(\lri):=Z(\mfh_n(\lri))\cong\lri\) for the cyclic centre of
\(\mfh_n(\lri)\), and set
\[
  V_n(\lri):=\mfh_n(\lri)/Z_n(\lri)\cong \lri^{2n} \]
for the commutator quotient. The Lie bracket on \(\mfh_n(\lri)\) induces on
\(V_n(\lri)\) the structure of a non-degenerate symplectic space via the form
\[
  \omega_n: V_n(\lri)\times V_n(\lri)\rarr Z_n(\lri),
  \qquad
  (\overline g,\overline h)\mapsto [g,h].
\]

Identifying \(V_n(\lri)\) with \(\lri^n\oplus \lri^n\), where both copies of
\(\lri^n\) are viewed as spaces of column vectors, the form is given by
\[
  \omega_n\bigl((\mathbf x_1,\mathbf y_1),(\mathbf x_2,\mathbf y_2)\bigr)
  = \mathbf x_1^T\mathbf y_2-\mathbf x_2^T\mathbf y_1
  \in Z_n(\lri)\cong \lri.
\]

\subsection{Subalgebras of \texorpdfstring{\(\mfh_n(\lri)\)}{hn(o)}}

We express the subalgebra zeta function of \(\mfh_n(\lri)\) in terms
of the invariants \(N'_\lri(\mu)\) (cf.\ \Cref{def:sym-comp}) and
\(\alpha_n(\mu;q^2)\) (cf.\ \eqref{eq:Birk-n-closed}), culminating in
\Cref{prop:zeta.formula}. We write \(\phi:\mfh_n(\lri)\rarr
V_n(\lri)\) for the quotient map modulo~\(Z_n(\lri)\).

Let \(\mathfrak a\le\mfh_n(\lri)\) be an \(\lri\)-subalgebra of finite index.
Setting \(\Lambda:=\phi(\mathfrak a)\) and
\(K:=\mathfrak a\cap Z_n(\lri)\), we obtain an exact sequence of
\(\lri\)-modules
\[
  0\to K\to \mathfrak a\to \Lambda\to 0.
\]
Here \(\Lambda\le V_n(\lri)\) is a finite-index \(\lri\)-lattice,
\(K\le Z_n(\lri)\) is an ideal, and the subalgebra condition is equivalent to
\begin{equation}\label{cond:subalgebra}
  \omega_n(\Lambda,\Lambda)\subseteq K.
\end{equation}

Let \(\mu(\Lambda)\in\mcP_n\) denote the alternating type of
\(\omega_n|_\Lambda\) in the sense of~\Cref{subsec:inv.symp.lat}, with
smallest part \(\mu_n(\Lambda)\in\N_0\).  Then
\(\omega_n(\Lambda,\Lambda)=\pi^{\mu_n(\Lambda)}Z_n(\lri)\).  By the
subalgebra condition~\eqref{cond:subalgebra}, there exists an integer
\(c\in[\mu_n(\Lambda)]_0\) such that \(K=\pi^cZ_n(\lri)\).

Conversely, given a finite-index sublattice \(\Lambda\le V_n(\lri)\) and
\(c\in[\mu_n(\Lambda)]_0\), the number of \(\lri\)-subalgebras
\(\mathfrak a\le\mfh_n(\lri)\) satisfying
\[
  \phi(\mathfrak a)=\Lambda
  \qquad\text{and}\qquad
  \mathfrak a\cap Z_n(\lri)=\pi^cZ_n(\lri)
\]
is \(q^{2nc}\). Indeed, let \(\{\overline e_1,\dots,\overline e_{2n}\}\) be
an \(\lri\)-basis of \(\Lambda\). Each \(\overline e_i\) admits exactly \(q^c\)
lifts to \(\phi^{-1}(\Lambda)\) modulo \(\pi^cZ_n(\lri)\), and these choices
are mutually independent.

The index of any such subalgebra \(\mathfrak a\) is
\[
  [\mfh_n(\lri):\mathfrak a]
  = [Z_n(\lri):\pi^cZ_n(\lri)]\,[V_n(\lri):\Lambda]
  = q^{c+|\mu|}.
\]
Hence, for a fixed finite-index lattice \(\Lambda\leq V_n(\lri)\) of
alternating type~\(\mu\), the contribution to
\(\zeta_{\mfh_n(\lri)}(s)\) from the \(\lri\)-subalgebras \(\mathfrak
a\le\mfh_n(\lri)\) with \(\phi(\mathfrak a)=\Lambda\) is
\[
  \sum_{c=0}^{\mu_n} q^{2nc-s(c+|\mu|)}
  =
  q^{-s|\mu|}
  \frac{1-q^{(\mu_n+1)(2n-s)}}{1-q^{2n-s}}.
\]

Recall \eqref{def:sym} that \(N_\lri(\mu)\) counts the finite-index
sublattices \(\Lambda\le\lri^{2n}\) of alternating type~\(\mu\). Thus
the \(\lri\)-subalgebras \(\mathfrak a\le\mfh_n(\lri)\) whose image
\(\phi(\mathfrak a)\subseteq \lri^{2n}\) has alternating type \(\mu\)
contribute
\[
  N_\lri(\mu)\,q^{-s|\mu|}
  \frac{1-q^{(\mu_n+1)(2n-s)}}{1-q^{2n-s}}
\]
to \(\zeta_{\mfh_n(\lri)}(s)\). Summing over all \(\mu\in\mcP_n\) yields
\[
  \zeta_{\mfh_n(\lri)}(s)
  =
  \frac{1}{1-q^{2n-s}}
  \sum_{\mu\in\mcP_n}
  N_\lri(\mu)\,q^{-s|\mu|}
  \bigl(1-q^{(\mu_n+1)(2n-s)}\bigr).
\]

Finally, using \Cref{cor:N.N'} we obtain the following formula.
\begin{prop}\label{prop:zeta.formula}
  We have
  \[
    \zeta_{\mfh_n(\lri)}(s)
    =
    \frac{1}{1-q^{2n-s}}
    \sum_{\mu\in\mcP_n}
    N'_\lri(\mu)\,\alpha_n(\mu;q^2)\,q^{-s|\mu|}
    \bigl(1-q^{(\mu_n+1)(2n-s)}\bigr).
  \]
\end{prop}

\section{Proof of \Cref{thm:A}}\label{sec:Wn-parametrized-formula}

\Cref{prop:zeta.formula} expresses the Heisenberg subgroup zeta function in
terms of the polynomials \(N'_\lri(\mu)\) and~\(\alpha_n(\mu;q^2)\). In this
section we isolate the contribution of \(N'_\lri(\mu)\) and derive a first
explicit formula for the local zeta function in terms of augmented Igusa
functions. This formula is indexed by the finite set \(\mcW_n\) (see
\Cref{def:Wn}) and will serve as the starting point for the later
simplifications in \Cref{sec:simplifying-zeta-formula} and the
type-\(\mathsf{B}\) reformulation in~\Cref{sec:Bn}.

Recall from \Cref{prop:closed.N'} that
\(N'_\lri(\mu)=\sum_{\mathbf{w}\in\mathcal{W}_n}
C_n(\mathbf{w})q^{\mathbf{w}\cdot\mu}\) is a polynomial in the
quantities~\(q^{\mu_1},q^{\mu_2},\dots,q^{\mu_n}\), featuring the
rational expressions \(C_n(\bfw)\). To isolate the analytic
contribution coming from these monomials we set, for
\(\mathbf{w}=(w_1,\dots,w_n)\in\Z^n\),
\begin{equation}\label{eq:Z-w-Birkhoff}
  Z(\mathbf{w})
  =
  \frac{1}{1-q^{2n-s}}
  \sum_{\mu\in\mcP_n}
  q^{\mu \cdot \bfw}\,
  \alpha_n(\mu;q^2)\,
  q^{-s|\mu|}\,
  \bigl(1-q^{\,(\mu_n+1)(2n-s)}\bigr)\in \Z(q,T).
\end{equation}

Then \Cref{prop:zeta.formula} states that
\begin{equation}
  \label{eq:zeta-formula-2}
  \zeta_{\mfh_{n}(\lri)}(s)
  =
  \sum_{\mathbf{w}\in\mathcal{W}_n} C_n(\mathbf{w})Z(\mathbf{w}).
\end{equation}

Our aim is to rewrite each \(Z(\mathbf{w})\) as an \emph{augmented Igusa
  function} (cf.\ \Cref{def:igusa.A}) and thereby obtain an explicit
\(\mcW_n\)-parametrized formula for~\(\zeta_{\mfh_n(\lri)}(s)\)
in~\Cref{thm:final-zeta}.

\subsection{Rewriting \texorpdfstring{$Z(\mathbf{w})$}{Z(w)} as Igusa functions}\label{subsec:Z-as-igusa}

Let \(\mu=(\mu_1,\dots,\mu_n)\in\mcP_n\), and let
\(\mathbf d=d(\mu)\in\N_0^n\) be its difference vector, defined by
\(d_i=\mu_i-\mu_{i+1}\) for \(i<n\) and \(d_n=\mu_n\). Conversely, every
\(\mathbf d\in\N_0^n\) determines a unique partition
\(\mu(\bfd)=\bigl(\sum_{j=i}^n d_j\bigr)_{i=1}^n\in\mcP_n\). Given
\(\mathbf{w}=(w_1,\dots,w_n)\in\Z^n\), define
\(u_k:=\sum_{i=1}^k w_i\) for \(k\in[n]\) and
\(\mathbf u=(u_1,\dots,u_n)\in\Z^n\). Then
\(\mu\cdot\mathbf w=\mathbf d\cdot\mathbf u\).

Recall from \Cref{eq:Birk-n-support} that
\[
  \alpha_n(\mu;q^2)
  =
  q^{2\,\mathbf d\cdot\boldsymbol{\rho}'}
  \binom{n}{\Supp^+_{n-1}(\mathbf d)}_{q^{-2}},
  \qquad
  \rho'_k=k(n-k).
\]
Substituting these expressions into \eqref{eq:Z-w-Birkhoff} yields
\begin{equation}\label{eq:Z.rewrite}
  Z(\mathbf{w})
  =
  \frac{1}{1-q^{2n-s}}
  \sum_{\mathbf d\in\N_0^n}
  q^{\mathbf d\cdot(\mathbf u+2\boldsymbol{\rho}')}
  q^{-s|\mu(\mathbf d)|}
  \bigl(1-q^{(d_n+1)(2n-s)}\bigr)
  \binom{n}{\Supp^+_{n-1}(\mathbf d)}_{q^{-2}}.
\end{equation}

We further define vectors \(\boldsymbol u_1,\boldsymbol u_2\in\Z[s]^n\) by setting, for \(k\in[n]\),
  \[
    u_{1,k}
    =
    u_k+2\rho'_k-sk
    =
    \sum_{i=1}^k w_i+2k(n-k)-ks,
    \qquad
    u_{2,k}
    =
    u_{1,k}+(2n-s)\delta_{k,n}.
  \]
Note that \(u_{2,i}=u_{1,i}\) for \(i\in[n-1]\) and \(u_{2,n}=u_{1,n}+2n-s\).

\begin{prop}\label{prop:Z-w-simplified}
  We have
  \[
    Z(\mathbf w)
    =
    \frac{1}{(1-q^{u_{1,n}})(1-q^{u_{1,n}+2n-s})}
    \sum_{I\subseteq[n-1]}
    \binom{n}{I}_{q^{-2}}
    \prod_{i\in I}\frac{q^{u_{1,i}}}{1-q^{u_{1,i}}}.
  \]
\end{prop}

\begin{proof}
  Using \(|\mu(\mathbf d)|=\sum_{k=1}^n k\,d_k\), we may rewrite \eqref{eq:Z.rewrite} as
  \[
    (1-q^{2n-s})\,Z(\mathbf w)
    =
    \sum_{\mathbf d\in\N_0^n}
    \Bigl(
      q^{\mathbf d\cdot\boldsymbol u_1}
      -
      q^{2n-s}\,q^{\mathbf d\cdot\boldsymbol u_2}
    \Bigr)
    \binom{n}{\Supp^+_{n-1}(\mathbf d)}_{q^{-2}}.
  \]

  We sort the summands by their support
  \(I=\Supp^+_{n-1}(\mathbf d)\subseteq[n-1]\).
  For \(I\subseteq[n-1]\), set
  \[
    \mathcal D(I)
    =
    \bigl\{\mathbf d\in\N_0^n :
      d_i\ge1 \textup{ for } i\in I,\;
      d_j=0 \textup{ for } j\notin I,\ j<n,\;
      d_n\ge0
    \bigr\}.
  \]
  The claimed identity follows from the fact that
  \begin{align*}
    (1-q^{2n-s})Z(\mathbf w)
    &=
    \sum_{I\subseteq[n-1]}
    \binom{n}{I}_{q^{-2}}
    \sum_{\mathbf d\in\mathcal D(I)}
    \Bigl(
      q^{\mathbf d\cdot\boldsymbol u_1}
      -
      q^{2n-s}\,q^{\mathbf d\cdot\boldsymbol u_2}
    \Bigr)
    \\
    &=
    \sum_{I\subseteq[n-1]}
    \binom{n}{I}_{q^{-2}}
    \Biggl[
      \prod_{i\in I}\frac{q^{u_{1,i}}}{1-q^{u_{1,i}}}\cdot
      \frac{1}{1-q^{u_{1,n}}}
      -
      q^{2n-s}
      \prod_{i\in I}\frac{q^{u_{2,i}}}{1-q^{u_{2,i}}}\cdot
      \frac{1}{1-q^{u_{2,n}}}
    \Biggr]\\
    &=
    \left(
      \sum_{I\subseteq[n-1]}
      \binom{n}{I}_{q^{-2}}
      \prod_{i\in I}\frac{q^{u_{1,i}}}{1-q^{u_{1,i}}}
    \right)
    \left(
      \frac{1}{1-q^{u_{1,n}}}
      -
      \frac{q^{2n-s}}{1-q^{u_{1,n}+2n-s}}
    \right).
  \end{align*}
  Cancelling the prefactor \(1-q^{2n-s}\) gives the claim.
\end{proof}

This allows us to express
\(Z(\bfw)\) in terms of classical generating functions known as \emph{Igusa
  functions}. Let \(Y\) and \(X_0,X_1,X_2,\dots\) be indeterminates.

\begin{definition}[Igusa functions]\label{def:igusa.A}
  For \(n\ge 1\) we define the
  \begin{alignat*}{2}
    \textup{truncated Igusa function} &&\quad    \igm{n}(Y;X_1,\dots,X_{n-1})
    &:=
      \sum_{I\subseteq [n-1]}
      \binom{n}{I}_Y
      \prod_{i\in I}\frac{X_i}{1-X_i},\\
    \textup{Igusa function} &&\quad     \ig{n}(Y;X_1,\dots,X_n)
    &:=
      \sum_{I\subseteq [n]}
      \binom{n}{I}_Y
      \prod_{i\in I}\frac{X_i}{1-X_i},\\
    \textup{augmented Igusa function} && \quad     \igu{n}(Y;X_0,X_1,\dots,X_n)
    &:=
      \sum_{I\subseteq [n]_0}
      \binom{n}{I}_Y
      \prod_{i\in I}\frac{X_i}{1-X_i}
  \end{alignat*}
  of degree \(n\), respectively. For \(n=0\) we define $\ig{0}(Y;-):=1$, 
    $\igu{0}(Y;X_0):=\frac{1}{1-X_0}$.
\end{definition}

\begin{remark}
  The augmented Igusa function admits a Coxeter-theoretic description
  in terms of descent sets and the Coxeter length
  function~$\ell$. More precisely, for $g\in S_n$ let $\Des(g):=\{\,
  i\in[n-1]\mid g(i)>g(i+1)\,\}$ and \(\des(g):=|\Des(g)|\). Then
  (e.g.~\cite[Rem.~3.12]{SV2/16})
  \begin{equation}\label{eq:igu-descent}
    \igu{n}(Y;X_0,X_1,\dots,X_n)
    =
    \frac{\displaystyle\sum_{g\in S_n}
    Y^{\ell(g)} \prod_{j\in \Des(g)} X_j}{\displaystyle\prod_{i=0}^n(1-X_i)}.
\end{equation}
\end{remark}

\begin{remark}
  Igusa functions feature widely in the literature on enumerative
  algebra. Going back to work of Igusa on $p$-adic integrals
  associated with representations of $p$-adic Lie groups
  \cite{Igusa/89}, they occur (in Lie type $\mathsf{A}$) in various
  enumerative contexts, including \cite{Voll2005, SV1/15}. Analogues
  and generalizations were studied, for instance, in \cite{KV/09,
    CSV/24, MV2}. See \cite{Voll/25} for a short, informal overview of
  recent applications of Igusa-type functions in lattice enumeration.
\end{remark}

The well-known symmetry $\binom{n}{n-I}_Y = \binom{n}{I}_Y$ of $Y$-multinomial
coefficients implies the following easy lemma.

\begin{lemma}\label{lem:Igusa-reversal}
  For \(n\ge 1\) we have
  \[
    \igm{n}(Y;X_1,\dots,X_{n-1})
    =
    \igm{n}(Y;X_{n-1},\dots,X_1).
  \]
\end{lemma}

We also omit the proof of the following elementary lemma.
\begin{lemma}\label{lem:bridge-Igusa}
  For \(n\ge 1\) we have
  \[
    \frac{\igm{n}(Y;X_1,\dots,X_{n-1})}{(1-X_0)(1-X_n)}
    =
    \frac{\ig{n}(Y;X_1,\dots,X_n)}{1-X_0}
    =
    \igu{n}(Y;X_0,X_1,\dots,X_n).
  \]
\end{lemma}

Combining \Cref{prop:Z-w-simplified} with \eqref{eq:zeta-formula-2} gives a
first explicit expression for the local subalgebra zeta function of the higher
Heisenberg Lie algebras. For \(\mathbf w\in\mcW_n\), define
\begin{equation}\label{def:Xk}
    X_k(\mathbf w)
    =
    \begin{cases}
      q^{u_{1,k}}
      =
      q^{\sum_{i=1}^k w_i+2k(n-k)-ks},
      & \textup{for } k\in[n],\\[4pt]
      q^{2n-s}X_n(\mathbf w),
      & \textup{for } k=0.
    \end{cases}
\end{equation}

\begin{theorem}\label{thm:final-zeta}
We have
  \begin{equation}\label{eq:final-zeta}
    \zeta_{\mfh_{n}(\lri)}(s)
    =
    \sum_{\mathbf{w}\in\mcW_n}
    C_n(\mathbf{w})\,
    \igu{n}\!\bigl(q^{-2};X_0(\mathbf w),X_1(\mathbf w),\dots,X_n(\mathbf w)\bigr).
  \end{equation}
\end{theorem}

\begin{proof}
  By \Cref{prop:Z-w-simplified} and \Cref{lem:bridge-Igusa}, we have
  \[
    Z(\mathbf w)
    =
    \frac{\igm{n}\!\bigl(q^{-2};X_1(\mathbf w),\dots,X_{n-1}(\mathbf w)\bigr)}
    {\bigl(1-X_n(\mathbf w)\bigr)\bigl(1-X_0(\mathbf w)\bigr)}
    =
    \igu{n}\!\bigl(q^{-2};X_0(\mathbf w),X_1(\mathbf w),\dots,X_n(\mathbf w)\bigr).\qedhere
  \]
\end{proof}

To summarize, \Cref{thm:final-zeta} expresses the subalgebra zeta function of
the higher Heisenberg algebra of degree $n$ as a sum of $|\mcW_n|=2^n$
instantiations of augmented Igusa functions. Each summand is a rational
function in \(q^{-s}\) with coefficients in \(\Q(q)\); by the descent
formula \eqref{eq:igu-descent}, each Igusa numerator comprises \(n!\)
monomials, indexed by the elements of \(S_n\), and is defined via descent
data and the numerical data specified in \eqref{def:Xk}.

\subsection{Low-degree examples}\label{subsec:zeta-examples}
We spell out \Cref{thm:final-zeta} for \(n\in\{1,2,3\}\). This serves
two purposes. First, we recover the known formulas in these
degrees. Second, we reveal additional structure hidden by the explicit
formula given in \Cref{eq:final-zeta}. This motivates the further
simplification carried out in \Cref{sec:simplifying-zeta-formula}.

\subsubsection{\(n=1\)}
Here \(\mathcal{W}_1=\{(0),(1)\}\), 
\(
  \igu{1}(Y;X_0,X_1) = \frac{1}{(1-X_0)(1-X_1)}
\),
and
\[
  C_1(w_1)=\frac{1}{1-q^{1-2w_1}}, \qquad X_1(w_1)=q^{w_1}T, \qquad X_0(w_1) =
  q^2T\cdot X_1(w_1).
\]
Thus
\[
  \zeta_{\mfh_1(\lri)}(s)
  =
  \sum_{w_1\in\{0,1\}}
  \frac{1}{1-q^{1-2w_1}}\cdot
  \frac{1}{(1-q^{w_1}T)\,(1-q^{2}T\cdot q^{w_1}T)}.
\]
Simplifying gives the well-known subalgebra zeta function
\[
  \zeta_{\mfh_1(\lri)}(s)
  =
  \frac{1-q^3T^3}{(1-T)(1-qT)(1-q^3T^2)(1-q^2T^2)},
\]
of the Heisenberg Lie algebra \(\mfh_1(\lri)\) as first established by
\cite[Proposition~8.1]{GSS/88}.

\subsubsection{\(n=2\)}
\Cref{thm:final-zeta} gives
\[
  \zeta_{\mfh_2(\lri)}(s)
  =
  \sum_{(w_1,w_2)\in\mathcal{W}_2} C_2(w_1,w_2)\cdot
  \frac{1+q^{-2}X_1(w_1,w_2)}
  {(1-X_1(w_1,w_2))(1-X_2(w_1,w_2))(1-q^{4}T\,X_2(w_1,w_2))},
\]
where \(\mathcal{W}_2=\{(0,0),(0,3),(1,1),(1,2)\}\) and
\[
  \igu{2}(q^{-2};X_0,X_1,X_2)
  =
  \frac{1+q^{-2}X_1}{(1-X_0)(1-X_1)(1-X_2)}.
\]
Moreover,
\[
  C_2(w_1,w_2)
  =
  \frac{1}{(1-q^{1-2w_1})(1-q^{3-2w_2})},
\]
and
\[
  X_1(w_1,w_2)=q^{w_1+2}T, \qquad
  X_2(w_1,w_2)=q^{w_1+w_2}T^{2}, \qquad
  X_0(w_1,w_2)=q^{4}T\,X_2(w_1,w_2).
\]

Grouping the four summands $\Sigma_1,\dots,\Sigma_4$ according to
the factor $(1-q^aT^3)^{-1}$, with $a=4+w_1+w_2\in\{4,6,7\}$, we
obtain three blocks
\begin{align*}
  \Sigma_1= B_0&=
  \frac{1}
  {(1-q)(1-q^3)(1-T)(1-q^2T)(1-q^4T^3)},\\
  \Sigma_3=B_1&=
  -\frac{q}
  {(1-q)^2(1-qT)(1-q^3T)(1-q^6T^3)},\\
  \Sigma_2+\Sigma_4=B_2&=
  \frac{q^2(1+q^2)}
  {(1-q)(1-q^3)(1-q^2T)(1-q^3T)(1-q^7T^3)}.
\end{align*}

Simplifying yields
\[
  \zeta_{\mfh_2(\lri)}(s)
  =
  \frac{
    1+q^{5}T^{3}-(q^{5}+q^{6}+q^{7}+q^{8})T^{4}+q^{8}T^{5}+q^{13}T^{8}
  }
  {(1-T)(1-qT)(1-q^{2}T)(1-q^{3}T)(1-q^{4}T^{3})(1-q^{6}T^{3})(1-q^{7}T^{3})}.
\]
This confirms the formula in \cite[Thm.~2.22]{duSW/08}, presumably due
to Woodward. Moreover, it realizes the form of \Cref{thm:Bn-target} in
the case \(n=2\).

\subsubsection{\(n=3\)}
For \(n=3\), \Cref{thm:final-zeta} gives
\[
  \zeta_{\mfh_3(\lri)}(s)
  =
  \sum_{w=(w_1,w_2,w_3)\in\mathcal W_3}
  C_3(w)\,
  \igu{3}\!\left(q^{-2};X_0(w),X_1(w),X_2(w),X_3(w)\right),
\]
where
\[
  \mathcal W_3=
  \{(0,0,0),(0,0,5),(0,3,3),(0,3,2),
    (1,1,1),(1,1,4),(1,2,2),(1,2,3)\},
\]
\[
  \igu{3}(q^{-2};X_0,X_1,X_2,X_3)
  =
  \frac{
    1+(q^{-2}+q^{-4})X_1+(q^{-2}+q^{-4})X_2+q^{-6}X_1X_2
  }
  {(1-X_0)(1-X_1)(1-X_2)(1-X_3)},
\]
\[
  C_3(w)
  =
  \frac{1}{(1-q^{1-2w_1})(1-q^{3-2w_2})(1-q^{5-2w_3})},
\]
\[
  X_1(w)=q^{w_1+4}T,\quad
  X_2(w)=q^{w_1+w_2+4}T^2,\quad
  X_3(w)=q^{w_1+w_2+w_3}T^3,\quad
  X_0(w)=q^6T\,X_3(w).
\]

For example, the second summand, corresponding to \(\mathbf w=(0,0,5)\), is
\[
  \Sigma_2
  =
  \frac{1}{(1-q)(1-q^{3})(1-q^{-5})}\cdot
  \frac{1+T+q^{2}T+T^{2}+q^{2}T^{2}+q^{2}T^{3}}
  {(1-q^{4}T)(1-q^{4}T^{2})(1-q^{5}T^{3})(1-q^{11}T^{4})}.
\]
Grouping the eight summands \(\Sigma_1,\dots,\Sigma_8\) according
to the factor \((1-X_0)^{-1}=(1-q^aT^4)^{-1}\), with
\(a=6+w_1+w_2+w_3\in \{6,9,11,12\}\), we obtain four blocks
\begin{align*}
  \Sigma_1=B_0&=
  \frac{1}
  {(1-q)(1-q^3)(1-q^5)(1-T)(1-q^2T)(1-q^4T)(1-q^6T^4)},\\
  \Sigma_5=B_1&=
  -\frac{q}
  {(1-q)^2(1-q^3)(1-qT)(1-q^3T)(1-q^5T)(1-q^9T^4)},\\
  \Sigma_2+\Sigma_4+\Sigma_7=B_2&=
  \frac{q^2(1+q^2+q^4)}
  {(1-q)^2(1-q^5)(1-q^2T)(1-q^4T)(1-q^5T)(1-q^{11}T^4)},\\
  \Sigma_3+\Sigma_6+\Sigma_8=B_3&=
  -\frac{q^3(1+q^3)}
  {(1-q)(1-q^2)(1-q^3)(1-q^3T)(1-q^4T)(1-q^5T)(1-q^{12}T^4)}.
\end{align*}

Combining the four block terms gives
\[
  \zeta_{\mfh_3(\lri)}(s)
  =
  \frac{M_3(q,T)}
  {\left(\prod_{i=0}^{5}(1-q^{i}T)\right)(1-q^{6}T^{4})(1-q^{9}T^{4})(1-q^{11}T^{4})(1-q^{12}T^{4})},
\]
where
\begin{align*}
  M_3(q,T) = &\, 1 + q^{7}T^4 + q^{8}T^4 + q^{9}T^4 + q^{10}T^4 -
  q^{7}T^5 - q^{8}T^5 - 2q^{9}T^5 - 2q^{10}T^5 \\
  &- 2q^{11}T^5 -
  2q^{12}T^5 - q^{13}T^5 - q^{14}T^5 + q^{10}T^6 + q^{11}T^6 +
  q^{12}T^6 + q^{13}T^6 \\
  &+ q^{14}T^6 + q^{15}T^6 - q^{15}T^7 +
  q^{17}T^8 - q^{17}T^9 - q^{18}T^9 - q^{19}T^9 - q^{20}T^9 -
  q^{21}T^9 \\
  &- q^{22}T^9 + q^{18}T^{10} + q^{19}T^{10} +
  2q^{20}T^{10} + 2q^{21}T^{10} + 2q^{22}T^{10} + 2q^{23}T^{10} \\
  &+
  q^{24}T^{10} + q^{25}T^{10} - q^{22}T^{11} - q^{23}T^{11} -
  q^{24}T^{11} - q^{25}T^{11} - q^{32}T^{15}.
\end{align*}
This recovers the formula in \cite[Sec.\ 3.3.13.9]{Berman2005}, due to Klopsch
and the second author. Moreover, it realizes the form of \Cref{thm:Bn-target}
in the case \(n=3\).

These examples suggest that the \(2^n\) summands of the
\(\mcW_n\)-parametrized sum formula given by \Cref{thm:final-zeta} naturally
cluster into larger blocks, according to certain denominator
factors~\((1-q^aT^{n+1})^{-1}\). In \Cref{sec:simplifying-zeta-formula} we
exploit this to prove \Cref{thm:B}.

\section{Proof of \Cref{thm:B}}\label{sec:simplifying-zeta-formula}
Motivated by the computations in \Cref{subsec:zeta-examples}, we
collect the terms in \Cref{thm:final-zeta} with the same factor
\((1-q^aT^{n+1})^{-1}\). Equivalently, we group the summands by the
terminal fibre \(\{w_n,2n+1-w_n\}\), indexed by \(r\in[n]_0\). This
leads to the \((n+1)\)-term formula of \Cref{thm:simplified-zeta}
(=\Cref{thm:B}), compactifying the original \(2^n\)-term expression in
\Cref{thm:final-zeta}.

\begin{theorem}\label{thm:simplified-zeta}
  We have
  \[
    \zeta_{\mfh_n(\lri)}(s)
    =
    \sum_{r=0}^{n}
    \frac{1}{\left(1-q^{\,2n+\frac{r(2n+1-r)}{2}}\,T^{n+1}\right)}
    \cdot
    \frac{
      (-q)^r(1-q^{2n-2r+1})(q^2;q^2)_n
    }{
      (q;q)_{2n-r+1}(q;q)_r
      (q^{r}T;q^{2})_{n-r}
      (q^{2n-r}T;q)_{r}
    }.
  \]
\end{theorem}
Here and throughout this section, for \(m\in\Z\) we use
\[
  (a;q)_m
  :=
  \begin{cases}
    \prod_{i=0}^{m-1}(1-aq^i), & m\ge 0,\\[2mm]
    \bigl((aq^m;q)_{-m}\bigr)^{-1}, & m<0.
  \end{cases}
\]
The whole section is devoted to the proof of
\Cref{thm:simplified-zeta}. In \Cref{subsec:Wn-structure}, we group
the \(2^n\) summands in \Cref{thm:final-zeta} according to the
terminal value \(w_n\). This gives the decomposition
\eqref{eq:zeta-Innr} into \emph{terminal fibre sums} \(\mathcal
I_n^{n,r}\). We then state the main intermediate identity
\Cref{thm:I-equals-K}, which relates the fibre sums \(\mathcal
I_n^{k,r}\) to a coset model \(\mathcal K_n^{k,r}\) on \(S_n/S_k\),
and use its terminal case \(k=n\) to prove \Cref{thm:simplified-zeta}.
In the remaining subsections we prove \Cref{thm:I-equals-K}:
\Cref{subsec:Jn} derives the recursion for a normalised version of
\(\mathcal I_n^{k,r}\), \Cref{subsec:Ekr} establishes the coefficient
identities for \(E_{k,r}\), and \Cref{subsec:Kn} applies these
identities to show that \(\mathcal K_n^{k,r}\) satisfies the same
initial condition and recursion.

\subsection{Fibre sums and the coset model}
\label{subsec:Wn-structure}

Recall the finite set \(\mcW_k\) from \Cref{def:Wn}.
The following identity for partial sums is immediate by induction.

\begin{lemma}\label{lem:Wn-partial-sum}
  For \(\mathbf w=(w_1,\dots,w_k)\in\mcW_k\) and \(j\in[k]\), we have
  $
    \sum_{i=1}^j w_i
    =
    \frac{w_j(2j+1-w_j)}{2}.
  $
\end{lemma}

We partition \(\mcW_k\) according to the last entry. For \(k\in\N_0\) and
\(r\in\Z\), set
\begin{equation}\label{def:Wkr}
  \mcW_{k,r}
  :=
  \{\mathbf w\in\mcW_k \mid w_k\in\{r,2k+1-r\}\}.
\end{equation}
Thus \(\mcW_{k,r}=\mcW_{k,2k+1-r}\), and
\(\mcW_{k,r}=\varnothing\) unless \(r\in[2k+1]_0\). For \(k=0\), we use the convention
\(\mcW_{0,0}=\mcW_{0,1}=\{\varnothing\}\). Then
\(\mcW_k=\bigsqcup_{r=0}^{k}\mcW_{k,r}\).

The fibres satisfy the following elementary recursion.

\begin{lemma}\label{lem:Wn-recursion}
  Let \(k\in\N_0\) and \(r\in\Z\), and set \(r':=2k+3-r\). Then
  \[
    \mcW_{k+1,r}
    =
    \{(\mathbf w,r)\mid \mathbf w\in\mcW_{k,r}\}
    \sqcup
    \{(\mathbf w,r')\mid \mathbf w\in\mcW_{k,r'}\}.
  \]
\end{lemma}

\begin{proof}
  This follows directly from the defining condition
  \(w_{k+1}\in\{w_k,\,2k+1-w_k\}\).
\end{proof}

This is the combinatorial input for regrouping the \(2^n\)-term formula of
\Cref{thm:final-zeta} into \(n+1\) terminal fibre sums. Recall from
\eqref{def:C} that
\(C_k(\mathbf w)=\prod_{i=1}^k(1-q^{2i-1-2w_i})^{-1}.\)

For \(j\in\N\) and \(r\in\Z\), set
\begin{equation*}
  Y_j(r,T)
  :=
  q^{\frac{r(2j+1-r)}{2}-2j^2}\,T^j.
\end{equation*}
By \Cref{lem:Wn-partial-sum}, the expressions in \Cref{def:Xk} may be
written as
\[
  X_k(\mathbf w)
  =
  q^{\frac{w_k(2k+1-w_k)}{2}+2k(n-k)-ks}
  =
  Y_k(w_k,q^{2n}T),
  \qquad k\in[n].
\]

\begin{definition}[Fibre sums]\label{def:Inkr}
  For \(k\in[n]_0\) and \(r\in\Z\), define
  \[
    \mathcal I_n^{k,r}(X_{k+1},\dots,X_n;T)
    :=
    \sum_{\mathbf w\in\mcW_{k,r}}
    C_k(\mathbf w)\,
    \ig{n}\bigl(q^{-2};
      Y_1(w_1,T),\dots,Y_k(w_k,T),
      X_{k+1},\dots,X_n
    \bigr).
  \]
\end{definition}
This is the partial substitution of the Igusa function appearing in
\Cref{thm:final-zeta}, summed over the fibre \(\mcW_{k,r}\). For
\(k=0\), this gives \( \mathcal I_n^{0,0} = \mathcal I_n^{0,1} =
\ig{n}(q^{-2};X_1,\dots,X_n), \) and \(\mathcal I_n^{0,r}=0\) for
\(r\notin\{0,1\}\).

In \eqref{eq:final-zeta}, the factor
\(
  X_0(\mathbf w)
  =
  q^{\,2n+\frac{w_n(2n+1-w_n)}{2}}\,T^{n+1}
\)
depends only on \(w_n\). Grouping the summands in \eqref{eq:final-zeta}
according to
\(\mcW_n=\bigsqcup_{r=0}^{n}\mcW_{n,r}\) gives
\begin{equation}\label{eq:zeta-Innr}
  \zeta_{\mfh_n(\lri)}(s)
  =
  \sum_{r=0}^{n}
  \frac{\mathcal I_n^{n,r}(-;q^{2n}T)}
  {1-q^{\,2n+\frac{r(2n+1-r)}{2}}\,T^{n+1}}.
\end{equation}

We next introduce the coset model \(\mathcal K_n^{k,r}\).

\begin{definition}\label{def:Ekr-Bkrt}
  For \(s,r\in\Z\) and \(k, t\in\N_0\), set
  \[
    F_s(x):=\bigl(-q^{\,1-s}x;q^2\bigr)_{\lfloor s/2\rfloor},
    \qquad
    E_{k,r}(x):=F_r(x)F_{2k+1-r}(x).
  \]
  Let \(\varepsilon_{k,r}^{(t)}:=[x^t]E_{k,r}(x)\), where \(E_{k,r}\) is
  expanded at \(x=0\).
  Put \(e_{k,r}^{(t)}=\varepsilon_{k,r}^{(t)}\) for
  \(r\in[2k+1]_0\), and \(e_{k,r}^{(t)}=0\) otherwise. For \(t\in[k]_0\), set
  \[
    B_{k,r}^{(t)}
    :=
    q^{-t(t-1)}[t]_{q^2}![k-t]_{q^2}!\,e_{k,r}^{(t)}.
  \]
\end{definition}

For \(k\in[n]_0\) and \(gS_k\in S_n/S_k\), put \(g(n+1)=n+1\). We define
the following statistics: the insertion rank
\[
  t_k(g):=\#\{\,i\in[k]\mid g(i)>g(k+1)\,\},
\]
the tail inversion number
\[
  \ell_k^+(g):=\#\{(i,j)\mid 1\le i<j\le n,\ j\ge k+1,\ g(i)>g(j)\},
\]
and
\[
  \Des_{>k}(g):=\Des(g)\cap\{k+1,\dots,n-1\},
  \qquad
  \mathbf X^{\Des_{>k}(g)}:=\prod_{j\in\Des_{>k}(g)}X_j.
\]
They depend only on the coset \(gS_k\).

\begin{definition}\label{def:Knkr}
  For \(k\in[n]_0\), \(r\in\Z\), and \(gS_k\in S_n/S_k\), define
  \begin{align*}
    \mathcal K_n^{k,r}\big|_{gS_k}
    &:=
    B_{k,r}^{(t_k(g))}\,
    q^{-2\ell_k^+(g)}\,
    \mathbf X^{\Des_{>k}(g)}\,
    T^{t_k(g)}, \\
    \mathcal K_n^{k,r}
    &:=
    \sum_{gS_k\in S_n/S_k}
    \mathcal K_n^{k,r}\big|_{gS_k}.
  \end{align*}
\end{definition}

We show that the fibre sums \(\mathcal I_n^{k,r}\) are given by the
preceding coset model after multiplication by an explicit
prefactor. Set
\begin{equation*}
  P_{k,r}(q)
  :=
  \frac{(-q)^r(1-q^2)^k(1-q^{2k-2r+1})}
  {(q;q)_{2k-r+1}(q;q)_r}.
\end{equation*}

\begin{theorem}\label{thm:I-equals-K}
  Let \(k\in[n]_0\) and \(r\in\Z\). Then
  \begin{equation*}
    \mathcal I_n^{k,r}(X_{k+1},\dots,X_n;T)
    =
    P_{k,r}(q)\,
    \frac{\mathcal K_n^{k,r}(X_{k+1},\dots,X_n;T)}
    {E_{k,r}(-T)\prod_{j=k+1}^{n}(1-X_j)}.
  \end{equation*}
\end{theorem}

The proof is given in \Cref{subsec:Jn,subsec:Ekr,subsec:Kn}:
\emph{normalised fibre sums} \(\mathcal J_n^{k,r}\) are introduced in
\eqref{def:Jnkr}, and \Cref{thm:J=K} proves \(\mathcal
J_n^{k,r}=\mathcal K_n^{k,r}\).

\begin{ex}\label{exa:422}
  Take \((n,k,r)=(4,2,2)\). Then
  \[
    P_{2,2}(q)=\frac{q^2}{(1-q)(1-q^3)},\qquad
    E_{2,2}(x)=(1+q^{-2}x)(1+q^{-1}x),
  \]
  so
  \[
    e_{2,2}^{(0)}=1,\qquad
    e_{2,2}^{(1)}=q^{-2}+q^{-1},\qquad
    e_{2,2}^{(2)}=q^{-3}.
  \]
  The coset model gives
  \begin{align*}
    \mathcal K_4^{2,2}(X_3,X_4;T)
    &=
    1+q^2
    +(q^{-6}+q^{-5}+q^{-4}+q^{-3})T \\
    &\quad
    +(q^{-13}+2q^{-11}+2q^{-9}+q^{-7})T^2
    +(q^{-6}+2q^{-4}+2q^{-2}+1)X_3 \\
    &\quad
    +(q^{-10}+q^{-9}+q^{-8}+q^{-7})TX_3
    +(q^{-15}+q^{-13})T^2X_3.
  \end{align*}
  Hence
  \[
    \mathcal I_4^{2,2}
    =
    \frac{q^2\,\mathcal K_4^{2,2}}
    {(1-q)(1-q^3)(1-q^{-2}T)(1-q^{-1}T)(1-X_3)(1-X_4)},
  \]
  in agreement with the direct computation from \Cref{def:Inkr}.
\end{ex}

We now deduce \Cref{thm:simplified-zeta} from the terminal case of
\Cref{thm:I-equals-K}.
\begin{proof}[Proof of \Cref{thm:simplified-zeta}]
  Specialising \Cref{thm:I-equals-K} to \(k=n\), we use that \(S_n/S_n\) is a
  singleton. Hence \(t_n(g)=0\), \(\ell_n^+(g)=0\), and
  \(\Des_{>n}(g)=\varnothing\), so
  \[
    \mathcal K_n^{n,r}(-;T)
    =
    B_{n,r}^{(0)}
    =
    [n]_{q^2}!\,e_{n,r}^{(0)}
    =
    [n]_{q^2}!.
  \]
  Thus, for \(r\in[n]_0\),
  \[
    \mathcal I_n^{n,r}(-;T)
    =
    \frac{(-q)^r(q^2;q^2)_n(1-q^{2n-2r+1})}
    {(q;q)_{2n-r+1}(q;q)_r}\,
    \frac{1}{E_{n,r}(-T)}.
  \]
  Moreover,
  \begin{align*}
    E_{n,r}(-T)
    &=
    (q^{r-2n}T;q^2)_{\lfloor(2n+1-r)/2\rfloor}
    (q^{1-r}T;q^2)_{\lfloor r/2\rfloor} \\
    &=
    (q^{r-2n}T;q^2)_{n-r}
    (q^{-r}T;q^2)_{\lceil r/2\rceil}
    (q^{1-r}T;q^2)_{\lfloor r/2\rfloor} \\
    &=
    (q^{r-2n}T;q^2)_{n-r}(q^{-r}T;q)_r.
  \end{align*}
  Therefore
  \[
    \mathcal I_n^{n,r}(-;q^{2n}T)
    =
    \frac{(-q)^r(q^2;q^2)_n(1-q^{2n-2r+1})}
    {(q;q)_{2n-r+1}(q;q)_r
     (q^rT;q^2)_{n-r}(q^{2n-r}T;q)_r}.
  \]
  Substitution into \eqref{eq:zeta-Innr} gives the formula in
  \Cref{thm:simplified-zeta}.
\end{proof}

\subsection{Normalised fibre sums}\label{subsec:Jn}
To start the proof of \Cref{thm:I-equals-K} we first normalise the
fibre sums and derive their recursion. For \(k\in[n]_0\) and
\(r\in\Z\), define
\begin{equation}\label{def:Jnkr}
  \mathcal J_n^{k,r}(X_{k+1},\dots,X_n;T)
  :=
  \frac{\mathcal I_n^{k,r}(X_{k+1},\dots,X_n;T)}
       {P_{k,r}(q)}
  E_{k,r}(-T)\prod_{j=k+1}^{n}(1-X_j).
\end{equation}

The fibre sums satisfy the following two-term recursion.

\begin{lemma}\label{lem:In-recursion}
  Let \(k\in[n-1]_0\), \(r\in\Z\), and \(r':=2k+3-r\). Then
  \[
    \mathcal I_n^{k+1,r}(X_{k+2},\dots,X_n;T)
    =
    \sum_{u\in\{r,r'\}}
    \frac{
      \mathcal I_n^{k,u}
      \bigl(Y_{k+1}(r,T),X_{k+2},\dots,X_n;T\bigr)
    }
    {1-q^{2k+1-2u}} .
  \]
\end{lemma}
\begin{proof}
  By \Cref{lem:Wn-recursion},
  \[
    \mcW_{k+1,r}
    =
    \{(\mathbf w,r)\mid \mathbf w\in\mcW_{k,r}\}
    \sqcup
    \{(\mathbf w,r')\mid \mathbf w\in\mcW_{k,r'}\}.
  \]
  Also
  \[
    C_{k+1}(w_1,\dots,w_{k+1})
    =
    C_k(w_1,\dots,w_k)(1-q^{2k+1-2w_{k+1}})^{-1},
  \]
  and \(Y_{k+1}(r,T)=Y_{k+1}(r',T)\). Substitution in
  \Cref{def:Inkr} gives the formula.
\end{proof}

After normalisation, this recursion takes the following form.

\begin{lemma}\label{lem:Jn-recursion}
  Let \(k\in[n-1]_0\), \(r\in\Z\), and \(r':=2k+3-r\). Then
  \[
    \mathcal J_n^{k+1,r}
    =
    \frac{
      \left(
        (1-q^{1-r'}T)A_{r'}\mathcal J_n^{k,r}
        -
        (1-q^{1-r}T)A_r\mathcal J_n^{k,r'}
      \right)\big|_{X_{k+1}=Y_{k+1}(r,T)}
    }
    {(q^{-r'}-q^{-r})(1-q^2)(1-Y_{k+1}(r,T))},
  \]
  where \(A_u:=(1-q^{u-1})(q^{-u}-1)\).
\end{lemma}

\begin{proof}
  Substitute \Cref{lem:In-recursion} into \Cref{def:Jnkr}. The ratios
  \(E_{k+1,r}(-T)/E_{k,r}(-T)\) and
  \(E_{k+1,r}(-T)/E_{k,r'}(-T)\) are \(1-q^{1-r'}T\) and
  \(1-q^{1-r}T\), respectively. Also
  \[
    \frac{P_{k,r}(q)}{P_{k+1,r}(q)}
    \frac{1}{1-q^{2k+1-2r}}
    =
    \frac{A_{r'}}{(1-q^2)(q^{-r'}-q^{-r})},
  \]
  while, using \(P_{k+1,r}(q)=P_{k+1,r'}(q)\), the same computation with
  \(r\) and \(r'\) interchanged gives
  \[
    \frac{P_{k,r'}(q)}{P_{k+1,r}(q)}
    \frac{1}{1-q^{2k+1-2r'}}
    =
    -\frac{A_r}{(1-q^2)(q^{-r'}-q^{-r})}.
  \]
  These identities give the stated recursion.
\end{proof}

\subsection{Coefficient identities for \texorpdfstring{$E_{k,r}$}{Ekr}}
\label{subsec:Ekr}

We record the identities for the polynomials \(E_{k,r}\) needed in
the coset recursion. First,
\begin{equation}\label{eq:Fs-ratio}
  \frac{F_s(x)}{F_{s-2}(x)}=1+q^{-(s-1)}x
  \qquad (s\in \Z).
\end{equation}
Next, \(E_{k,r}(x)=E_{k,2k+1-r}(x)\), and \(E_{k,r}(x)\) is a polynomial of
degree \(k\) for \(r\in[2k+1]_0\).

For the rest of this subsection set \(r':=2k+3-r\), and write
\[
  A_r:=(1-q^{r-1})(q^{-r}-1).
\]

\begin{lemma}\label{lem:top-coeff-ek+1}
  Let \(k\in\N_0\) and \(r\in[2k+3]_0\). Then
  \[
    e_{k+1,r}^{(k+1)}
    =
    q^{\frac{rr'}2-(k+1)(k+2)}.
  \]
\end{lemma}

\begin{proof}
  By symmetry, we may assume that \(r\) is odd. Then
  \[
    E_{k+1,r}(x)
    =
    (-q^{1-r}x;q^2)_{\frac{r-1}{2}}
    (-q^{1-r'}x;q^2)_{\frac{r'}2}.
  \]
  Taking leading coefficients gives
  \[
    [x^{k+1}]E_{k+1,r}(x)
    =
    q^{-\frac{r^2-1}{4}}q^{-\frac{{r'}^2}{4}}
    =
    q^{\frac{rr'}2-(k+1)(k+2)},
  \]
  since \(r+r'=2k+3\).
\end{proof}

\begin{prop}\label{prop:Ekr-identity}
  Let \(k\in\N_0\) and \(r\in\Z\). Then
  \[
    A_{r'}E_{k,r}(x)-A_rE_{k,r'}(x)
    =
    (q^{-r'}-q^{-r})
    \bigl(E_{k+1,r}(q^2x)-q^{2k+2}E_{k+1,r}(x)\bigr).
  \]
\end{prop}

\begin{proof}
  Using \eqref{eq:Fs-ratio}, we write
  \[
    E_{k,r}(x)
    =
    F_{r-2}(x)F_{r'-2}(x)(1+q^{-(r-1)}x),
    \quad
    E_{k,r'}(x)
    =
    F_{r-2}(x)F_{r'-2}(x)(1+q^{-(r'-1)}x),
  \]
  and
  \[
    E_{k+1,r}(x)
    =
    F_{r-2}(x)F_{r'-2}(x)
    (1+q^{-(r-1)}x)(1+q^{-(r'-1)}x).
  \]
  Moreover,
  \[
    F_s(q^2x)=F_{s-2}(x)\times
    \begin{cases}
      1+qx,& s\ \text{even},\\
      1+x,& s\ \text{odd}.
    \end{cases}
  \]
  Since \(r+r'=2k+3\) is odd, it follows that
  \[
    E_{k+1,r}(q^2x)=F_r(q^2x)F_{r'}(q^2x)
    =F_{r-2}(x)F_{r'-2}(x)(1+x)(1+qx).
  \]
  After factoring out \(F_{r-2}(x)F_{r'-2}(x)\), the claim reduces to a direct
  expansion using \(r+r'=2k+3\).
\end{proof}

Taking coefficients gives the following identities; \Cref{lem:crossdiff-with-T}
is the main one used in \Cref{subsec:Kn}.

\begin{prop}\label{prop:ekr-relations}
  Let \(k\in\N_0\) and \(r\in\Z\). For all \(t\in\Z\),
  \begin{equation}\label{eq:ekr-Pieri}
    \varepsilon_{k+1,r}^{(t)}
    =
    \varepsilon_{k,r}^{(t)}+q^{1-r'}\varepsilon_{k,r}^{(t-1)},
    \qquad
    \varepsilon_{k+1,r}^{(t)}
    =
    \varepsilon_{k,r'}^{(t)}+q^{1-r}\varepsilon_{k,r'}^{(t-1)}.
  \end{equation}
  Moreover,
  \begin{equation}\label{eq:ekr-crossdiff}
    A_{r'}\,\varepsilon_{k,r}^{(t)}-A_r\,\varepsilon_{k,r'}^{(t)}
    =
    (q^{-r'}-q^{-r})(q^{2t}-q^{2k+2})\,\varepsilon_{k+1,r}^{(t)}.
  \end{equation}
\end{prop}

\begin{proof}
  The first identities follow from
  \[
    E_{k+1,r}(x)
    =
    E_{k,r}(x)(1+q^{1-r'}x)
    =
    E_{k,r'}(x)(1+q^{1-r}x).
  \]
  The second is obtained by applying \([x^t]\) to
  \Cref{prop:Ekr-identity}.
\end{proof}

\begin{prop}\label{lem:crossdiff-with-T}
  Let \(k\in\N_0\) and \(r\in\Z\). For all \(t\in[k]_0\),
  \[
    \begin{aligned}
      &(1-q^{1-r'}T)A_{r'}e_{k,r}^{(t)}
      -
      (1-q^{1-r}T)A_re_{k,r'}^{(t)} \\
      &\qquad =
      (q^{-r'}-q^{-r})
      \Bigl(
        (q^{2t}-q^{2k+2})e_{k+1,r}^{(t)}
        +(q^{2t+2}-1)e_{k+1,r}^{(t+1)}T
      \Bigr).
    \end{aligned}
  \]
\end{prop}

\begin{proof}
  If \(r\notin[2k+3]_0\), both sides vanish. Otherwise the \(e\)'s may be
  replaced by the corresponding \(\varepsilon\)'s: the only possible boundary
  discrepancies occur for \(r\in\{2k+2,2k+3\}\) or \(r'\in\{2k+2,2k+3\}\),
  where the relevant \(A\)-factor is zero. Thus it remains to prove the
  identity with the \(e\)'s replaced by the corresponding
  \(\varepsilon\)'s.

  The constant term in \(T\) is \eqref{eq:ekr-crossdiff}. For the coefficient
  of \(T\), use \eqref{eq:ekr-Pieri} to write
  \[
    q^{1-r'}\varepsilon_{k,r}^{(t)}
    =
    \varepsilon_{k+1,r}^{(t+1)}-\varepsilon_{k,r}^{(t+1)},
    \qquad
    q^{1-r}\varepsilon_{k,r'}^{(t)}
    =
    \varepsilon_{k+1,r}^{(t+1)}-\varepsilon_{k,r'}^{(t+1)}.
  \]
  Together with \eqref{eq:ekr-crossdiff} at \(t+1\) and
  \(A_r-A_{r'}=-(q^{-r'}-q^{-r})(1-q^{2k+2})\), this gives the stated
  \(T\)-coefficient.
\end{proof}

\subsection{The coset recursion}\label{subsec:Kn}

It remains to show that the functions \(\mathcal K_n^{k,r}\) satisfy
the same initial condition and recursion as the normalised fibre
sums~\(\mathcal J_n^{k,r}\).

\begin{lemma}\label{lem:J=K-base}
  For all \(r\in\Z\),
  \[
    \mathcal J_n^{0,r}(X_1,\dots,X_n;T)
    =
    \mathcal K_n^{0,r}(X_1,\dots,X_n;T).
  \]
\end{lemma}

\begin{proof}
  Both sides vanish for \(r\notin\{0,1\}\). By symmetry it suffices to take
  \(r=0\). Then \(P_{0,0}(q)=1\), \(E_{0,0}(x)=1\), and
  \(\mathcal I_n^{0,0}=\ig{n}(q^{-2};X_1,\dots,X_n)\). Hence, by \eqref{def:Jnkr} and
  \eqref{eq:igu-descent},
  \[
    \mathcal J_n^{0,0}
    =
    \ig{n}(q^{-2};X_1,\dots,X_n)\prod_{j=1}^n(1-X_j)
    =
    \sum_{g\in S_n}q^{-2\ell(g)}\mathbf X^{\Des(g)}.
  \]
  For \(k=0\), the cosets \(S_n/S_0\) are elements of \(S_n\), and
  \(t_0(g)=0\), \(\ell_0^+(g)=\ell(g)\), \(\Des_{>0}(g)=\Des(g)\). Thus
  \(\mathcal K_n^{0,0}\) is the same sum.
\end{proof}

To compare the recursions, we work blockwise with respect to the projection
\(S_n/S_k\to S_n/S_{k+1}\). If \(k<n\) and
\(hS_{k+1}\in S_n/S_{k+1}\), define the \(hS_{k+1}\)-block by
\begin{equation*}
  \mathcal K_n^{k,r}\big|_{hS_{k+1}}
    :=
  \sum_{\substack{gS_k\in S_n/S_k\\ gS_k\subseteq hS_{k+1}}}
  \mathcal K_n^{k,r}\big|_{gS_k}.
\end{equation*}
\begin{lemma}\label{lem:K-block-structure}
  Let \(k<n\), \(hS_{k+1}\in S_n/S_{k+1}\), and set
  \(u:=t_{k+1}(h)\). Then the lifts
  \(gS_k\subseteq hS_{k+1}\) are indexed bijectively by \(t=t_k(g)\in[k]_0\).
  For the lift with parameter \(t\),
  \[
    \ell_k^+(g)=\ell_{k+1}^+(h)+t,
    \qquad
    \Des_{>k}(g)=
    \begin{cases}
      \Des_{>k+1}(h)\cup\{k+1\}, & t<u,\\
      \Des_{>k+1}(h), & t\ge u.
    \end{cases}
  \]
\end{lemma}

\begin{proof}
  Fix a representative \(h\). A lift is obtained by choosing \(g(k+1)\) among
  \(h(1),\dots,h(k+1)\), giving \(k+1\) lifts. As this value varies, the number
  of earlier entries larger than it runs through \(0,\dots,k\).

  The difference \(\ell_k^+(g)-\ell_{k+1}^+(h)\) counts precisely the
  inversions with second index \(k+1\), hence equals \(t\). Finally, \(k+1\in\Des(g)\) iff \(g(k+1)>h(k+2)\). If
  \(g(k+1)>h(k+2)\), then
  \[
    \{\,i\in[k]\mid g(i)>g(k+1)\,\}\sqcup\{k+1\}
    \subseteq
    \{\,i\in[k+1]\mid h(i)>h(k+2)\,\},
  \]
  so \(t+1\le u\), hence \(t<u\). If \(g(k+1)<h(k+2)\), the reverse inclusion
  gives \(u\le t\). This gives the claimed descent condition.
  
\end{proof}

\begin{prop}\label{prop:K-block-recursion}
  Let \(k<n\), \(r\in\Z\), and \(r':=2k+3-r\). For every
  \(hS_{k+1}\in S_n/S_{k+1}\),
  \begin{equation}\label{eq:K-block-recursion}
    \mathcal K_n^{k+1,r}\big|_{hS_{k+1}}
    =
    \frac{
      \left(
        (1-q^{1-r'}T)A_{r'}\mathcal K_n^{k,r}\big|_{hS_{k+1}}
        -
        (1-q^{1-r}T)A_r\mathcal K_n^{k,r'}\big|_{hS_{k+1}}
      \right)\big|_{X_{k+1}=Y_{k+1}(r,T)}
    }
    {(q^{-r'}-q^{-r})(1-q^2)(1-Y_{k+1}(r,T))}.
  \end{equation}
  Consequently the same recursion holds for the total sums
  \(\mathcal K_n^{k,r}\).
\end{prop}

\begin{proof}
  Fix \(hS_{k+1}\), and set
  \[
    u:=t_{k+1}(h),\qquad
    L:=\ell_{k+1}^+(h),\qquad
    D:=\Des_{>k+1}(h),\qquad
    Y:=Y_{k+1}(r,T).
  \]
  By definition,
  \[
    \mathcal K_n^{k+1,r}\big|_{hS_{k+1}}
    =
    B_{k+1,r}^{(u)}q^{-2L}\mathbf X^D T^u.
  \]
  Using \Cref{lem:K-block-structure}, the \(k\)-level block is
  \[
    \mathcal K_n^{k,r}\big|_{hS_{k+1}}
    =
    q^{-2L}\mathbf X^D
    \sum_{t=0}^{k}
    q^{-2t}B_{k,r}^{(t)}
    X_{k+1}^{\mathbf 1_{t<u}}T^t.
  \]
  Insert this expression, and the analogous one with \(r'\), into the
  numerator of the claimed recursion. By \Cref{lem:crossdiff-with-T}, the
  numerator of \eqref{eq:K-block-recursion} becomes
  \[
    q^{-2L}\mathbf X^D(q^{-r'}-q^{-r})
    \sum_{t=0}^{k}
    q^{-2t}[t]_{q^2}![k-t]_{q^2}!
    q^{-t(t-1)}
  \]
  \[
    {}\times
    \Bigl(
      (q^{2t}-q^{2k+2})e_{k+1,r}^{(t)}
      +(q^{2t+2}-1)e_{k+1,r}^{(t+1)}T
    \Bigr)
    Y^{\mathbf 1_{t<u}}T^t.
  \]
  The elementary identities
  \[
    (q^{2t+2}-1)[t]_{q^2}!=-(1-q^2)[t+1]_{q^2}!,
  \]
  \[
    [k-t]_{q^2}!(q^{2t}-q^{2k+2})
    =
    q^{2t}(1-q^2)[k-t+1]_{q^2}!
  \]
  rewrite this as
  \[
    q^{-2L}\mathbf X^D(q^{-r'}-q^{-r})(1-q^2)
    \sum_{t=0}^{k}
    \bigl(B_{k+1,r}^{(t)}-B_{k+1,r}^{(t+1)}T\bigr)
    Y^{\mathbf 1_{t<u}}T^t.
  \]
  The sum telescopes:
  \[
    \sum_{t=0}^{k}
    (B^{(t)}-B^{(t+1)}T)Y^{\mathbf 1_{t<u}}T^t
    =
    YB^{(0)}+(1-Y)B^{(u)}T^u-B^{(k+1)}T^{k+1}.
  \]
  Here \(B^{(t)}=B_{k+1,r}^{(t)}\). Since
  \(B^{(0)}=[k+1]_{q^2}!\) and, by \Cref{lem:top-coeff-ek+1},
  \[
    B^{(k+1)}T^{k+1}
    =
    [k+1]_{q^2}!
    q^{\frac{rr'}2-(k+1)(k+2)}T^{k+1}
    =
    YB^{(0)},
  \]
  the first and last terms cancel. The numerator is therefore
  \[
    (q^{-r'}-q^{-r})(1-q^2)(1-Y)
    B_{k+1,r}^{(u)}q^{-2L}\mathbf X^D T^u,
  \]
  which is exactly the denominator times
  \(\mathcal K_n^{k+1,r}\big|_{hS_{k+1}}\).

  Summing over \(hS_{k+1}\in S_n/S_{k+1}\) gives the recursion for the total
  sums.
\end{proof}

\begin{theorem}\label{thm:J=K}
  Let \(k\in[n]_0\) and \(r\in\Z\). Then
  \[
    \mathcal J_n^{k,r}(X_{k+1},\dots,X_n;T)
    =
    \mathcal K_n^{k,r}(X_{k+1},\dots,X_n;T).
  \]
\end{theorem}

\begin{proof}
  Induct on \(k\). The case \(k=0\) is \Cref{lem:J=K-base}. Assume the claim
  at level \(k<n\). By \Cref{lem:Jn-recursion,prop:K-block-recursion}, both
  \(\mathcal J_n^{k+1,r}\) and \(\mathcal K_n^{k+1,r}\) are obtained from the
  \(k\)-level entries with parameters \(r\) and \(2k+3-r\) by the same
  recursion. The induction hypothesis therefore gives equality at level
  \(k+1\).
\end{proof}

\begin{proof}[Proof of \Cref{thm:I-equals-K}]
  Immediate from \Cref{def:Jnkr,thm:J=K}.
\end{proof}

\section{Proof of \Cref{thm:C}}\label{sec:Bn}

In this section we define type-\(\mathsf{B}\) Igusa functions and express the
zeta function \(\zeta_{\mfh_n(\lri)}(s)\) as a specialization of these
functions; cf.\ \Cref{thm:Bn-target}(=\Cref{thm:C})
Let \(B_n\) denote the hyperoctahedral group of rank \(n\), i.e.\ the Coxeter
group of type \(\mathsf{B}_n\), realized as the group of signed permutations of
\(\{\pm1,\dots,\pm n\}\). Let \(S=\{s_0,s_1,\dots,s_{n-1}\}\) be the standard
generating set, where \(s_0\) changes the sign of \(1\) and \(s_i=(i,i+1)\) for
\(i\in[n-1]\). For \(g\in B_n\), we set \(g(0)=0\). We utilize a number of classical
statistics on~\(B_n\):
\begin{itemize}
  \item \(\ell(g)\) is the Coxeter length of \(g\) with respect to \(S\),
  \item \(\Des_B(g)=\{\,i\in[n-1]_0\mid g(i)>g(i+1)\,\}\) is the (type-\(\mathsf{B}\)) descent set of~\(g\),
  \item \(\des_B(g):=|\Des_B(g)|\) is the number of type-\(\mathsf{B}\) descents of \(g\),
  \item \(\negg(g)\) is the number of negative entries of \(g\) in one-line notation.
\end{itemize}

Let \(Z\), \(Y\), and \(X_0,X_1,\dots\) be indeterminates.

\begin{definition}[Type-\(\mathsf{B}\) Igusa functions]\label{def:igusa.B}
  For \(n\geq 1\), we define the \emph{type-\(\mathsf{B}\) Igusa function}
  \[
    \ig{B_n}(Y,Z;X_0,\dots,X_n)
    :=
    \frac{
      \displaystyle
      \sum_{g\in B_n}
      Y^{\ell(g)}\,Z^{\negg(g)}
      \prod_{i\in\DesB(g)} X_i
    }{
      \displaystyle
      \prod_{i=0}^{n}(1-X_i)
    },
  \]
  and the \emph{truncated type-\(\mathsf{B}\) Igusa function}
  \[
    \igm{B_n}(Y,Z;X_0,\dots,X_{n-1})
    =(1-X_n)\ig{B_n}(Y,Z;X_0,\dots,X_n).
  \]
\end{definition}

Recall that we set \(c_i=\frac{n(n+5)-i(i+1)}{2}\) for \(i\in[n]\). For \(g\in B_n\), set
\begin{equation}\label{eq:Cg}
    C(g)
    =
    n\,\negg(g)
    -\ell(g)
    +\sum_{i\in\DesB(g)}c_i.
\end{equation}

\begin{theorem}\label{thm:Bn-target}
  We have
  \begin{equation}\label{eq:Bn-Igusa-zeta}
    \zeta_{\mfh_n(\lri)}(s)
    =
    \frac{1}{(T;q)_{2n}}\,
    \ig{B_n}\!\left(
      q^{-1},-q^nT;\,
      q^{c_0}T^{n+1},\dots,q^{c_n}T^{n+1}
    \right).
  \end{equation}
  Equivalently,
  \begin{equation*}
    \zeta_{\mfh_n(\lri)}(s)
    =
    \frac{
      \displaystyle
      \sum_{g\in B_n}
      (-1)^{\negg(g)}\,
      q^{C(g)}\,
      T^{(n+1)\desB(g)+\negg(g)}
    }{
      \displaystyle
      (T;q)_{2n}
      \prod_{m=0}^{n}
      \bigl(1-q^{c_m}T^{n+1}\bigr)
    }
  \end{equation*}
\end{theorem}
Let
\[
  \zeta^{\triangleleft}_{\mfh_n(\lri)}(s) :=
  \sum_{I\triangleleft \mfh_n} |\mfh_n:I|^{-s}
\]
denote the local ideal zeta function of \(\mfh_n\). By the same argument as in
\cite[\S 8.1]{GSS/88} we have
\begin{equation}\label{equ:Heisenberg.normal}
  \zeta^{\triangleleft}_{\mfh_n(\lri)}(s) =
  \frac{1}{(T;q)_{2n}(1-q^{2n}T^{n+1})}.
\end{equation}

\begin{cor}\label{cor:subalg-vs-normal-typeB}
   We have
  \[
    \zeta_{\mfh_n(\lri)}(s)
    =
    \zeta^{\triangleleft}_{\mfh_n(\lri)}(s)\,
    \igm{B_n}\!\left(
      q^{-1},-q^nT;\,
      q^{c_0}T^{n+1},\dots,q^{c_{n-1}}T^{n+1}
    \right).
  \]
\end{cor}

\begin{proof}
  Since \(c_n=2n\), the identity
  \(
    \igm{B_n}(Y,Z;X_0,\dots,X_{n-1})
    =
    (1-X_n)\,\ig{B_n}(Y,Z;X_0,\dots,X_n)
  \)
  gives
  \[
    \igm{B_n}\!\left(
      q^{-1},-q^nT;\,
      q^{c_0}T^{n+1},\dots,q^{c_{n-1}}T^{n+1}
    \right)
    =
    (1-q^{2n}T^{n+1})\,
    \ig{B_n}\!\left(
      q^{-1},-q^nT;\,
      q^{c_0}T^{n+1},\dots,q^{c_n}T^{n+1}
    \right).
  \]
  The result follows from \Cref{eq:Bn-Igusa-zeta}.
\end{proof}

The proof of \Cref{thm:Bn-target} occupies the remainder of this section. In
\Cref{subsec:residue-reduction} we reduce \Cref{eq:Bn-Igusa-zeta} to the
residue identity \eqref{eq:residue-identity}. In
\Cref{subsec:Bn-Igusa} we study the subset expansion and residue
factorization for type-\(\mathsf{B}\) Igusa functions. In
\Cref{subsec:q-hypergeometric-Igusa} we establish the
\(q\)-hypergeometric identities needed to evaluate the resulting factors.
Finally, in \Cref{subsec:residue-proof} we combine these evaluations to prove
\eqref{eq:residue-identity}, and hence \Cref{thm:Bn-target}.

\subsection{Reduction to a residue identity}\label{subsec:residue-reduction}

We begin the proof of \Cref{thm:Bn-target} by introducing an auxiliary
variable \(Y\). This reduces the desired identity to a residue computation. 

Define
\begin{equation}\label{eq:ZTY}
  Z(T,Y)
  :=
  \sum_{r=0}^{n}
  \frac{1}{1-q^{\,2n+\frac{r(2n+1-r)}{2}}\,Y}
  \cdot
  \frac{
    (-q)^r(1-q^{2n-2r+1})(q^2;q^2)_n
  }{
    (q;q)_{2n-r+1}(q;q)_r
    (q^{r}T;q^{2})_{n-r}
    (q^{2n-r}T;q)_{r}
  }.
\end{equation}
Then \(\zeta_{\mfh_n(\lri)}(s)=Z(T,T^{n+1})\) by
\Cref{thm:simplified-zeta}. Likewise, set
\begin{equation*}
  Z_2(T,Y)
  :=
  \frac{1}{(T;q)_{2n}}\,
  \ig{B_n}\!\left(
    q^{-1},-q^nT;\,
    q^{c_0}Y,\dots,q^{c_n}Y
  \right).
\end{equation*}
Thus the right-hand side of \eqref{eq:Bn-Igusa-zeta} is the specialization at
\(Y=T^{n+1}\) of \(Z_2(T,Y)\).  It suffices to show
\(Z(T,Y)=Z_2(T,Y)\). Substituting \(r=n-m\) into \eqref{eq:ZTY} gives
\begin{equation*}
  Z(T,Y)=\sum_{m=0}^{n}\frac{\mathcal A_m(T)}{1-q^{c_m}Y},
\end{equation*}
where
\begin{equation*}
  \mathcal A_m(T)
  :=
  \frac{
    (-q)^{n-m}(1-q^{2m+1})(q^2;q^2)_n
  }{
    (q;q)_{n+m+1}(q;q)_{n-m}
    (q^{\,n-m}T;q^2)_m
    (q^{\,n+m}T;q)_{n-m}
  }.
\end{equation*}

On the other hand,
\[
  Z_2(T,Y)
  =\frac{\ig{B_n}\!\left(
      q^{-1},-q^nT;\,
      q^{c_0}T^{n+1},\dots,q^{c_n}T^{n+1}
    \right)}{(T;q)_{2n}}
    =
  \frac{
    \sum_{g\in B_n}
    (-T)^{\negg(g)}q^{C(g)}Y^{\desB(g)}
  }{
    (T;q)_{2n}\prod_{m=0}^n(1-q^{c_m}Y)
  }.
\]

Thus \(Z(T,Y)\) and \(Z_2(T,Y)\) have the same possible simple poles
in \(Y\), namely at \(Y=q^{-c_m}\) for \(m\in[n]_0\). Since both
rational functions have no polynomial part in \(Y\), it suffices to
compare their residues at these poles. Therefore \Cref{thm:Bn-target}
follows once we prove
\[
  \lim_{Y\to q^{-c_m}}(1-q^{c_m}Y)Z_2(T,Y)=\mathcal A_m(T)
  \qquad(m\in[n]_0).
\]
Equivalently, with
\begin{equation}\label{eq:Lnm}
  \mathcal L_{n,m}
  :=
  \lim_{Y\to q^{-c_m}}
  (1-q^{c_m}Y)\,
  \ig{B_n}\!\left(
    q^{-1},-q^nT;\,
    q^{c_0}Y,\dots,q^{c_n}Y
  \right),
\end{equation}
it remains to prove the residue identities
\begin{equation}\label{eq:residue-identity}
  \mathcal L_{n,m}=(T;q)_{2n}\mathcal A_m(T)
  \qquad \textup{ for } m\in[n]_0.
\end{equation}

In the following subsection we prove a residue factorization for
\(\ig{B_n}\), which will be used to evaluate \(\mathcal L_{n,m}\).

\subsection{Subset expansion and residue factorization for type-\(\mathsf{B}\) Igusa functions}\label{subsec:Bn-Igusa}

We first present a subset expansion formula for \(\ig{B_n}\) and
\(\igm{B_n}\).

\begin{prop}[Subset expansions for type-\(\mathsf{B}\) Igusa functions]
  \label{prop:IBn-subset-expansion}
  We have   \begin{equation}\label{eq:IBn-explicit-subset}
    \ig{B_n}(Y,Z;X_0,\dots,X_n)
    =
    \sum_{I\subseteq [n]_0}
    \binom{n}{I}_{Y}\,
    \bigl(-Y^{n}Z;\,Y^{-1}\bigr)_{\,n-\min(I\cup\{n\})}
    \prod_{i\in I}\frac{X_i}{1-X_i}.
  \end{equation}   \begin{equation}\label{eq:IBnminus-explicit-subset}
    \igm{B_n}(Y,Z;X_0,\dots,X_{n-1})
    =
    \sum_{I\subseteq [n-1]_0}
    \binom{n}{I}_{Y}\,
    \bigl(-Y^{n}Z;\,Y^{-1}\bigr)_{\,n-\min(I\cup\{n\})}
    \prod_{i\in I}\frac{X_i}{1-X_i}.
  \end{equation}
\end{prop}

\begin{proof}
  For every \(S\subseteq[n]_0\), set \(X^S:=\prod_{i\in S}X_i\). Then
  \begin{equation}\label{eq:superset-expansion}
    \frac{X^S}{\displaystyle\prod_{i=0}^{n}(1-X_i)}
    =
    \sum_{\substack{I\subseteq[n]_0\\ S\subseteq I}}
    \prod_{i\in I}\frac{X_i}{1-X_i}.
  \end{equation}
  Substituting \eqref{eq:superset-expansion} with \(S=\DesB(g)\) into
  \Cref{def:igusa.B} yields
    \begin{align*}
      \ig{B_n}(Y,Z;X_0,\dots,X_n)
      &=
      \sum_{g\in B_n}Y^{\ell(g)}Z^{\negg(g)}
      \sum_{\substack{I\subseteq[n]_0\\ \DesB(g)\subseteq I}}
      \prod_{i\in I}\frac{X_i}{1-X_i} \\
      &=
      \sum_{I\subseteq[n]_0}
      \left(
        \sum_{\substack{g\in B_n\\ \DesB(g)\subseteq I}}
        Y^{\ell(g)}Z^{\negg(g)}
      \right)
      \prod_{i\in I}\frac{X_i}{1-X_i}.
    \end{align*}
    By \cite[Lemma~4.5]{SV/14}, equivalently
      \cite[Lemma~3.1]{reiner1993}, the inner sum is
      \begin{equation}\label{eq:Reiner-lemma-our-form}
        \sum_{\substack{g\in B_n\\ \DesB(g)\subseteq I}}
        Y^{\ell(g)}Z^{\negg(g)}
        =
        \binom{n}{I}_{Y}
        (-Y^{n}Z;Y^{-1})_{n-\min(I\cup\{n\})}.
      \end{equation}
      This proves \eqref{eq:IBn-explicit-subset}.
   The proof of
  \eqref{eq:IBnminus-explicit-subset} is identical, restricting to
  subsets \(I\subseteq[n-1]_0\).
\end{proof}

\begin{remark}
  The evaluation in \eqref{eq:Reiner-lemma-our-form} ultimately comes
  from the parabolic factorization of the bivariate generating
  function $\Phi(Y,Z;M):=\sum_{g\in M} Y^{\ell(g)}Z^{\negg(g)}$ over
  suitable subsets \(M\subseteq B_n\); compare
  \cite[Lem.~3.1]{reiner1993}.
\end{remark}

We next factor the residues needed for \(\mathcal L_{n,m}\).

\begin{definition}[\(m\)-th type-\(\mathsf{B}\) Igusa residue]\label{def:IBn-residue}
  For \(m\in[n]_0\), define
  \[
    \Res_m \ig{B_n}(Y,Z;X_0,\dots,\widehat{X_m},\dots,X_n)
    :=
    \lim_{X_m\to 1}
    (1-X_m)\,
    \ig{B_n}(Y,Z;X_0,\dots,X_n).
  \]
\end{definition}

Then by \eqref{eq:Lnm}, we have
\begin{equation}\label{eq:Lnm-as-IBn-residue}
  \mathcal L_{n,m}
  =
  \left.
  \Res_m \ig{B_n}(Y,Z;X_0,\dots,\widehat{X_m},\dots,X_n)
  \right|_{
    \substack{
      Y=q^{-1},\ Z=-q^nT,\\
      X_i=q^{c_i-c_m}\ (i\ne m)
    }
  }.
\end{equation}

\begin{prop}[Residue factorization for \(\ig{B_n}\)]
  \label{prop:IBn-residue-factorization}
  For \(m\in[n]_0\),
  \begin{multline*}
    \Res_m \ig{B_n}(Y,Z;X_{\ne m})
    =
    \binom{n}{m}_{Y}
    (-Y^{n}Z;Y^{-1})_{n-m}  \\
    {}\times
    \igm{B_m}(Y,Z;X_0,\dots,X_{m-1})
    \ig{n-m}(Y;X_{m+1},\dots,X_n).
  \end{multline*}
\end{prop}

\begin{proof}
  Multiplying \eqref{eq:IBn-explicit-subset} by \(1-X_m\) and letting \(X_m\to 1\), only the terms with
  \(m\in I\) survive. Thus
  \[
    \Res_m \ig{B_n}
    =
    \sum_{\substack{I\subseteq [n]_0\\ m\in I}}
    \binom{n}{I}_{Y}\,
    \bigl(-Y^{n}Z;\,Y^{-1}\bigr)_{\,n-\min(I\cup\{n\})}
    \prod_{i\in I\setminus\{m\}}\frac{X_i}{1-X_i}.
  \]

  Write \(I=I_1\sqcup\{m\}\sqcup I_2\), where
  \(I_1\subseteq [m-1]_0\) and \(I_2\subseteq \{m+1,\dots,n\}\). Then
  \[
    \min(I\cup\{n\})=\min(I_1\cup\{m\}), \quad 
    \binom{n}{I}_{Y}
    =
    \binom{n}{m}_{Y}
    \binom{m}{I_1}_{Y}
    \binom{n-m}{I_2-m}_{Y},
  \]
  and
  \[
    (-Y^nZ;Y^{-1})_{n-\min(I_1\cup\{m\})}
    =
    (-Y^nZ;Y^{-1})_{n-m}
    (-Y^mZ;Y^{-1})_{m-\min(I_1\cup\{m\})}.
  \]
  Then \(\Res_m \ig{B_n}\) factors as
  \[
    \binom{n}{m}_{Y}
    (-Y^nZ;Y^{-1})_{n-m}
    \left(
      \sum_{I_1\subseteq[m-1]_0}
      \binom{m}{I_1}_{Y}
      (-Y^mZ;Y^{-1})_{m-\min(I_1\cup\{m\})}
      \prod_{i\in I_1}\frac{X_i}{1-X_i}
    \right)
  \]
  \[
    {}\times
    \left(
      \sum_{I_2\subseteq\{m+1,\dots,n\}}
      \binom{n-m}{I_2-m}_{Y}
      \prod_{i\in I_2}\frac{X_i}{1-X_i}
    \right).
  \]
  The two parenthesized sums are
  \(\igm{B_m}(Y,Z;X_0,\dots,X_{m-1})\) and
  \(\ig{n-m}(Y;X_{m+1},\dots,X_n)\), respectively.
\end{proof}

Combining \eqref{eq:Lnm-as-IBn-residue} with
\Cref{prop:IBn-residue-factorization} yields
\begin{align}
  \mathcal L_{n,m}
  &=
  \binom{n}{m}_{q^{-1}}
  (T;q)_{n-m}\,
  \igm{B_m}\!\left(
    q^{-1},-q^nT;
    \left(
      q^{\frac{m(m+1)-i(i+1)}{2}}
    \right)_{i=0}^{m-1}
  \right) \notag\\
  &\qquad\times
  \ig{n-m}\!\left(
    q^{-1};
    q^{-m-1},
    q^{-2m-3},
    \dots,
    q^{\frac{m(m+1)-n(n+1)}{2}}
  \right).
  \label{eq:Lnm-factorized}
\end{align}
The next subsection evaluates the two Igusa factors in
\eqref{eq:Lnm-factorized}.

\subsection{\(q\)-hypergeometric identities for Igusa functions}\label{subsec:q-hypergeometric-Igusa}

We collect the specializations of Igusa functions needed for
\eqref{eq:residue-identity}. We use the classical \(q\)-Chu--Vandermonde summation (see \cite[Eq.~(1.5.2)]{GasperRahman2009})
\begin{equation}\label{eq:q-Chu-Vandermonde}
  {}_2\phi_1\!\left[
    \begin{matrix}
      Q^{-N},\, b\\
      c
    \end{matrix}
    ;Q,\frac{cQ^N}{b}
  \right]
  =
  \sum_{r=0}^N
  \frac{(Q^{-N};Q)_r(b;Q)_r}{(c;Q)_r(Q;Q)_r}
  \left(\frac{cQ^N}{b}\right)^r
  =
  \frac{(c/b;Q)_N}{(c;Q)_N}
\end{equation}

\subsubsection{Induction for type-\(\mathsf{A}\) Igusa functions}
The type-\(\mathsf{A}\) Igusa functions satisfy the following recursion, expressing \((1-X_n)\ig{n}\) in terms of lower-degree Igusa functions.

\begin{lemma}\label{lem:Igusa-induction}
  For \(n\ge 0\), with \(\ig{0}(Y;-):=1\) and \(X_0:=1\),
  \[
    \ig{n}(Y;X_1,\dots,X_n)
    =
    \sum_{j=0}^{n}
    \binom{n}{j}_Y X_j\,\ig{j}(Y;X_1,\dots,X_j).
  \]
\end{lemma}

\begin{proof}
  Partition the subset expansion of \(\ig{n}\) from \Cref{def:igusa.A}
  according to \(j=\max(I\cup\{0\})\). For such \(I\),
  \(\binom{n}{I}_Y=\binom{n}{j}_Y\binom{j}{I}_Y\). Hence the \(j\)-th part is
  \begin{align*}
    \binom{n}{j}_Y
    \sum_{\substack{I\subseteq[j]\\ \max(I\cup\{0\})=j}}
    \binom{j}{I}_Y
    \prod_{i\in I}\frac{X_i}{1-X_i}
    &=
    \binom{n}{j}_Y
    \bigl(\ig{j}(Y;X_1,\dots,X_j)
    -\igm{j}(Y;X_1,\dots,X_{j-1})\bigr) \\
    &=
    \binom{n}{j}_Y X_j\,\ig{j}(Y;X_1,\dots,X_j),
  \end{align*}
  where the last equality follows from
  \Cref{lem:bridge-Igusa}. Summing over \(j\) gives the claim.
\end{proof}

\begin{lemma}\label{lem:Igusa-specialization}
  For \(n\ge0\),
  \begin{equation*}
    \ig{n}\!\left(q^{-1};\left(q^{-\binom{r+1}{2}}U^r\right)_{r=1}^n\right)    =
    \frac{(-q^{-1}U;q^{-1})_n}{(q^{-2}U^2;q^{-1})_n}.
  \end{equation*}
\end{lemma}

\begin{proof}
  Set \(I(n,q,U):=\ig{n}\!\left(q^{-1};\left(q^{-\binom{r+1}{2}}U^r\right)_{r=1}^n\right)\).
  We argue by induction on \(n\). For \(n=0\), both sides equal \(1\).
  For \(n\ge 1\), by \Cref{lem:Igusa-induction} and the induction hypothesis, 
  \[
    I(n,q,U)
    =
    \sum_{j=0}^{n}
    \binom{n}{j}_{q^{-1}}
    q^{-j(j+1)/2}U^j\,
    \frac{(-q^{-1}U;q^{-1})_j}{(q^{-2}U^2;q^{-1})_j}.
  \]
  Using
  \[
    \binom{n}{j}_{q^{-1}}
    =
    \frac{(q^{-n};q)_j}{(q^{-1};q^{-1})_j}
    =
    (-1)^j q^{-nj+\binom{j}{2}}
    \frac{(q^{n};q^{-1})_j}{(q^{-1};q^{-1})_j},
  \]
  we obtain
  \begin{align*}
    I(n,q,U)
    &=
    \sum_{j=0}^{n}
    \frac{(q^{n};q^{-1})_j\,(-q^{-1}U;q^{-1})_j}
    {(q^{-2}U^2;q^{-1})_j\,(q^{-1};q^{-1})_j}\,
    (-q^{-n-1}U)^j \\
    &=
    {}_2\phi_1\!\left[
      \begin{matrix}
        q^{n},\, -q^{-1}U\\
        q^{-2}U^2
      \end{matrix}
      ;q^{-1},-q^{-n-1}U
    \right]     =
    \frac{(-q^{-1}U;q^{-1})_n}{(q^{-2}U^2;q^{-1})_n},
  \end{align*}
  by
  \eqref{eq:q-Chu-Vandermonde}, with \(Q=q^{-1}\),
  \(b=-q^{-1}U\), and \(c=q^{-2}U^2\).
\end{proof}

\begin{remark}\label{rem:SV14.1}
  \Cref{lem:Igusa-specialization} is a triangular-number analogue of the
  square-number specialization implicit in \cite[Prop.~4.2]{SV/14}.
  In our notation, that result gives
  \[
    \igm{n}\!\left(X^{-1};(X^{i(n-i)}Z^i)_{i=1}^{n-1}\right)
    =
    \frac{1-Z^n}{(Z;X)_n}.
  \]
  Since \(\igm{n}=(1-X_n)\ig{n}\), the substitution \(Z=X^{-n}U\) yields
  \begin{equation*}
    \ig{n}\!\left(X^{-1};X^{-i^2}U^i\ (1\le i\le n)\right)
    =
    \frac{1}{(X^{-n}U;X)_n}
    =
    \frac{1}{(X^{-1}U;X^{-1})_n}.
  \end{equation*}
\end{remark}

\subsubsection{Relation between truncated type-\(\mathsf{B}\) and type-\(\mathsf{A}\) Igusa functions}

\begin{lemma}\label{lem:IBnminus-to-Igusa}
  For \(n\ge 0\), with \(X_n:=1\), we have
  \begin{equation}\label{eq:IBnminus-to-Igusa}
    \igm{B_n}(Y,Z;X_0,\dots,X_{n-1})
    =
    \sum_{j=0}^n
    \binom{n}{j}_Y\,
    X_{n-j}\,
    \ig{j}\bigl(Y;X_{n-1},X_{n-2},\dots,X_{n-j}\bigr)\,
    (-Y^nZ;Y^{-1})_j.
  \end{equation}
\end{lemma}
\begin{proof}
  Starting from \eqref{eq:IBnminus-explicit-subset}, we partition the sum over \(I\subseteq [n-1]_0\) according to
  \(r=\min(I\cup\{n\})\). For fixed \(r\),
  the \(r\)-th part is
  \begin{align*}
    (-Y^nZ;Y^{-1})_{n-r}
    \sum_{\substack{I\subseteq [n-1]_0\\ \min(I\cup\{n\})=r}}
    \binom{n}{I}_Y
    \prod_{i\in I}\frac{X_i}{1-X_i}
  \end{align*}

  If \(r=n\), then \(I=\varnothing\), and the sum equals~\(1\). If \(r\in[n-1]_0\), set \(J=I\setminus\{r\}\). Then
  \(\binom{n}{I}_Y=\binom{n}{r}_Y\binom{n-r}{J-r}_Y\), and hence the
  sum is
  \begin{align*}
    \binom{n}{r}_Y
    \frac{X_r}{1-X_r}
    \sum_{J\subseteq\{r+1,\dots,n-1\}}
    \binom{n-r}{J-r}_Y
    \prod_{i\in J}\frac{X_i}{1-X_i}
    &=
    \binom{n}{r}_Y
    \frac{X_r}{1-X_r}\,
    \igm{n-r}(Y;X_{r+1},\dots,X_{n-1}) \\
    &=
    \binom{n}{r}_Y
    X_r\,
    \ig{n-r}(Y;X_{n-1},\dots,X_{r+1},X_r),
  \end{align*}
  where the last equality uses \Cref{lem:Igusa-reversal,lem:bridge-Igusa}. Hence
  \[
    \igm{B_n}(Y,Z;X_0,\dots,X_{n-1})
    =
    \sum_{r=0}^n
    \binom{n}{r}_Y\,
    X_r\,
    \ig{\,n-r\,}(Y;X_{n-1},\dots,X_{r+1},X_r)\,
    (-Y^nZ;Y^{-1})_{n-r}.
  \]
  Substituting \(j=n-r\) gives \eqref{eq:IBnminus-to-Igusa}.
\end{proof}

\subsubsection{A specialization of truncated type-\(\mathsf{B}\) Igusa functions}
\begin{prop}\label{prop:IBnminus-specialization}
  Let \(k\in\N_0\). Then
  \begin{equation*}
    \igm{B_k}\!\left(
      q^{-1},Z;
      \left(
        X_r=q^{\frac{k(k+1)}{2}-\frac{r(r+1)}{2}}
      \right)_{r=0}^{k-1}
    \right)
    =
    \frac{(-q^{1-k}Z;q^2)_k}{(q;q^2)_k}.
  \end{equation*}
\end{prop}

\begin{proof}
  Put
  \(
    F_k(Z):=
    \igm{B_k}\!\left(
      q^{-1},Z;
      X_r
    \right).
  \)
  Since \(X_{k-j}=q^{-\frac{j(j+1)}2}q^{(k+1)j}\), \Cref{lem:IBnminus-to-Igusa}
  and \Cref{lem:Igusa-specialization}, with \(U=q^{k+1}\), give
  \[
    F_k(Z)=
    \sum_{j=0}^k
    \binom{k}{j}_{q^{-1}}
    q^{-\frac{j(j+1)}2}q^{(k+1)j}
    \frac{(-q^k;q^{-1})_j}{(q^{2k};q^{-1})_j}
    (-q^{-k}Z;q)_j .
  \]

  We compare coefficients. By the \(q\)-binomial theorem,
  \[
    [Z^\ell](-q^{-k}Z;q)_j
    =
    q^{\binom{\ell}{2}-k\ell}\binom{j}{\ell}_q .
  \]
  Hence, using
  \(\binom{k}{j}_{q^{-1}}=q^{-j(k-j)}\binom{k}{j}_q\) and
  \(\binom{k}{j}_q\binom{j}{\ell}_q
  =\binom{k}{\ell}_q\binom{k-\ell}{j-\ell}_q\), we obtain
  \[
  \begin{aligned}
    [Z^\ell]F_k(Z)
    &=
    q^{\binom{\ell}{2}-k\ell}\binom{k}{\ell}_q
    \sum_{s=0}^{k-\ell}
    \binom{k-\ell}{s}_q
    q^{\frac{(s+\ell)(s+\ell+1)}2}
    \frac{(-q^k;q^{-1})_{s+\ell}}
         {(q^{2k};q^{-1})_{s+\ell}}  \\
    &=
    q^{\binom{\ell}{2}-k\ell}\binom{k}{\ell}_q
    q^{\frac{\ell(\ell+1)}2}
    \frac{(-q^k;q^{-1})_\ell}{(q^{2k};q^{-1})_\ell}
    {}_2\phi_1\!\left[
      \begin{matrix}
        q^{k-\ell},\, -q^{k-\ell}\\
        q^{2k-\ell}
      \end{matrix}
      ;q^{-1},-q^\ell
    \right].
  \end{aligned}
  \]
  The \(q\)-Chu--Vandermonde summation~\eqref{eq:q-Chu-Vandermonde}, with base \(q^{-1}\), expresses this as

  \[
  \begin{aligned}
    [Z^\ell]F_k(Z)
    &=
    q^{\ell(\ell-k)}
    \binom{k}{\ell}_q
    \frac{(-q^k;q^{-1})_\ell(-q^k;q^{-1})_{k-\ell}}
         {(q^{2k};q^{-1})_k} =
    q^{\ell(\ell-k)}
    \frac{\binom{k}{\ell}_{q^2}}{(q;q^2)_k}.
  \end{aligned}
  \]

 On the other hand,
  \[
    (-q^{1-k}Z;q^2)_k
    =
    \sum_{\ell=0}^k
    q^{\ell(\ell-k)}
    \binom{k}{\ell}_{q^2} Z^\ell .
  \]
  Hence, \(F_k(Z)=(-q^{1-k}Z;q^2)_k/(q;q^2)_k\).
\end{proof}

\begin{remark}\label{rem:SV14.2}
  This specialization is a triangular-number analogue to
  \cite[Prop.~1.5]{SV/14}, which in our notation reads
  \[
    \igm{B_n}\!\left(
      X^{-1},-Y; \left((X^iZ)^{\,n-i}\right)_{i=0}^{n-1}     \right)
    =
    \frac{(X^{-n}YZ;X)_n}{(Z;X)_n}.
  \]
  Equivalently, writing \(q=X\), \(Z'=-Y\), and \(U=X^nZ\), one obtains
  \begin{equation*}
    \igm{B_n}\!\left(
      q^{-1},Z'; \left(q^{-(n-i)^2}U^{\,n-i}\right)_{i=0}^{n-1}     \right)
    =
    \frac{(-q^{-2n}Z'U;q)_n}{(q^{-n}U;q)_n}.
  \end{equation*}
\end{remark}

\subsection{Proof of the residue identity}\label{subsec:residue-proof}

We now prove \eqref{eq:residue-identity}.  Recall from
\Cref{subsec:residue-reduction} that
\begin{equation}
    \label{eq:AmT}
    (T;q)_{2n}\mathcal A_m(T)
  =
  \frac{
    (-q)^{n-m}(1-q^{2m+1})(q^2;q^2)_n
  }{
    (q;q)_{n+m+1}(q;q)_{n-m}
  }\cdot
  \frac{(T;q)_{2n}}
  {(q^{n-m}T;q^2)_m(q^{n+m}T;q)_{n-m}}.
\end{equation}
It remains to simplify the factorization of \(\mathcal L_{n,m}\)
from \eqref{eq:Lnm-factorized}:
\begin{align*}
  \mathcal L_{n,m}
  &=
  \binom{n}{m}_{q^{-1}}
  (T;q)_{n-m}\,
  \igm{B_m}\!\left(
    q^{-1},-q^nT;
    \left(
      q^{\frac{m(m+1)-i(i+1)}{2}}
    \right)_{i=0}^{m-1}
  \right) \\
  &\qquad\times\;
  \ig{n-m}\!\left(
    q^{-1};
    q^{-m-1},
    \dots,
    q^{\frac{m(m+1)-n(n+1)}{2}}
  \right).
\end{align*}
The type-\(\mathsf{B}\) factor is evaluated by
\Cref{prop:IBnminus-specialization}. For the type-\(\mathsf{A}\) factor, note
that, for \(1\le r\le n-m\),
\[
  q^{\frac{m(m+1)-(m+r)(m+r+1)}2}
  =
  q^{-mr-\frac{r(r+1)}2}.
\]
Hence \Cref{lem:Igusa-specialization}, with \(U=q^{-m}\), together with
\Cref{prop:IBnminus-specialization}, gives
\begin{equation}\label{eq:Lnm-after-Igusa-evaluation}
  \mathcal L_{n,m}
  =
  \binom{n}{m}_{q^{-1}}
  (T;q)_{n-m}\,
  \frac{(q^{n-m+1}T;q^2)_m}{(q;q^2)_m}\,
  \frac{(-q^{-m-1};q^{-1})_{n-m}}
       {(q^{-2m-2};q^{-1})_{n-m}}.
\end{equation}

Set
\[
  Q_{n,m}:=
  \binom{n}{m}_{q^{-1}}
  \left(\frac{1}{(q;q^2)_m}
  \frac{(-q^{-m-1};q^{-1})_{n-m}}
       {(q^{-2m-2};q^{-1})_{n-m}}\right).
\]
Then \(Q_{n,m}\) simplifies as
\begin{align}
  Q_{n,m}
  &=
  (-q)^{n-m}
  \frac{(q;q)_n}{(q;q)_m(q;q)_{n-m}}\,
  \left(\frac{(q^2;q^2)_m}{(q;q)_{2m}}\,
  \frac{(-q;q)_n}{(-q;q)_m}\,
  \frac{(q;q)_{2m+1}}{(q;q)_{n+m+1}}\right) \notag  \\
  &=
  \frac{(-q)^{n-m}(1-q^{2m+1})(q^2;q^2)_n}
       {(q;q)_{n+m+1}(q;q)_{n-m}}\label{eq:Qnm}.
\end{align}

The \(T\)-dependent part is
\begin{align}
  (T;q)_{n-m}(q^{n-m+1}T;q^2)_m
  &=
  (T;q)_{n-m}
  \frac{(q^{n-m}T;q)_{2m}}{(q^{n-m}T;q^2)_m}\notag  \\
  &=
  \frac{(T;q)_{2n}}
       {(q^{n-m}T;q^2)_m(q^{n+m}T;q)_{n-m}}\label{eq:T-part}.
\end{align}

Combining \eqref{eq:Lnm-after-Igusa-evaluation}, \eqref{eq:Qnm},
\eqref{eq:T-part}, and \eqref{eq:AmT}, we obtain
\[
  \mathcal L_{n,m}=(T;q)_{2n}\mathcal A_m(T).
\]
This proves \eqref{eq:residue-identity}, and hence \Cref{thm:Bn-target}.

\section{Local poles and functional equations}\label{sec:local.props}

We record some structural consequences of the explicit formulas we
gave for the local zeta functions~\(\zeta_{\mfh_n(\lri)}(s)\). In
\Cref{subsec:poles} we consider the poles of the rational
functions~\(\zeta_{\mfh_n(\lri)}(s)\). The local functional equations
in \Cref{cor:funeq} are proved in~\Cref{subsec:funeq}.

\subsection{Local poles}\label{subsec:poles}

\Cref{thm:simplified-zeta} or \Cref{thm:Bn-target} allow us to
determine the real poles of the local zeta
functions~$\zeta_{\mfh_n(\lri)}(s)$. For $r\in[n]_0$, set $$
a_{n,r}:=2n+\frac{r(2n+1-r)}2, \quad \alpha_r = a_{n,r}/(n+1).$$ We
define sets of \emph{integral} resp.\ \emph{fractional} pole
candidates \[ \mathcal P_n^{\mathrm{int}}:= [2n-1]_0, \qquad \mathcal
P_n^{\mathrm{fra}}:= \left\{\alpha_r \mid r \in [n]_0 \right\}.
\]

\begin{cor}
The real poles of \(\zeta_{\mfh_n(\lri)}(s)\) lie in \(\mathcal
P_n^{\mathrm{int}}\cup\mathcal P_n^{\mathrm{fra}}\).
\end{cor}

\begin{proof}
  Write \(T=q^{-s}\).
  By \Cref{thm:simplified-zeta},
  \(\zeta_{\mfh_n(\lri)}(s)=\sum_{r=0}^n S_r(T)\), where
  \[
    S_r(T)=
    \frac{C_r(q)}
    {(1-q^{a_{n,r}}T^{n+1})(q^{r}T;q^{2})_{n-r}(q^{2n-r}T;q)_{r}}
  \]
  for a fractional expression \(C_r(q)\) in~\(q\).
 The factors in
\((q^{r}T;q^{2})_{n-r}(q^{2n-r}T;q)_r\) give the integral pole
candidates, while the ``special factor'' \(1-q^{a_{n,r}}T^{n+1}\)
yields \(s=\alpha_r\).
\end{proof}

For later use, note that the leading \(q\)-term of \(C_r(q)\)
is~\((-1)^{n+r}q^{-(n-r)^2}\).

\begin{prop}\label{prop:local-real-poles}
  For generic \(q\), every number in \(\mathcal
  P_n^{\mathrm{int}}\cup\mathcal P_n^{\mathrm{fra}}\) is a pole of
  \(\zeta_{\mfh_n(\lri)}(s)\). A double pole occurs precisely at
  \(s=m\), where \(m\in[n]\) and \(m(m+1)=4n\). All other poles are
  simple.
\end{prop}

\begin{proof}
  \emph{(1) Double poles.}
  A repeated pole inside \(S_r(T)\) can occur only if the special pole
  \(T=q^{-\alpha_r}\) coincides with one of the integral roots \(T=q^{-j}\),
  where
  \[
    j\in
    \{r,r+2,\dots,2n-r-2\}
    \sqcup
    \{2n-r,2n-r+1,\dots,2n-1\}.
  \]
  Since \(a_{n,r}=(n+1)(r+2)-\left(2+\frac{r(r+1)}2\right)\),
  integrality of \(\alpha_r=a_{n,r}/(n+1)\) implies \(\alpha_r\le
  r+1\). Hence a doubling-up can only occur at \(\alpha_r=r\) or
  \(\alpha_r=r+1\). The latter would force \(r=n\), but then
  \(a_{n,n}=(n+1)^2\) is equivalent to \(n^2-n+2=0\), which is
  impossible. Thus a doubling-up occurs exactly when \(\alpha_r=r\),
  equivalently \(r(r+1)=4n\). In this case \(S_r\) has a double pole
  at \(T=q^{-r}\), while all other summands have at most simple poles
  there; hence the local zeta function has a double pole at \(s=r\).

  \emph{(2) Simple fractional poles.}
  Suppose that \(\alpha_r\in
  \mathcal P_n^{\mathrm{fra}}\setminus\mathcal P_n^{\mathrm{int}}\).
  For \(r'\ne r\), the special factor
  \(1-q^{a_{n,r'}}T^{n+1}\) does not vanish at \(T=q^{-\alpha_r}\). Since
  \(\alpha_r\notin\mathcal P_n^{\mathrm{int}}\), the remaining denominator
  factors do not vanish there. Hence \(S_{r'}(T)\) is holomorphic at \(T=q^{-\alpha_r}\) for
  all \(r'\ne r\), while \(S_r(T)\) contributes a simple nonzero pole. Thus
  \(s=\alpha_r\) is a simple pole of
  \(\zeta_{\mfh_n(\lri)}(s)\).
  
  \emph{(3) Integral poles.}
  Now fix \(i\in\mathcal P_n^{\mathrm{int}}\) such that \(T=q^{-i}\) is not
  a repeated pole. By (1), this implies \(i(i+1)\ne4n\). Set, for
  \(r\in[n]_0\),
  \[
    \Phi_{i,r}(q):=\lim_{T\to q^{-i}}(1-q^iT)S_r(T).
  \]
  For nonzero \(\Phi_{i,r}\), let
  \(L_{i,r}:=\deg_q\Phi_{i,r}\), where \(\deg_q\) means numerator degree
  minus denominator degree. We prove that
  \(\sum_{r=0}^n\Phi_{i,r}(q)\ne0\) by comparing the highest \(q\)-degrees,
  and, when necessary, the leading \(q\)-coefficients \(\operatorname{lc}_q\).
  Since \(T=q^{-i}\) is not a repeated pole, each nonzero residue
  \(\Phi_{i,r}(q)\) comes from exactly one denominator factor.

  \emph{(3a) Residues from the \(q\)-Pochhammer factors.}
    Suppose first that \(0\le i\le n-1\). Then \(1-q^iT\) can occur only in
  \((q^rT;q^2)_{n-r}\), so \(r=i-2t\). For \(t\ge1\), the special factor
  contributes no degree, since \(a_{n,i-2t}\le a_{n,i-2}\) and
  \begin{equation}\label{eq:key-degree-drop}
    a_{n,i-2}-(n+1)i
    =
    2n-\frac{i(i+1)}2-\bigl(2(n-i)+3\bigr)
    =
    \frac{-i^2+3i-6}{2}<0.
  \end{equation}
  A degree count gives
  \[
    L_{i,i}
    =
    -\binom{2n-i}{2}
    -
    \max\left\{2n-\frac{i(i+1)}2,0\right\},
    \qquad
    L_{i,i-2t}
    =
    -\binom{2n-i}{2}
    -t\bigl(2(n-i)+3t\bigr)
  \]
  for \(t\ge1\). Hence \(L_{i,i-2t}<L_{i,i}\): if the maximum in \(L_{i,i}\) is zero this is
  immediate; otherwise \eqref{eq:key-degree-drop} gives this already for
  \(t=1\), and larger \(t\) only decreases \(L_{i,i-2t}\).

  Now suppose \(n\le i\le2n-1\). The roots from
  \((q^rT;q^2)_{n-r}\) again have \(r=i-2t\), now with
  \(t\ge i-n+1\), and the same degree count shows that these contributions
  are smaller than the one at \(r=n\). It remains to compare the roots from
  \((q^{2n-r}T;q)_r\), where \(2n-i\le r\le n\). For these roots, the first
  \(q\)-Pochhammer factor has degree \(0\), while the second contributes the
  constant degree \(\binom{2n-i}{2}\). Hence, for \(2n-i\le r<n\),
  \[
    L_{i,r+1}-L_{i,r}
    =
    2(n-r)-1-\Delta_r,
  \]
  where
  \[
    \Delta_r
    :=
    \max\{0,a_{n,r+1}-(n+1)i\}
    -
    \max\{0,a_{n,r}-(n+1)i\}.
  \]
  Since \(a_{n,r+1}-a_{n,r}=n-r\), while
  \[
    a_{n,r}-(n+1)(2n-r)
    =
    \frac{3r-(2n-r)^2}{2}<0
  \]
  for \(r\leq n-1\), and since \(i\ge2n-r\), we have
  \(a_{n,r}<(n+1)i\). Hence
  \(\Delta_r<n-r\). Thus \(L_{i,r+1}>L_{i,r}\), and the largest
  Pochhammer contribution is uniquely attained at \(r=n\).

  \emph{(3b) Residues from the special factor.}
  It remains to rule out cancellation by a possible special contribution.
  Suppose \(i=\alpha_u\). If \(u=i\), this is the excluded double-pole case.
  If \(u<i\), then necessarily \(u=i-1\), and
  \[
    \deg_q\Phi_{i,i-1}=\deg_q\Phi_{i,i},
    \qquad
    \operatorname{lc}_q(\Phi_{i,i-1})
    =
    \frac{1}{n+1}\operatorname{lc}_q(\Phi_{i,i}),
  \]
  so the leading terms do not cancel. If \(u>i\), then
  \[
    \deg_q\Phi_{i,u}-\deg_q\Phi_{i,i}
    =
    \frac{(u-i)(2n-i-u-1)}2>0,
  \]
  so the special contribution has strictly larger \(q\)-degree than all
  Pochhammer contributions.

  Therefore the leading \(q\)-term of
  \[
    \sum_{r=0}^n\Phi_{i,r}(q)
    =
    \lim_{s\to i}(1-q^{i-s})\zeta_{\mfh_n(\lri)}(s)
  \]
  is nonzero. Thus this residue is nonzero as a rational function of \(q\),
  hence nonzero for generic \(q\), and \(s=i\) is a simple pole.
\end{proof}

Thus repeated poles occur only for the sparse pairs \((n,m)\) satisfying
\(m(m+1)=4n\), for instance \((3,3)\), \((5,4)\), \((14,7)\), and so on.

\subsection{Local functional equations --- proof of \Cref{cor:funeq}}\label{subsec:funeq}
By \Cref{thm:simplified-zeta},
\(\zeta_{\mfh_n(\lri)}(s)=Z_n(q,q^{-s})\), where
  \[
    Z_n(q,T)
    :=
    \sum_{r=0}^{n}
    \frac{1}{1-q^{\,2n+\frac{r(2n+1-r)}{2}}T^{n+1}}
    \cdot
    \frac{
      (-q)^r(1-q^{2n-2r+1})(q^2;q^2)_n
    }{
      (q;q)_{2n-r+1}(q;q)_r
      (q^{r}T;q^{2})_{n-r}
      (q^{2n-r}T;q)_{r}
    }.
  \]
  A direct substitution \((q,T)\mapsto(q^{-1},T^{-1})\), using
  \(
    (a^{-1};q^{-1})_m
    =
    (-a^{-1})^m q^{-\binom{m}{2}}(a;q)_m
  \)
  and its \(q^2\)-analogue, shows that each summand transforms with the same
  factor:
  \[
    Z_n(q^{-1},T^{-1})
    =
    -q^{\binom{2n+1}{2}}T^{-(2n+1)}Z_n(q,T).
  \]
  Setting \(T=q^{-s}\) yields \Cref{cor:funeq}, recovering
  \cite[Thm.~A]{Voll/10} in the relevant cases.

  \section{Global analytic properties and reduced zeta functions}

We leverage our description of local subalgebra zeta functions of the
higher Heisenberg algebras to deduce key properties of global
subalgebra and subgroup zeta functions in \Cref{subsec:global}. In
\Cref{subsec:red} we consider reduced zeta functions.

\subsection{Global abscissa and subgroup growth --- proof of~\Cref{cor:global}}
\label{subsec:global}

As recalled in \Cref{subsec:intro-zeta}, the global subgroup zeta function
\(\zeta_{H_n(\Z)}(s)\) is the Euler product of the local subgroup zeta functions
\(\zeta_{H_n(\Z),p}(s)\), which in turn agree with
\(\zeta_{\mfh_n(\Zp)}(s)\). With \(N_n(X,Y)\) as defined in~\eqref{def:N}, our
\Cref{thm:C} thus gives
\begin{equation}\label{eq:global-product-growth}
  \zeta_{H_n(\Z)}(s)
  = \left( \prod_p N_n(p,p^{-s})\right)
  \prod_{i=0}^{2n-1}\zeta(s-i)
  \prod_{i=0}^{n}
  \zeta\left((n+1)s-\frac{n(n+5)}2+\frac{i(i+1)}2\right).
\end{equation}

\begin{lemma}\label{lem:max-C-minus-2nD}
  We have
  \[
    \max_{1\ne g \in B_n}\bigl(C(g)-2nD(g)\bigr)
    =-\frac{3n^2-n+4}{2}.
  \]
  The maximum is attained uniquely at \(s_0\) for \(n=1\) and at \(s_1\) for
  \(n\ge2\).
\end{lemma}

\begin{proof}
  The case \(n=1\) is clear, so assume \(n\ge2\). By \eqref{eq:Cg},
  \[
    C(g)-2nD(g)
    =-\frac{n(3n-1)}2\desB(g)-\ell(g)
     -\left(n\negg(g)+\sum_{i\in\DesB(g)}\frac{i(i+1)}2\right).
  \]
  For \(g\ne1\), we have \(\desB(g)\ge1\), \(\ell(g)\ge1\), and the
  parenthesized term is at least \(1\). Hence
  \(C(g)-2nD(g)\le -n(3n-1)/2-2\). Equality forces
  \(\desB(g)=\ell(g)=1\), \(\negg(g)=0\), and \(\DesB(g)=\{1\}\), whence
  \(g=s_1\). Thus
  \[
    \max_{g\ne1}\bigl(C(g)-2nD(g)\bigr)
    =C(s_1)-2nD(s_1)
    =-\frac{n(3n-1)}2-2
    =-\frac{3n^2-n+4}{2}.\qedhere
  \]
\end{proof}

We now prove~\Cref{cor:global}. For \(n=1\), the abscissa assertion follows
from the formula for \(\zeta_{H_1(\Z)}(s)\) recalled in the introduction.
Assume henceforth that \(n\ge2\). Among the zeta factors in
\eqref{eq:global-product-growth}, only
\(\zeta(s-2n+1)\) has abscissa \(2n\); indeed, for \(i\in[n]_0\),
  \[
    \frac{\frac{n(n+5)}2-\frac{i(i+1)}2+1}{n+1}<2n.
  \]
  It remains to control the numerator product. By
  \Cref{lem:max-C-minus-2nD}, every nonconstant monomial \(X^aY^b\) of
  \(N_n(X,Y)\) satisfies \(a-2nb\le-(3n^2-n+4)/2<-1\). Hence, uniformly in
  \(p\) and for \(\Re(s)\ge2n\),
  \[
    N_n(p,p^{-s})=1+O\left(p^{-(3n^2-n+4)/2}\right).
  \]
  Thus \(\prod_pN_n(p,p^{-s})\) is absolutely convergent near
  \(s=2n\), and the only pole on \(\Re(s)=2n\) is the simple pole of
  \(\zeta(s-2n+1)\), with residue~\(R_n\); see~\eqref{def:R}. Applying
  \cite[Thm.~4.20]{duSG00} with \(a=2n\) and \(w=1\) gives
  \[
    s_N(H_n(\Z)) \sim \frac{R_n}{2n\Gamma(1)}N^{2n} =\frac{R_n}{2n}N^{2n}.
  \]
  This concludes the proof of \Cref{cor:global}.

We record in \Cref{tab:N} the Euler factors \(N_n(p,p^{-2n})\) appearing in
\(R_n\) for \(n\in\{1,2,3\}\). For \(n=1\), this factor equals \(1-p^{-3}\),
producing the denominator \(\zeta(3)\) in \eqref{equ:n=2}.

\begin{table}[h]
  \begin{tabular}{r|l}
    \(n\) & \(N_n(p,p^{-2n})\) \\ \hline \\[-1em]
    \(1\) & \(1-p^{-3}\)
    \\[0.1em]
    \(2\) & \(1+p^{-7}-p^{-8}-p^{-9}-p^{-10}-p^{-11}+p^{-12}+p^{-19}\)
    \\[0.1em]
    \(3\) & \scriptsize\(1+p^{-14}+p^{-15}-2p^{-18}-2p^{-19}
      -2p^{-20}-p^{-21}+p^{-24}+p^{-25}+p^{-26}-p^{-27}
      +p^{-31}-\cdots-p^{-58}\)
  \end{tabular}
  \caption{Euler factors \(N_n(p,p^{-2n})\) appearing in the residue factor~\(R_n\); cf.\ \eqref{def:R}}
  \label{tab:N}
\end{table}

\subsection{Reduced zeta functions}\label{subsec:red}

We use \Cref{thm:Bn-target} to compute the reduced subalgebra zeta function
\(Z^{\textup{red}}_{\mfh_n}(T)\in\Q(T)\), obtained from
\(\zeta_{\mfh_n(\lri)}(s)\) by formally setting \(q=1\) while keeping
\(T=q^{-s}\); see \cite{Evseev2009Reduced}.

Let \(A_0(X)=1\), and for \(d\ge1\) let
\(A_d(X)=\sum_{w\in S_d}X^{\des(w)+1}\). These are the Eulerian
polynomials \cite[Eq.~(1.36)]{stanley2011}.
We use the Eulerian identity
\begin{equation}\label{eq:stanley}
  \sum_{i\ge0}i^dX^i=\frac{A_d(X)}{(1-X)^{d+1}}
\end{equation}
from \cite[Prop.~1.4.4]{stanley2011}. We also use Brenti's
type-\(\mathsf B\) Eulerian polynomials
\[
  B_n(X,Y):=\sum_{\sigma\in B_n}Y^{\negg(\sigma)}X^{\desB(\sigma)}
\]
introduced in \cite[(10)]{brenti1994}, together with the generating-series
identity \cite[(12)]{brenti1994}
\begin{equation}\label{eq:brenti}
  \sum_{i\ge0}\bigl(1+(1+Y)i\bigr)^nX^i
  =
  \frac{B_n(X,Y)}{(1-X)^{n+1}}.
\end{equation}

\begin{prop}\label{prop:zeta-red-brenti}
  We have
  \[
    Z^{\textup{red}}_{\mfh_n}(T)
    =
    \sum_{d=0}^n
    \binom{n}{d}
    \frac{A_d(T^{n+1})}{(1-T)^{2n-d}(1-T^{n+1})^{d+1}}.
  \]
\end{prop}

\begin{proof}
  Setting \(q=1\) in \Cref{thm:Bn-target} gives
  \[
    Z^{\textup{red}}_{\mfh_n}(T)
    =
    \frac{
      \sum_{g\in B_n}
      (-1)^{\negg(g)}T^{(n+1)\desB(g)+\negg(g)}
    }{
      (1-T)^{2n}(1-T^{n+1})^{n+1}
    }
    =
    \frac{B_n(T^{n+1},-T)}
      {(1-T)^{2n}(1-T^{n+1})^{n+1}}.
  \]
  Applying \eqref{eq:brenti} with \(X=T^{n+1}\) and \(Y=-T\), we obtain
  \begin{align*}
    Z^{\textup{red}}_{\mfh_n}(T)
    &=
    \frac{1}{(1-T)^{2n}}
    \sum_{i\ge0}\bigl(1+(1-T)i\bigr)^nT^{(n+1)i} \\
    &=
    \frac{1}{(1-T)^{2n}}
    \sum_{d=0}^n\binom{n}{d}(1-T)^d
    \sum_{i\ge0}i^d(T^{n+1})^i.
  \end{align*}
  The proposition then follows from \eqref{eq:stanley}.
\end{proof}

\begin{cor}\label{cor:zeta-red-denominator}
  There exists a polynomial \(P_n(T)\in\Z[T]\) such that
  \[
    Z^{\textup{red}}_{\mfh_n}(T)
    =
    \frac{P_n(T)}{(1-T)^{n}(1-T^{n+1})^{n+1}}.
  \]
\end{cor}
\begin{proof}
  In \Cref{prop:zeta-red-brenti}, multiplying the \(d\)-th summand by
  \((1-T)^n(1-T^{n+1})^{n+1}\) leaves the polynomial factor
  \[
    \left(\frac{1-T^{n+1}}{1-T}\right)^{n-d}.
  \]
  Hence the claimed denominator suffices.
\end{proof}

\begin{remark}
  The reduced zeta function \(Z^{\textup{red}}_{\mfh_n}(T)\) is the
  integer-point transform of a \(2n+1\)-dimensional rational polyhedral cone:
  \begin{equation}\label{eq:cone}
    Z^{\textup{red}}_{\mfh_n}(T)
    =
    \sum_{\substack{(e_0,e_1,\dots,e_{2n})\in \N_0^{2n+1} \\ e_0 \leq e_1+e_2,\dots,e_{2n-1}+e_{2n}}}
    T^{\sum_{i=0}^{2n}e_i}.
  \end{equation}
  Indeed, the basis \((x_1,\dots,x_{2n},y)\) in \Cref{eq:pres} is nice and
  simple in the sense of~\cite{Evseev2009Reduced}, so that
  \cite[Prop.\ 4.1]{Evseev2009Reduced} applies.
\end{remark}

Note that the local functional equation \Cref{cor:funeq} implies the
self-reciprocity
\[
  \left.Z^{\textup{red}}_{\mfh_n}(T)\right|_{T\rarr T^{-1}}
  =
  -T^{2n+1}Z^{\textup{red}}_{\mfh_n}(T),
\]
which may also be deduced directly from~\Cref{eq:cone}; cf.\
\cite[Prop.~5.2]{Evseev2009Reduced}.

\subsubsection*{Behaviour at \(T=1\)}
By \Cref{cor:zeta-red-denominator}, \(Z^{\textup{red}}_{\mfh_n}(T)\)
has a pole of order at most \(2n+1\) at \(T=1\). We consider
\[
  c_n
  :=
  \left.Z^{\textup{red}}_{\mfh_n}(T)(1-T)^{2n+1}\right|_{T=1}.
\]

\begin{prop}\label{prop:cn-limit}
  We have
  \[
    c_n
    =
    \sum_{k=0}^n \binom{n}{k}\frac{k!}{(n+1)^{k+1}}
    =
    1-n\sum_{k=1}^{n}\binom{n-1}{k-1}\frac{k!}{(n+1)^{k+1}}
    \in(0,1).
  \]
\end{prop}

\begin{proof}
  By \Cref{prop:zeta-red-brenti},
  \[
    (1-T)^{2n+1}Z^{\textup{red}}_{\mfh_n}(T)
    =
    (1-T)
    \sum_{k=0}^n
    \binom{n}{k}(1-T)^k
    \frac{A_k(T^{n+1})}{(1-T^{n+1})^{k+1}}.
  \]
  As \(T\to1\), we have \(A_k(T^{n+1})\to A_k(1)=k!\) and
  \(\dfrac{1-T}{1-T^{n+1}}\to \dfrac{1}{n+1}\). Therefore,
  \[
    c_n
    =
    \sum_{k=0}^n \binom{n}{k}\frac{k!}{(n+1)^{k+1}}.
  \]

  For the second equality, observe that
  \begin{align*}
    \sum_{k=0}^n \binom{n}{k}\frac{k!}{(n+1)^{k+1}}
   & +n\sum_{k=1}^{n}\binom{n-1}{k-1}\frac{k!}{(n+1)^{k+1}} 
    = \sum_{k=0}^n\frac{n!(k+1)}{(n-k)!(n+1)^{k+1}}.
  \end{align*}
  Rewriting \(k+1=(n+1)-(n-k)\), the last sum telescopes to \(1\). Hence the
  second formula follows, and the two formulas imply \(0<c_n<1\).
\end{proof}

In particular, \(Z^{\textup{red}}_{\mfh_n}(T)\) has a pole of order exactly
\(2n+1\) at \(T=1\).
\begin{remark}
  The numerators of the rational numbers \(c_n\) form
  \href{https://oeis.org/A393141}{OEIS sequence~A393141} \cite{OEIS:A393141}.
\end{remark}

\begin{ex}
  For \(n\in\{1,2,3\}\), the reduced zeta functions are
\begin{alignat*}{2}
  Z^{\textup{red}}_{\mfh_1}(T)
  &=
  \frac{1+T+T^2}{(1-T)(1-T^2)^2},
  \qquad &
  c_1&=\frac34,\\
  Z^{\textup{red}}_{\mfh_2}(T)
  &=
  \frac{1+2T+3T^2+5T^3+3T^4+2T^5+T^6}{(1-T)^2(1-T^3)^3},
  \qquad&
  c_2&=\frac{17}{27},\\
  Z^{\textup{red}}_{\mfh_3}(T)
  &=
  \frac{1+3T+6T^2+10T^3+19T^4+21T^5+22T^6+\dots+T^{12}}
    {(1-T)^3(1-T^4)^4},\quad & c_3 &=\frac{71}{128}.
\end{alignat*}
\end{ex}

\begin{acknowledgements}
  This research was funded by the Deutsche Forschungsgemeinschaft (DFG,
  German Research Foundation) -- Project-ID 491392403 -- TRR 358.
\end{acknowledgements}

\printbibliography
\end{document}